\documentclass[11pt, reqno]{amsart}
\usepackage{amssymb}
\usepackage{amsmath}
\usepackage{amsthm}
\usepackage{braket}
\usepackage{pdfpages}
\usepackage{graphicx}
\usepackage{caption}
\usepackage{subcaption}
\usepackage{mathrsfs} 
\usepackage{tikz-cd} 
\usepackage{bm}
\usepackage{enumitem}
\usepackage{mathtools}
\usepackage{pgfplots}
\usepackage{import}
\usepackage{xifthen}
\usepackage{transparent}
\usepackage{mathtools}
\usepackage{tabularx}
\usepackage{hhline}
\usepackage{footnote}
\makesavenoteenv{tabular}
\usepackage{pgf}
\usepackage{tikz}
\usepackage[utf8]{inputenc}
\usetikzlibrary{arrows,automata}
\usetikzlibrary{positioning}
\usepackage{mathptmx}
\usetikzlibrary{calc}
\usepackage{scrextend}

\newcolumntype{L}[1]{>{\raggedright\arraybackslash}p{#1}}
\newcolumntype{C}[1]{>{\centering\arraybackslash}p{#1}}
\newcolumntype{R}[1]{>{\raggedleft\arraybackslash}p{#1}}

\usepackage[utf8]{inputenc}
\usepackage[english]{babel}

\usepackage{hyperref}

\tikzset{
    state/.style={
           rectangle,
           rounded corners,
           draw=black, very thick,
           minimum height=2em,
           inner sep=2pt,
           text centered,
           },
}

\usepackage{lmodern,babel,adjustbox,booktabs,multirow}

\numberwithin{equation}{section}
\theoremstyle{plain}
\newtheorem{theorem}{Theorem}[section]

\theoremstyle{theorem}
\newtheorem{prop}[theorem]{Proposition}
\newtheorem{lem}[theorem]{Lemma}
\newtheorem{cor}[theorem]{Corollary}

\newtheorem*{question*}{Question}

\theoremstyle{plain}
\newtheorem{ih}{IH}

\theoremstyle{definition}
\newtheorem{defn}[theorem]{Definition}

\newcommand{\R}{\mathbb{R}}
\newcommand{\C}{\mathbb{C}}

\newcommand{\Proj}{\mathbb{P}}
\newcommand{\Q}{\mathbb{Q}}
\newcommand{\Z}{\mathbb{Z}}
\newcommand{\Hyp}{\mathbb{H}}

\newcommand{\N}{\mathbb{N}}
\newcommand{\D}{\mathbb{D}}

\newcommand{\p}{p}
\newcommand{\vp}{p}

\newcommand{\T}{\mathcal{T}}

\DeclareMathOperator{\Rat}{Rat}

\DeclareMathOperator{\PH}{\mathcal{H}}

\DeclareMathOperator{\BP}{\mathcal{B}}
\DeclareMathOperator{\MP}{\mathcal{P}}
\DeclareMathOperator{\interior}{int}

\DeclareMathOperator{\Isom}{Isom}

\DeclareMathOperator{\Res}{Res}

\DeclareMathOperator{\Int}{Int}
\DeclareMathOperator{\chull}{Cvx\, Hull}

\DeclareMathOperator{\proj}{proj}

\numberwithin{figure}{section}

\title[On geometrically finite degenerations I]{On geometrically finite degenerations I: boundaries of main hyperbolic components}
\author{Yusheng Luo}
\address{Dept. of Mathematics \& University of Michigan, Ann Arbor, MI 48109 USA}
\email{yusheng.s.luo@gmail.com}
\date{\today}
\begin{document}

\begin{abstract}
In this paper, we develop a theory of {\em quasi post-critically finite} degenerations of Blaschke products.
This gives us tools to study the boundaries of hyperbolic components of rational maps in higher dimensional moduli spaces.
We use it to obtain a combinatorial classification of geometrically finite polynomials on the boundary of the {\em main hyperbolic component} $\PH_d$, i.e., the hyperbolic component in the space of monic and centered polynomials that contains $z^d$.
We also show that the closure $\overline{\PH_d}$ is not a topological manifold with boundary for $d\geq 4$ by constructing {\em self-bumps} on its boundary.
\end{abstract}

\maketitle

\setcounter{tocdepth}{1}
\tableofcontents

\section{Introduction}\label{sec:intro}
Let $f: \hat\C \longrightarrow \hat\C$ be a rational map of degree $d \geq 2$. 
It is said to be {\em hyperbolic} if all the critical points converge to attracting periodic cycles under iteration.
Hyperbolic maps form an open set in suitable moduli spaces,
and a connected component $U$ is called a {\em hyperbolic component}.
The topology of hyperbolic components has been studied extensively and is well understood in various settings \cite{BDK91, Milnor12, WY17}.
However, the boundaries of hyperbolic components and the interactions between hyperbolic components remain mysterious.
In a series of two papers, we develop a theory on `geometrically finite' degenerations to investigate these questions.
In this paper,
\begin{itemize}
\item We study {\em quasi post-critically finite} degenerations of Blaschke products $f_n \in \BP_d$, where $\BP_d$ is the space of normalized and marked Blaschke products of degree $d$ (see Equation \ref{eqn: BP}).
For such degenerations, we construct quasi-invariant trees modeled by a limiting simplicial tree map
$f: (\mathcal{T}, \p) \longrightarrow (\mathcal{T}, \p)$
with rescaling limits.
These quasi-invariant trees are the analogs of the Hubbard trees for post-critically finite polynomial.
We prove a realization theorem for quasi-invariant trees and thus classify quasi post-critically finite degenerations in $\BP_d$.
\item A rational map is said to be geometrically finite if the critical points in the Julia set have finite orbits.
Using the realization theorem, we obtain a classification of geometrically finite polynomials on the boundary of the main hyperbolic component of polynomials $\PH_d$ containing $z^d$.
\item The study of the quasi-invariant tree maps reveals many different accesses to a boundary point from $\PH_d$ and thus a `self-bump' occurs on $\partial \PH_d$, showing the closure $\overline{\PH_d}$ is not a topological manifold with boundary for $d\geq 4$ (see Figure \ref{fig:SB}).
\end{itemize}

In the sequel \cite{L21}, we study the convergence and divergence of quasi post-critically finite degenerations for rational maps.
We prove the boundedness of quasi post-critically finite degenerations for hyperbolic rational maps with Sierpinski carpet Julia set.
We also prove a `double limit theorem' for `quasi-Blaschke products' by giving a criterion for the convergence of simultaneous quasi post-critically finite degenerations on the two Fatou components of $z^d$.
Together with the realization result proved in this paper, the convergence results can be applied to show the existence of polynomial mating (cf. \cite{Douady83, TanLei92}).

Our theory runs parallel with the developments in Kleinian groups, and the results fit into the well-known Sullivan's dictionary between the two fields.
We summarize the comparisons in the following table.

\begin{center}
\begin{tabular}{|C{0.47\textwidth} | C{0.47\textwidth}|}
     \hline & \\
\textbf{Complex dynamics} & \textbf{Kleinian groups}  \\ & \\ \hhline{|==|}
Blaschke product      & Fuchsian group \\ \hline
Quasi-Blaschke products         &  Quasi-Fuchsian group \\ \hline
Main hyperbolic component         & Bers slice \\ \hline
Geometrically finite rational map         & Geometrically finite Kleinian group \\ \hline
Sierpinski carpet rational maps        &  Acylindrical Kleinian groups \\ \hline
Geometrically finite polynomials on $\partial \PH_d$ & Cusps on the Bers boundary\\ \hline
Self-bumps on $\partial \PH_d$           & Self-bumps on the Bers boundary \\ \hline
Double limit theorem for quasi-Blaschke product        
& Double limit theorem for quasi-Fuchsian group \\ \hline
Boundedness for Sierpinski carpet rational maps       
& Thurston's compactness theorem for Acylindrical 3-manifold
\\ \toprule
\end{tabular} 
\end{center}

We now turn to a detailed statement of results.

\subsection*{Main Hyperbolic component}
A polynomial $P(z) = a_dz^d+...+a_0$ is said to be {\em monic} if $a_d = 1$ and {\em centered} if $a_{d-1} = 0$.
Let $\MP_d$ be the space of all monic and centered polynomials.
A degree $d$ polynomial $P$ with connected Julia set has $d-1$ invariant external rays.
A {\em marking} of $P$ is a particular choice of an invariant external ray.
A monic and centered polynomial $P$ with connected Julia set has a unique choice of the B\"ottcher map normalized so that the derivative at infinity is $1$.
The angle $0$ external ray under this normalization thus gives a marking of $P$.
Therefore, $P$ is regarded as a {\em marked polynomial}.
If $P$ has a Jordan curve Julia set, then a marking is equivalent to a choice of a repelling fixed point on its Julia set $J(P)$.

Let $\PH_d \subseteq \MP_d$ be the hyperbolic component that contains $z^d$.
We call $\PH_d$ the {\em main hyperbolic component} of degree $d$.

In the quadratic polynomial case, the Landing Theorem of Douady and Hubbard can be used to give a complete understanding of geometrically finite polynomials on the boundary of a hyperbolic component and which hyperbolic components of the Mandelbrot set have intersecting closures \cite{Milnor00b}.
In particular, geometrically finite polynomials on $\partial \PH_2$ are in correspondence with rational rotation numbers $\Q/\Z$ (see \cite[Theorem 6.5]{Milnor00b}).

To describe the dynamics of a geometrically finite polynomial on $\hat P \in \partial \PH_d$ in higher degrees, we introduce the notion of {\em pointed Hubbard tree}.
A polynomial $P$ is said to be {\em post-critically finite} if the critical points have finite orbits.
Given any geometrically finite polynomial $\hat P \in \MP_d$ with connected Julia set, there exists a post-critically finite polynomial $P \in \MP_d$ with topologically conjugate dynamics on the Julia sets compatible with the markings \cite{Haissinsky00}.
The dynamics of $P$ is described combinatorially by its (angled) Hubbard tree $H$.

If $\hat P \in \partial \PH_d$, there exists a special non-repelling fixed point $\hat p = \lim p_n$ of $\hat P$, where $p_n$ is the attracting fixed point of $P_n \in \PH_d$ and $P_n \to \hat P$.
This gives a fixed point $p\in H$, and we call $(H,p)$ the pointed Hubbard tree for $\hat P$.

The pointed Hubbard tree $(H,p)$ is said to be {\em simplicial} if there exists a finite simplicial structure on $H$ for which $P$ is a simplicial map, i.e., $P$ sends an edge of $H$ to an edge of $H$.
We say $(H',\vp')$ is a {\em pointed simplicial tuning} of $(H,\vp)$ if $(H',\vp')$ is constructed from $(H,\vp)$ by `replacing' the center $\vp$ of local degree $\delta(\p)$ by a simplicial pointed Hubbard tree of degree $\delta(\p)$ and modifying the backward orbits of $\p$ accordingly (see \S \ref{sec:pht} for precise definition).
We say a degree $d$ pointed Hubbard tree $(H,\vp)$ is {\em iterated-simplicial} if it can be inductively constructed from the trivial degree $d$ pointed Hubbard tree $(H=\{\vp\}, \vp)$ by a sequence of pointed simplicial tunings. 
We first show
\begin{theorem}\label{thm: eq}
If $\hat P\in \partial \PH_d$ is geometrically finite, then the associated pointed Hubbard tree $(H,\vp)$ is iterated-simplicial.
Conversely, if a degree $d$ pointed Hubbard tree $(H,\vp)$ is iterated-simplicial, then there exists a corresponding geometrically finite polynomial $\hat P$ on the boundary $\partial \PH_d$.
\end{theorem}

We remark that the classification naturally gives a level structure for geometrically finite polynomials on $\partial \PH_d$.
This level structure is a manifestation of the Schwarz lemma and the incompatible escaping rates for the critical points for the dynamics in the corresponding Fatou component (see Figure \ref{fig:D3} and the discussions below Theorem \ref{thm:shtr}).

Theorem \ref{thm: eq} gives concrete combinatorial models for geometrically finite polynomials on $\partial \PH_d$ and classifies all models that arise.
Let us denote by $U_{\hat P}$ the space of all geometrically finite polynomials corresponding to the same pointed Hubbard tree as $\hat P \in \partial \PH_d$.
We do not know a description of $U_{\hat P} \cap \partial \PH_d$, and it is expected this space can be quite complicated in general \cite{McM94b} (see \S \ref{sec:BPH} for some partial answers).
To discuss some of the subtleties of describing this space, it is well-known that the closure of a hyperbolic component $\mathcal{H}$ may not be {\em quasiconformally closed}: if $P \in \overline{\mathcal{H}}$, a quasiconformal deformation of $P$ may not be in $\overline{\mathcal{H}}$ (see \cite{Tan06, IM20}).

A hyperbolic component $\mathcal{H} \neq \PH_d$ is said to be a {\em satellite component} of $\PH_d$ if there exists a parabolic polynomial $\hat P\in \partial{\mathcal{H}} \cap \partial{\PH_d}$ that has conjugate dynamics on the Julia sets with any $P \in \mathcal{H}$.
As any parabolic polynomial can be perturbed to a hyperbolic polynomial with conjugate dynamics on the Julia sets \cite{Haissinsky00}, we immediately have the following corollary:
\begin{cor}\label{thm:eq2}
Let $\mathcal{H} \neq \PH_d$ be a hyperbolic component with connected Julia set. 
Let $H$ be the Hubbard tree of the post-critically finite center $P \in \mathcal{H}$.
Then $\mathcal{H}$ is a satellite component of $\PH_d$ if and only if there exists a fixed point $p\in H$ so that $(H,p)$ is iterated-simplicial.
\end{cor}

The direction that the pointed Hubbard tree for $\hat P\in \partial \PH_d$ is iterated-simplicial in Theorem \ref{thm: eq} follows from an analysis of cut points in the Julia set (see \S \ref{sec:pht}).
The other implication is proved by studying the degenerations of Blaschke products sketched in the following.

\subsection*{Blaschke products}
For $d\geq 2$, we let $\BP_d$ denote the space of {\em normalized and marked Blaschke products} $f: \D \longrightarrow \D$ of the form 
\begin{align}\label{eqn: BP}
f(z) = z\prod_{i=1}^{d-1} \frac{z-a_i}{1-\overline{a_i}z}, \text{ where }|a_i| < 1.
\end{align}
Note that $f(0) = 0$, and any proper holomorphic map from $\D$ to $\D$ of degree $d$ with a fixed point in $\D$ is holomorphically conjugate to a map in $\BP_d$.

Any map $f\in \BP_d$ can be extended to a rational map $f: \hat\C \longrightarrow \hat\C$.
Viewed as a rational map, the Julia set $J(f)$ is the circle $\mathbb{S}^1$, and there is a unique homeomorphism $\eta_f: \mathbb{S}^1 \cong\R/\Z \longrightarrow \mathbb{S}^1$, called the {\em marking}, that varies continuously with $f$, conjugates $z^d$ to $f$, and such that $\eta_f(z)$ is the identity map if $f(z) = z^d$.

The polynomials $P\in \PH_d$ are in correspondence with $f\in \BP_d$ by gluing $f\in \BP_d$ with $z^d$ using their markings on $\mathbb{S}^1$ \cite[\S 5]{McM88}.
This polynomial is denoted by $P = f \sqcup z^d$.
Thus, the study of $\partial \PH_d$ is naturally related to the study of degenerations in $\BP_d$.
We study `geometrically finite' degenerations in $\BP_d$ which give us uniform control on the rescaling dynamics at the critical orbits.

A sequence $f_n \in \BP_d$ is said to be {\em ($K$-)quasi post-critically finite} if we can label the critical points by $c_{i,n} \in \D$, $i=1,..., d-1$, such that for each $i$, there exist quasi pre-period $l_i$ and quasi period $q_i$ with
$$
d_{\D}(f_n^{l_i}(c_{i,n}), f_n^{l_i+q_i}(c_{i,n})) \leq K.
$$
The uniform bounds allow us to construct a sequence of quasi-invariant trees $\mathcal{T}_n \subseteq \D$ for the hyperbolic metric $d_{\D}$ on $\D$, capturing all the interesting dynamics (see Theorem \ref{thm:qit}, which is interesting by its own right). 
The dynamics on $\mathcal{T}_n$ is described by a simplicial tree map $f: (\mathcal{T}, \p) \longrightarrow (\mathcal{T}, \p)$, with {\em rescaling limits} $F_{v}: \D_v \longrightarrow \D_{f(v)}$ between vertices (cf. \cite{Kiwi15}).
This simplicial tree map plays a similar role as the Hubbard tree for a post-critically finite polynomial.

Let $P_n = f_n \sqcup z^d\in \PH_d$ be the corresponding quasi post-critically finite sequence of polynomials.
We show the limit of $P = \lim_{n\to\infty} P_n$ is geometrically finite, and the sequence of the quasi-invariant trees converges to a pointed Hubbard tree with `decorations' (see Figure \ref{fig:R} and Figure \ref{fig:HD}).
This allows us to describe the pointed Hubbard tree of the limit using quasi-invariant trees, building a bridge between geometrically finite polynomials on $\partial \PH_d$ and quasi post-critically finite degenerations in $\BP_d$.
\begin{figure}[ht]
  \centering
  \resizebox{1\linewidth}{!}{
    \def\svgwidth{\columnwidth}
\begingroup%
  \makeatletter%
  \providecommand\color[2][]{%
    \errmessage{(Inkscape) Color is used for the text in Inkscape, but the package 'color.sty' is not loaded}%
    \renewcommand\color[2][]{}%
  }%
  \providecommand\transparent[1]{%
    \errmessage{(Inkscape) Transparency is used (non-zero) for the text in Inkscape, but the package 'transparent.sty' is not loaded}%
    \renewcommand\transparent[1]{}%
  }%
  \providecommand\rotatebox[2]{#2}%
  \newcommand*\fsize{\dimexpr\f@size pt\relax}%
  \newcommand*\lineheight[1]{\fontsize{\fsize}{#1\fsize}\selectfont}%
  \ifx\svgwidth\undefined%
    \setlength{\unitlength}{388.5bp}%
    \ifx\svgscale\undefined%
      \relax%
    \else%
      \setlength{\unitlength}{\unitlength * \real{\svgscale}}%
    \fi%
  \else%
    \setlength{\unitlength}{\svgwidth}%
  \fi%
  \global\let\svgwidth\undefined%
  \global\let\svgscale\undefined%
  \makeatother%
  \begin{picture}(1,0.95559846)%
    \lineheight{1}%
    \setlength\tabcolsep{0pt}%
    \put(0,0){\includegraphics[width=\unitlength,page=1]{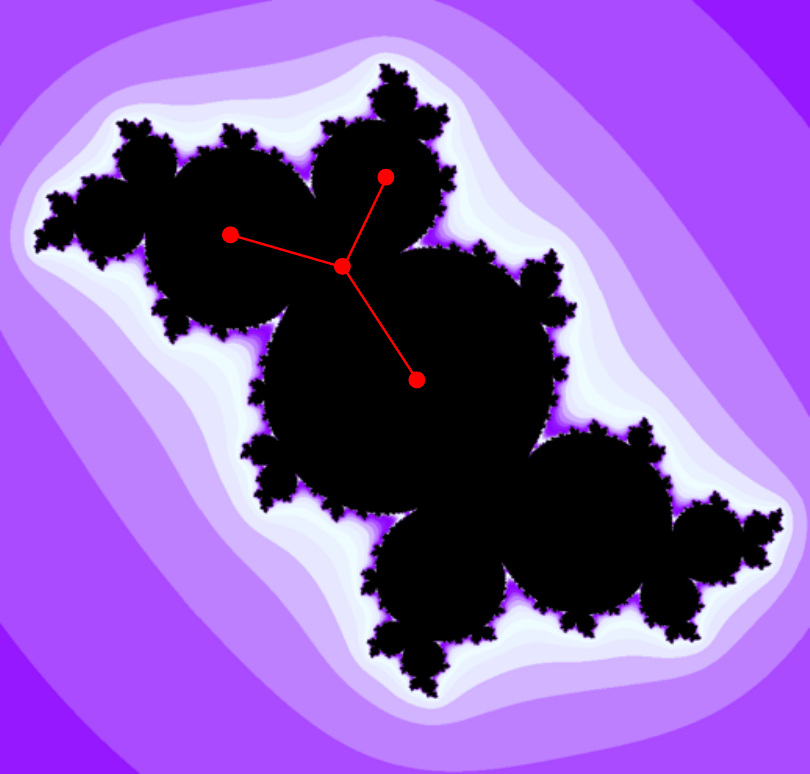}}%
    \put(0.53489906,0.48224011){\color[rgb]{1,0,0}\makebox(0,0)[lt]{\lineheight{1.25}\smash{\begin{tabular}[t]{l}{\LARGE$c_n$}\end{tabular}}}}%
    \put(0.44939229,0.62855636){\color[rgb]{1,0,0}\makebox(0,0)[lt]{\lineheight{1.25}\smash{\begin{tabular}[t]{l}{\LARGE$p_n$}\end{tabular}}}}%
    \put(0,0){\includegraphics[width=\unitlength,page=2]{R.pdf}}%
    \put(0.71662377,0.06613473){\color[rgb]{0,0,0}\makebox(0,0)[lt]{\lineheight{1.25}\smash{\begin{tabular}[t]{l}{\LARGE$\deg(c_n)=2$}\end{tabular}}}}%
  \end{picture}%
\endgroup%

    \def\svgwidth{\columnwidth}
\begingroup%
  \makeatletter%
  \providecommand\color[2][]{%
    \errmessage{(Inkscape) Color is used for the text in Inkscape, but the package 'color.sty' is not loaded}%
    \renewcommand\color[2][]{}%
  }%
  \providecommand\transparent[1]{%
    \errmessage{(Inkscape) Transparency is used (non-zero) for the text in Inkscape, but the package 'transparent.sty' is not loaded}%
    \renewcommand\transparent[1]{}%
  }%
  \providecommand\rotatebox[2]{#2}%
  \newcommand*\fsize{\dimexpr\f@size pt\relax}%
  \newcommand*\lineheight[1]{\fontsize{\fsize}{#1\fsize}\selectfont}%
  \ifx\svgwidth\undefined%
    \setlength{\unitlength}{388.5bp}%
    \ifx\svgscale\undefined%
      \relax%
    \else%
      \setlength{\unitlength}{\unitlength * \real{\svgscale}}%
    \fi%
  \else%
    \setlength{\unitlength}{\svgwidth}%
  \fi%
  \global\let\svgwidth\undefined%
  \global\let\svgscale\undefined%
  \makeatother%
  \begin{picture}(1,0.95559846)%
    \lineheight{1}%
    \setlength\tabcolsep{0pt}%
    \put(0,0){\includegraphics[width=\unitlength,page=1]{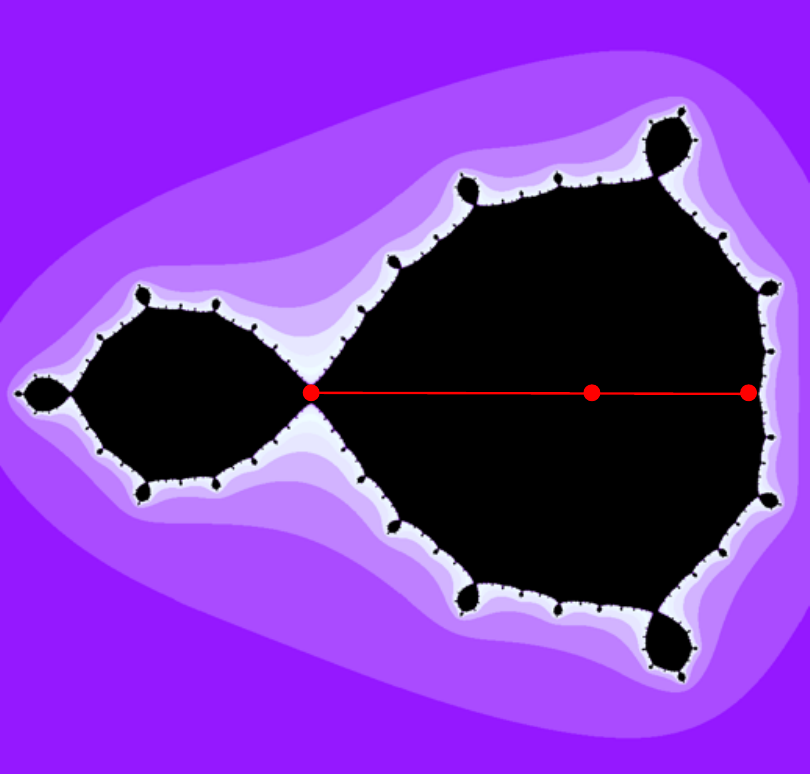}}%
    \put(0.72661393,0.43626924){\color[rgb]{1,0,0}\makebox(0,0)[lt]{\lineheight{1.25}\smash{\begin{tabular}[t]{l}{\LARGE$p_n$}\end{tabular}}}}%
    \put(0.40145289,0.44004794){\color[rgb]{1,0,0}\makebox(0,0)[lt]{\lineheight{1.25}\smash{\begin{tabular}[t]{l}{\LARGE$c_n$}\end{tabular}}}}%
    \put(0,0){\includegraphics[width=\unitlength,page=2]{G.pdf}}%
    \put(0.53729687,0.06663897){\color[rgb]{0,0,0}\makebox(0,0)[lt]{\lineheight{1.25}\smash{\begin{tabular}[t]{l}{\LARGE$\deg(p_n) = \deg(c_n)=2$}\end{tabular}}}}%
  \end{picture}%
\endgroup%

    \def\svgwidth{\columnwidth}
\begingroup%
  \makeatletter%
  \providecommand\color[2][]{%
    \errmessage{(Inkscape) Color is used for the text in Inkscape, but the package 'color.sty' is not loaded}%
    \renewcommand\color[2][]{}%
  }%
  \providecommand\transparent[1]{%
    \errmessage{(Inkscape) Transparency is used (non-zero) for the text in Inkscape, but the package 'transparent.sty' is not loaded}%
    \renewcommand\transparent[1]{}%
  }%
  \providecommand\rotatebox[2]{#2}%
  \newcommand*\fsize{\dimexpr\f@size pt\relax}%
  \newcommand*\lineheight[1]{\fontsize{\fsize}{#1\fsize}\selectfont}%
  \ifx\svgwidth\undefined%
    \setlength{\unitlength}{388.5bp}%
    \ifx\svgscale\undefined%
      \relax%
    \else%
      \setlength{\unitlength}{\unitlength * \real{\svgscale}}%
    \fi%
  \else%
    \setlength{\unitlength}{\svgwidth}%
  \fi%
  \global\let\svgwidth\undefined%
  \global\let\svgscale\undefined%
  \makeatother%
  \begin{picture}(1,0.95559846)%
    \lineheight{1}%
    \setlength\tabcolsep{0pt}%
    \put(0,0){\includegraphics[width=\unitlength,page=1]{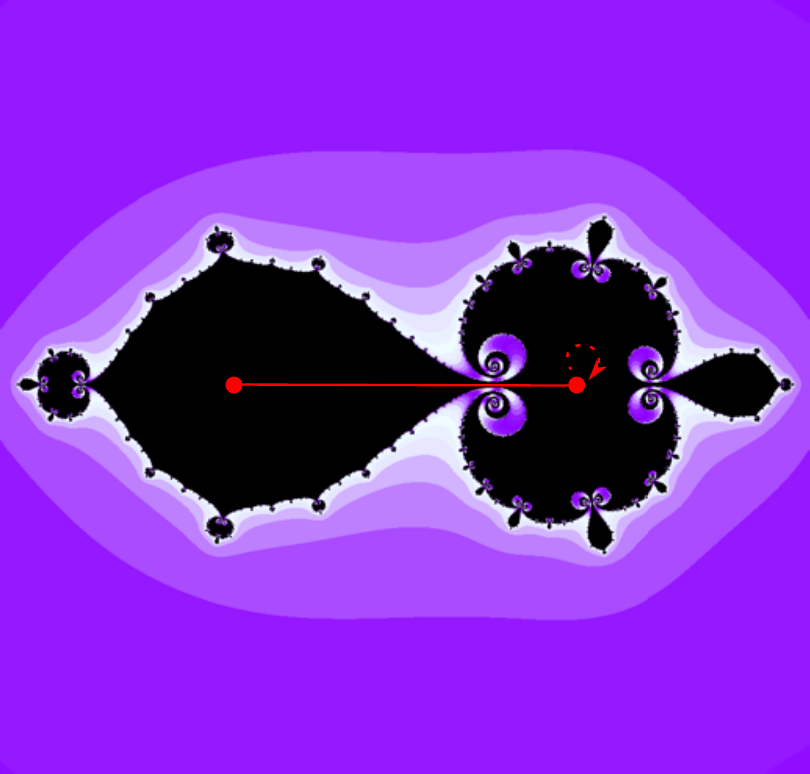}}%
    \put(0.288369,0.44570684){\color[rgb]{1,0,0}\makebox(0,0)[lt]{\lineheight{1.25}\smash{\begin{tabular}[t]{l}{\LARGE$p_n$}\end{tabular}}}}%
    \put(0.5401457,0.07876517){\color[rgb]{0,0,0}\makebox(0,0)[lt]{\lineheight{1.25}\smash{\begin{tabular}[t]{l}{\LARGE$\deg(p_n) = \deg(c_n)=2$}\end{tabular}}}}%
    \put(0.70269265,0.44289106){\color[rgb]{1,0,0}\makebox(0,0)[lt]{\lineheight{1.25}\smash{\begin{tabular}[t]{l}{\LARGE$c_n$}\end{tabular}}}}%
  \end{picture}%
\endgroup%

  }
  \caption{Illustrations of quasi-invariant trees for quasi post-critically finite degenerations of $(f_n)_n \in \BP_d$ in the bounded Fatou component of the corresponding polynomials $P_n = f_n \sqcup z^d$. }
  \label{fig:R}
\end{figure}

In \S \ref{sec:atm}, we define an abstract {\em angled tree map} 
$$
(f: (\mathcal{T}, p) \longrightarrow (\mathcal{T}, p), \delta, \alpha = \{\alpha_v\})
$$ 
with local degree function $\delta$ and angle functions $\alpha_v$ satisfying certain compatibility conditions. 
This data combinatorially classifies the simplicial tree map with the rescaling limits (cf. angled Hubbard tree in \cite{Poirier93}). 
Every simplicial pointed Hubbard tree $P:(H, \p)\longrightarrow (H, \p)$ gives an angled tree map.
In \S \ref{sec:raatm}, we define {\em admissible} angled tree maps and prove a realization theorem. 
In particular, we show
\begin{theorem}\label{thm:shtr}
Every simplicial pointed Hubbard tree
is realizable by $f_n \in \BP_d$ after an admissible splitting on its Julia branch points.
\end{theorem}

We remark that the precise statement for the realization theorem can be found in Theorem \ref{thm:ar} and Proposition \ref{prop:rt}.
The definition of admissible splitting is given in \S \ref{sec:BPH} (see also Figure \ref{fig:HD}).
The non-uniqueness of splittings is the source for self-bumps on $\partial \PH_d$ (see Figure \ref{fig:SB} and \S \ref{sec:sb}).

Theorem \ref{thm:shtr} allows us to construct geometrically finite polynomials $\hat P_1 \in \partial \PH_d$ for any simplicial pointed Hubbard trees $(H_1, \p_1)$.

The limit map $\hat P_1$ has at most one attracting Fatou component which is necessarily fixed. 
If it has one, we show we can further degenerate its dynamics on this attracting Fatou component (and its backward orbits under $\hat P_1$) by a quasi post-critically finite sequence, while staying on the boundary $\partial \PH_d$.
The limit of this sequence gives a geometrically finite polynomial $\hat P_2$.
The pointed Hubbard tree $(H_2, \p_2)$ of $\hat P_2$ is a pointed simplicial tuning of $(H_1, \p_1)$.
By induction, we thus construct geometrically finite polynomials on $\partial \PH_d$ for any iterated-simplicial Hubbard tree and conclude the proof of Theorem \ref{thm: eq}.

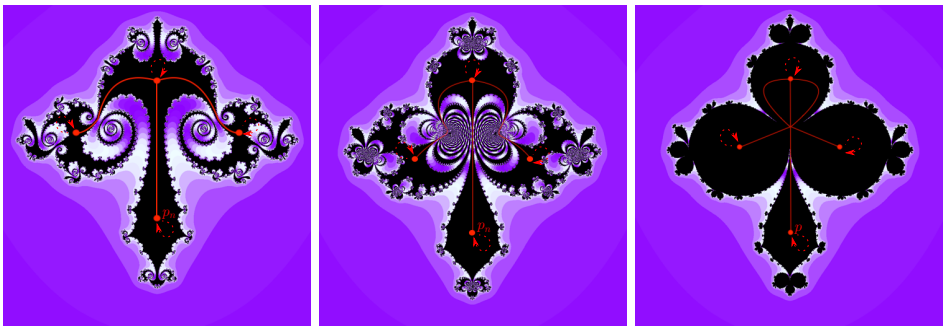
\begin{figure}[ht]
  \centering
  \resizebox{1\linewidth}{!}{
    \def\svgwidth{\columnwidth}
\begingroup%
  \makeatletter%
  \providecommand\color[2][]{%
    \errmessage{(Inkscape) Color is used for the text in Inkscape, but the package 'color.sty' is not loaded}%
    \renewcommand\color[2][]{}%
  }%
  \providecommand\transparent[1]{%
    \errmessage{(Inkscape) Transparency is used (non-zero) for the text in Inkscape, but the package 'transparent.sty' is not loaded}%
    \renewcommand\transparent[1]{}%
  }%
  \providecommand\rotatebox[2]{#2}%
  \newcommand*\fsize{\dimexpr\f@size pt\relax}%
  \newcommand*\lineheight[1]{\fontsize{\fsize}{#1\fsize}\selectfont}%
  \ifx\svgwidth\undefined%
    \setlength{\unitlength}{371.25bp}%
    \ifx\svgscale\undefined%
      \relax%
    \else%
      \setlength{\unitlength}{\unitlength * \real{\svgscale}}%
    \fi%
  \else%
    \setlength{\unitlength}{\svgwidth}%
  \fi%
  \global\let\svgwidth\undefined%
  \global\let\svgscale\undefined%
  \makeatother%
  \begin{picture}(1,1.04646465)%
    \lineheight{1}%
    \setlength\tabcolsep{0pt}%
    \put(0,0){\includegraphics[width=\unitlength,page=1]{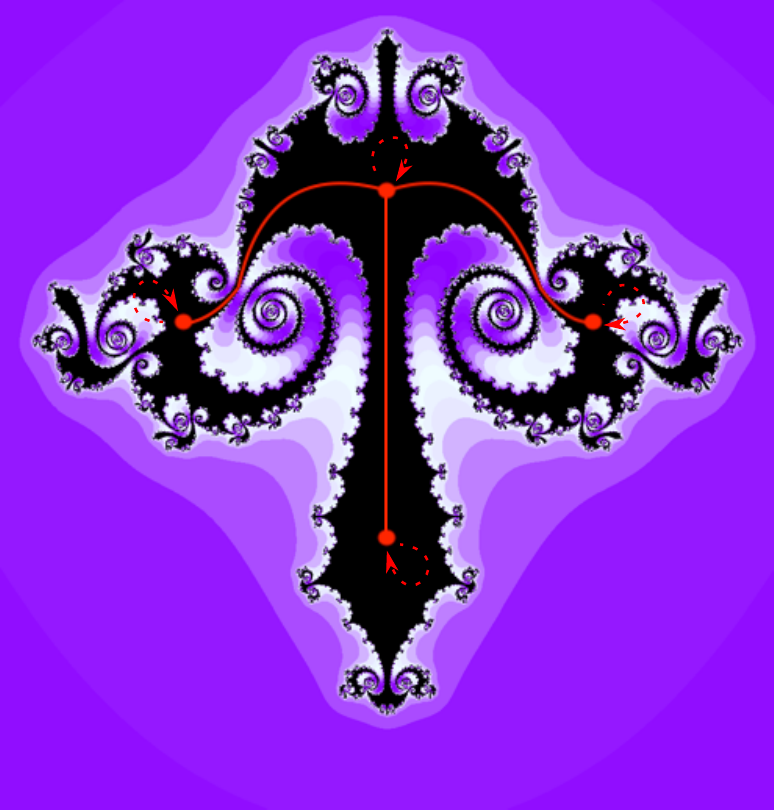}}%
    \put(0.51837745,0.35915691){\color[rgb]{1,0,0}\makebox(0,0)[lt]{\lineheight{1.25}\smash{\begin{tabular}[t]{l}{\LARGE$p_n$}\end{tabular}}}}%
  \end{picture}%
\endgroup%

    \def\svgwidth{\columnwidth}
\begingroup%
  \makeatletter%
  \providecommand\color[2][]{%
    \errmessage{(Inkscape) Color is used for the text in Inkscape, but the package 'color.sty' is not loaded}%
    \renewcommand\color[2][]{}%
  }%
  \providecommand\transparent[1]{%
    \errmessage{(Inkscape) Transparency is used (non-zero) for the text in Inkscape, but the package 'transparent.sty' is not loaded}%
    \renewcommand\transparent[1]{}%
  }%
  \providecommand\rotatebox[2]{#2}%
  \newcommand*\fsize{\dimexpr\f@size pt\relax}%
  \newcommand*\lineheight[1]{\fontsize{\fsize}{#1\fsize}\selectfont}%
  \ifx\svgwidth\undefined%
    \setlength{\unitlength}{371.25bp}%
    \ifx\svgscale\undefined%
      \relax%
    \else%
      \setlength{\unitlength}{\unitlength * \real{\svgscale}}%
    \fi%
  \else%
    \setlength{\unitlength}{\svgwidth}%
  \fi%
  \global\let\svgwidth\undefined%
  \global\let\svgscale\undefined%
  \makeatother%
  \begin{picture}(1,1.04646465)%
    \lineheight{1}%
    \setlength\tabcolsep{0pt}%
    \put(0,0){\includegraphics[width=\unitlength,page=1]{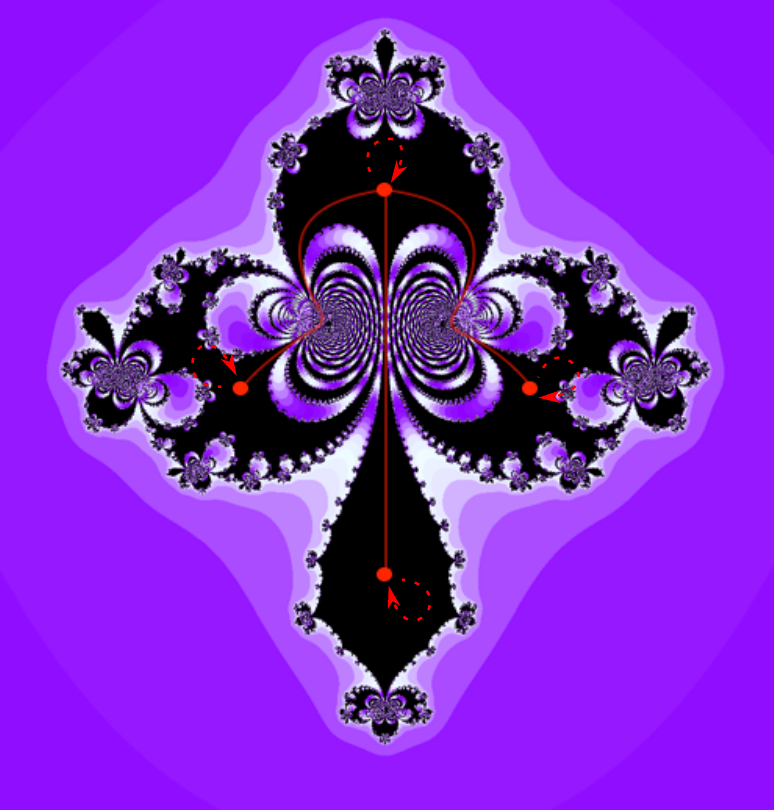}}%
    \put(0.5148569,0.31509154){\color[rgb]{1,0,0}\makebox(0,0)[lt]{\lineheight{1.25}\smash{\begin{tabular}[t]{l}{\LARGE$p_n$}\end{tabular}}}}%
  \end{picture}%
\endgroup%

    \def\svgwidth{\columnwidth}
\begingroup%
  \makeatletter%
  \providecommand\color[2][]{%
    \errmessage{(Inkscape) Color is used for the text in Inkscape, but the package 'color.sty' is not loaded}%
    \renewcommand\color[2][]{}%
  }%
  \providecommand\transparent[1]{%
    \errmessage{(Inkscape) Transparency is used (non-zero) for the text in Inkscape, but the package 'transparent.sty' is not loaded}%
    \renewcommand\transparent[1]{}%
  }%
  \providecommand\rotatebox[2]{#2}%
  \newcommand*\fsize{\dimexpr\f@size pt\relax}%
  \newcommand*\lineheight[1]{\fontsize{\fsize}{#1\fsize}\selectfont}%
  \ifx\svgwidth\undefined%
    \setlength{\unitlength}{371.25bp}%
    \ifx\svgscale\undefined%
      \relax%
    \else%
      \setlength{\unitlength}{\unitlength * \real{\svgscale}}%
    \fi%
  \else%
    \setlength{\unitlength}{\svgwidth}%
  \fi%
  \global\let\svgwidth\undefined%
  \global\let\svgscale\undefined%
  \makeatother%
  \begin{picture}(1,1.04646465)%
    \lineheight{1}%
    \setlength\tabcolsep{0pt}%
    \put(0,0){\includegraphics[width=\unitlength,page=1]{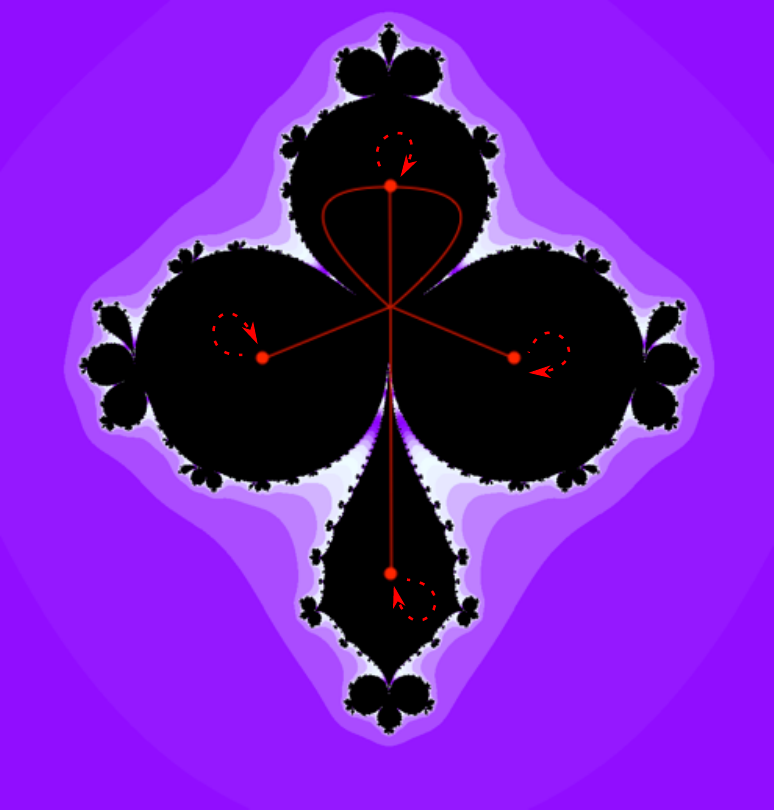}}%
    \put(0.51852734,0.31068698){\color[rgb]{1,0,0}\makebox(0,0)[lt]{\lineheight{1.25}\smash{\begin{tabular}[t]{l}{\LARGE$p$}\end{tabular}}}}%
  \end{picture}%
\endgroup%

  }
  \caption{Another example of quasi-invariant trees, where all vertices are quasi-fixed and have local degree $2$. The trees converge to a `decorated pointed Hubbard tree' in the limit. The corresponding pointed Hubbard tree is a `cross', and `splits' at the Julia branch point to get approximating quasi-invariant trees (see \S \ref{sec:BPH} for the definition of `splitting').}
  \label{fig:HD}
\end{figure}

\subsection*{Self-bumps on $\partial \PH_d$}
It is expected that the boundary of $\PH_d$ is quite complicated for $d\geq 3$ \cite{Milnor14}.
We construct polynomials $\hat P \in \partial \PH_d$ with different accesses from the main Hyperbolic component $\PH_d$ (see Figure \ref{fig:SB}).
Thus, {\em self bumps} occur on $\partial \PH_d$.
More precisely, we prove
\begin{theorem}\label{thm:sb}
For any degree $d\geq 4$, there exists a geometrically finite polynomial $\hat P \in \partial \PH_d$ such that for any sufficiently small neighborhood $U$ of $\hat P$, the intersection $U\cap \PH_d$ is disconnected.
\end{theorem}

\begin{figure}
   \begin{subfigure}[b]{0.45\textwidth}
     \centering
     \includegraphics[width=\textwidth]{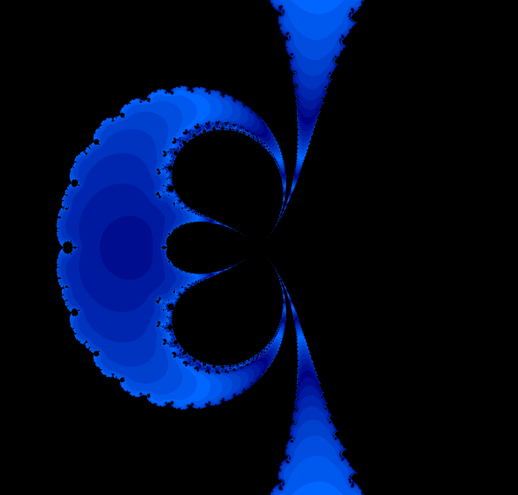}
     \caption{A self bump on $\partial \PH_4$ presented in a 1-D parameter slice.}
     \label{fig:SBP}
   \end{subfigure}
   \begin{subfigure}[b]{0.45\textwidth}
     \centering
     \includegraphics[width=\textwidth]{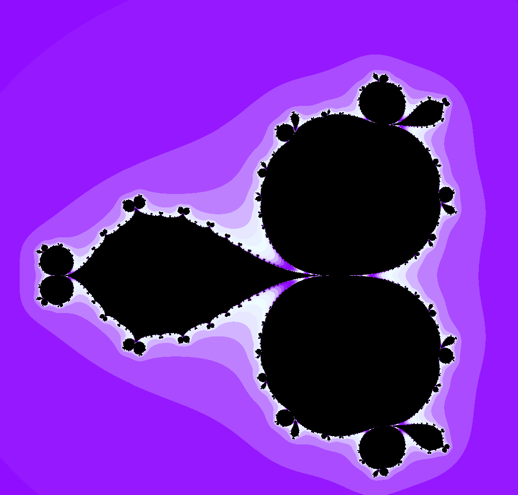}
     \caption{A self bump $\hat P\in \partial \PH_4$ with a double parabolic fixed point.}
     \label{fig:P}
   \end{subfigure}
   \begin{subfigure}[b]{0.45\textwidth}
     \centering
     \includegraphics[width=\textwidth]{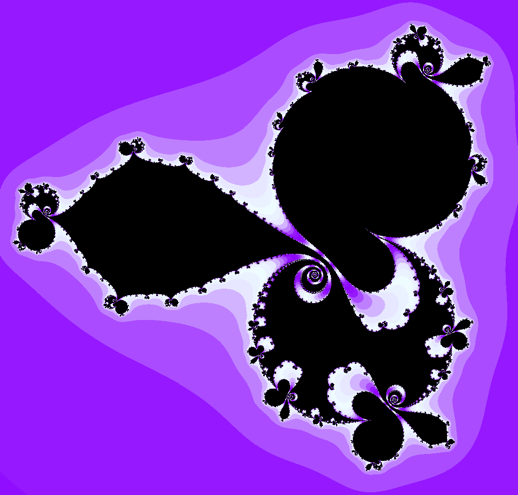}
     \caption{The access of the self-bump from `top'.}
     \label{fig:AccessTop}
   \end{subfigure}
     \begin{subfigure}[b]{0.45\textwidth}
     \centering
     \includegraphics[width=\textwidth]{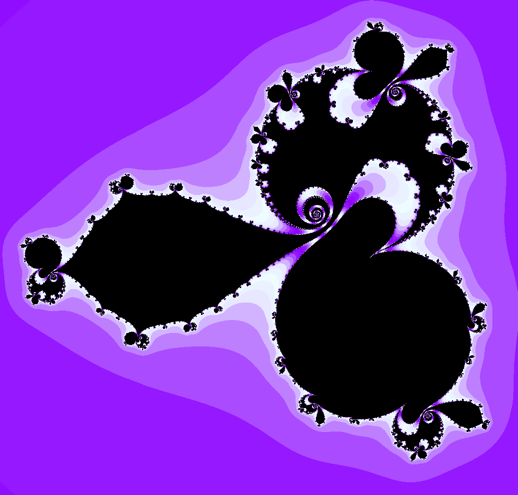}
     \caption{The access of the self-bump from `bottom'.}
     \label{fig:AccessBottom}
   \end{subfigure}
    \caption{A self-bump on $\partial \PH_4$ with two different accesses}
    \label{fig:SB}
\end{figure}

As an immediate corollary (cf. \cite[Theorem A.1]{McM98}), we have
\begin{cor}
The closure $\overline{\PH_d}$ is not a topological manifold with boundary for $d\geq 4$.
\end{cor}
\begin{proof}
By the characterizations of $J$-stable rational maps (see \cite[Theorem 4.2]{McM94}), 
$\interior \overline{\PH_d} = \PH_d$.
If $\overline{\PH_d}$ were a topological manifold with boundary, then there would be a small neighborhood $U$ of $\hat P$ meeting the manifold's interior in a connected set, contrary to Theorem \ref{thm:sb}.
Thus $\overline{\PH_d}$ is not a topological manifold with boundary.
\end{proof}

\subsection*{Comparisons with Kleinian groups}
We now discuss some similarities and differences between the theory in rational maps and Kleinian groups.
\begin{itemize}
\item {\em Dimension one vs higher dimension:} The Bers slice for a once punctured torus group has complex dimension one. It is known \cite{Minsky99} that the boundary is a Jordan curve, and hence locally connected.
When the deformation space has complex dimension $\geq 2$, it is known that the boundary of deformation spaces can be non-locally connected \cite{Bromberg11}.
For rational maps, $\partial \PH_2$ is a Jordan curve.
Many other hyperbolic components in one-dimensional slice are also proved to be Jordan disks \cite{Roesch07, Wang17}.
For $d\geq3$, it is conjectured that $\partial \PH_d$ is not locally connected \cite{Milnor14}.

\item {\em Geometric convergence:} A geometrically finite group $G$ on the Bers boundary can be constructed by pinching a system of disjoint simple closed curves.
The pinching deformations $G_n$ converge both algebraically and geometrically to $G$. 
In particular, the limit sets converge in Hausdorff topology to the limit set of $G$.
This is in contrast with rational maps. 
For most (in some suitable sense) iterated-simplicial pointed Hubbard trees, the convergence to any of the corresponding polynomial $\hat P \in \partial \PH_d$ is \textbf{never} {\em geometric}: the Julia sets do not converge in Hausdorff topology to $J(\hat P)$ for any $P_n \to \hat P$ (see Figure \ref{fig:HD}). See \cite{McM99, McM00, CT18} for related discussions.

\item {\em Self-bump:} Self-bumps on the Bers boundary were first constructed by McMullen using projective structures \cite{McM98}.
This was generalized to other deformation spaces in \cite{HB01}.
The self-bumps on $\partial \PH_d$ are direct analogues in the setting of rational maps.
We use the sign of the imaginary part of the multipliers to distinguish difference accesses (see \S \ref{sec:sb}), which are the analogues of the traces of the holonomy representations for Kleinian groups.
Our method does not provide any self-bumps in degree $3$.
\end{itemize}

\subsection*{Notes and references}
A large portion of the boundary $\partial \PH_3$ has been studied in \cite{PT09} and a combinatorial model of $\partial \PH_3$ is studied in \cite{BOPT14}.
The problem on how hyperbolic components are positioned for quadratic rational maps is studied in details in \cite{Rees90, Rees92}.
A related problem of self-bumps on $\partial \PH_3$ is studied in \cite{BOPT16, BOPT18}.

Degenerations of Fatou components have been studied extensively in terms of elementary moves like pinching and spinning \cite{Makienko00, TanLei02, HT04, PT04, CT18}.
The quasi post-critically finite degenerations in $\BP_d$ generalize these operations, and thus provide a unified framework.

Other comparisons of $\BP_d$ with Teichm\"uller theory can be found in \cite{McM08, McM09b, McM09, McM10}.
The quasi-invariant trees are closely related to the ribbon $\R$-trees as in \cite{McM09} and the analogous constructions of isometric group actions on $\R$-trees for Kleinian groups \cite{MorganShalen84, Bestvina88, Paulin88} (see discussions in \S \ref{sec: gfb}).

Related bumping problems for deformation spaces of Kleinian groups are studied in \cite{ACMcC00}.
Rescaling limits of Blaschke products have also been studied in \cite{Ivrii16}, and other application of rescaling limits in complex dynamics can be found in \cite{Epstein00, DeM05, Kiwi15, Arfeux17, L19a, L19}.

\subsection*{Outline of the paper}
We define and prove some basic properties of quasi post-critical finite sequences of Blaschke products in \S \ref{sec: gfb}. The abstract angled tree map is introduced in \S \ref{sec:atm} and the realization theorem is proved in \S \ref{sec:raatm}.
One direction of Theorem \ref{thm: eq} is proved in \S \ref{sec:pht}, and using the realization theorem for degenerations in $\BP_d$, we finish the proof of Theorem \ref{thm: eq} in \S \ref{sec:BPH}.
Finally, Theorem \ref{thm:sb} is proved in \S \ref{sec:sb}.

\subsection*{Acknowledgment}
The author thanks Curt McMullen and Kevin Pilgrim for useful suggestions and discussions on this problem.
The author gratefully thanks the anonymous reviewers for their careful reading and valuable comments and suggestions.

\begin{figure}
\centering
\begin{tikzpicture}[->,>=stealth']
 \node[state] (QPCF) 
 {\begin{tabular}{l}
 \parbox{4.5cm}{
  {Quasi post-critically finite degeneration $(f_n)_n \in \mathcal{B}_d$}}\\
 \end{tabular}};
 
 \node[state,    	
  yshift=-4cm, 		
  anchor=center] (GFP) 	
 {%
 \begin{tabular}{l} 	
  \parbox{5.3cm}{
  {Geometrically finite polynomial $\hat P_1 \in \partial\mathcal{H}_d$ of level 1}}\\
 \end{tabular}
 };
 
 \node[state,    	
  yshift=-7cm, 		
  anchor=center] (GGFP) 	
 {%
 \begin{tabular}{l} 	
  \parbox{5.3cm}{
  {Geometrically finite polynomial $\hat P_k \in \partial\mathcal{H}_d$ of level $k$}}\\
 \end{tabular}
 };
  
 \node[state,    	
  right of=QPCF, 	
  node distance=6.5cm, 	
  anchor=center] (QIT) 	
 {%
 \begin{tabular}{l} 	
  \parbox{5cm}{
  {Quasi invariant tree with\\ rescaling limits $((\mathcal{T}, \p), \{F_v\})$}
  }
 \end{tabular}
 };
 
 \node[state,
  right of=GFP,
  node distance=6.5cm,
  anchor=center] (PH) 
 {%
 \begin{tabular}{l}
  \parbox{3.8cm}{{Simplicial pointed\\ Hubbard tree $(H_1, p_1)$}}
 \end{tabular}
 };
 
 \node[state,
  right of=GGFP,
  node distance=6.5cm,
  anchor=center] (HT) 
 {%
 \begin{tabular}{l}
  \parbox{5cm}{{$k$-th iterated-simplicial\\ pointed Hubbard tree $(H_k, p_k)$}}
 \end{tabular}
 };

 \node[state,
  right of=QIT,
  yshift=-2cm, 		
  node distance=1cm,
  anchor=center] (ATM) 
 {%
 \begin{tabular}{l}
  \parbox{5.4cm}{{Admissible angled tree map\\ $(f: (\mathcal{T}, p) \longrightarrow (\mathcal{T}, p), \delta, \alpha=\{\alpha_v\})$}}
 \end{tabular}
 };
 
 \node[state,
  right of=QPCF,
  yshift=-2cm, 		
  node distance=1.8cm,
  anchor=center] (SB) 
 {%
 \begin{tabular}{l}
  \parbox{1.7cm}{{Self-bump on $\partial \PH_d$}}
 \end{tabular}
 };

 \path (QPCF) edge (QIT)
 (QPCF) edge[bend right=20] node[anchor=south,left]{\tiny $\hat P_1 = \lim f_n \sqcup z^d$ } (GFP)
  (GFP) edge[bend right=20] node[anchor=south,right,text width=4cm]{{\tiny Quasi post-critically finite degeneration on the bounded attracting Fatou component \\}} node[anchor=south,left]{\huge\vdots}  (GGFP)
 (QIT)  edge[bend left=20] node[anchor=south,right]{\tiny Realizing} (ATM)
 (ATM)  edge[bend left=20] node[anchor=south,left, text width=2cm]{\tiny Combinatorially classifying\\} (QIT)
 (PH)  edge[bend right=20] node[anchor=east,right,text width=3cm]{{\tiny Admissible splitting \\}} (ATM)
 (PH)  	edge[bend left=20] node[anchor=north,right, text width=2.6cm]{{\tiny Pointed simplicial tuning \\}} node[anchor=south,left]{\huge\vdots} (HT)
 (GFP)  edge (PH)
 (GGFP)  edge (HT)
 ($ (PH) !.5! (ATM) $) edge[bend left=10] node[anchor=east,below]{\tiny Different splittings} (SB);

\end{tikzpicture}
\caption{The logical flow for different concepts in this paper.}
\label{fig:LF}
\end{figure}
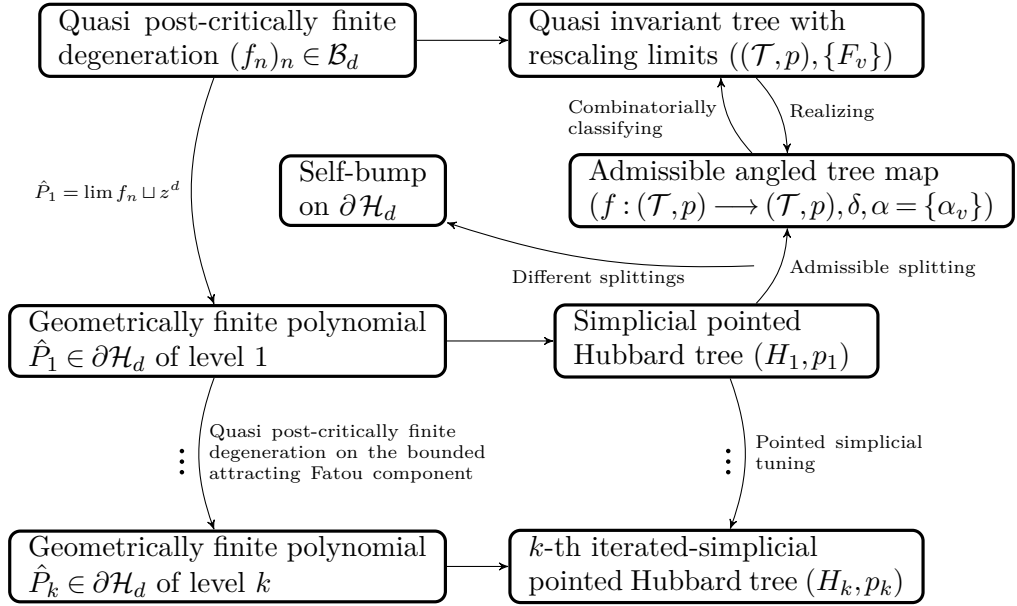

\section{Quasi post-critically finite sequences of Blaschke products}\label{sec: gfb}
A proper holomorphic map $f:\D \longrightarrow \D$ of degree $d \geq 2$ can be written as a {\em Blaschke product}
$$
f(z) = e^{i\theta}\prod_{i=0}^{d} \frac{z-a_i}{1-\overline{a_i}z},
$$
where $|a_i| < 1$.
By Schwartz reflection, any such map extends to a rational map $f: \hat\C \longrightarrow\hat \C$.
Both perspectives are useful. 
In particular, we shall use the fact that $f$ is defined on $\overline{\D}$ freely throughout the paper.

By Denjoy-Wolff theorem, there is a unique non-repelling fixed point of $f$ on $\overline{\D}$, which puts a Blaschke product $f$ into exactly three categories:
\begin{itemize}
\item $f$ is {\bf \em interior-hyperbolic} or simply {\bf \em hyperbolic} if $f$ has an attracting fixed point in $\D$;
\item $f$ is {\bf \em parabolic} if $f$ has a parabolic fixed point on $\mathbb{S}^1$;
\item $f$ is {\bf \em boundary-hyperbolic} if $f$ has an attracting fixed point on $\mathbb{S}^1$.
\end{itemize}
The parabolic Blaschke products can be further divided into {\bf \em singly parabolic} or {\bf \em doubly parabolic} depending on the multiplicities for the parabolic fixed points.
The Julia set of a hyperbolic or a doubly parabolic Blaschke product is the circle $\mathbb{S}^1$, while the Julia set of a singly parabolic or a boundary-hyperbolic Blaschke product is a Cantor set on $\mathbb{S}^1$.

In this paper, we shall mainly focus on hyperbolic Blaschke products, although the parabolic and boundary-hyperbolic ones will appear as {\em rescaling limits} as we degenerate hyperbolic Blaschke products.

For $d\geq 2$, let
$$
\BP_d := \{f(z) = z\prod_{i=1}^{d-1} \frac{z-a_i}{1-\overline{a_i}z}: |a_i| < 1\} \cong \D^{d-1}
$$
be the space of (normalized) marked hyperbolic Blaschke products. 
Each map in $\BP_d$ has an attracting fixed point at the origin.
Any hyperbolic Blaschke product of degree $d$ is conformally conjugate to a map in $\BP_d$.
We shall identify $\mathbb{S}^1 := \R/\Z$. Under this identification, we set $m_d(t) = d\cdot t$.
Note that for each $f\in \BP_d$, there exists a unique quasisymmetric homeomorphism $\eta_f:\mathbb{S}^1 \cong \R/\Z \longrightarrow \mathbb{S}^1$ with
\begin{enumerate}
\item $\eta_f \circ m_{d} = f \circ \eta_f$, and
\item $f\mapsto \eta_f$ is continuous on $\BP_d$ with $\eta_{z^d} = \mathrm{id}$.
\end{enumerate}
In particular, the conjugacy $\eta_f$ gives a way to label the periodic points of $f$ on $\mathbb{S}^1$ by periodic points of $m_d$ on $\R/\Z$ that varies continuously on $\BP_d$.

Note that the space $\BP_d\cong \D^{d-1}$ is not compact, we study a particular type of degenerations in $\BP_d$:
\begin{defn}\label{defn:qpcf}
Let $f_n \in \BP_d$.
We say it is {\em $K$-quasi post-critically finite} if we can label the critical points by $c_{i,n} \in \D, i=1,..., d-1$, such that for each sequence $c_{i,n}$ of critical points, there exists $l_i$ and $q_i$ called quasi pre-periods and quasi periods respectively with
$$
d_{\D}(f_n^{l_i}(c_{i,n}), f_n^{l_i+q_i}(c_{i,n})) \leq K.
$$
We say $f_n$ is {\em quasi post-critically finite} if it is $K$-quasi post-critically finite for some $K$.
\end{defn}

A {\em ribbon structure} on a finite tree is the choice of a planar embedding up to isotopy. The ribbon structure can be specified by a cyclic ordering of the edges incident to each vertex.
A finite tree with a ribbon structure is called a finite ribbon tree.
Recall that a map $f:\mathcal{T} \longrightarrow \mathcal{T}$ is said to be simplicial if $f$ sends an edge to an edge.

In this section, we first construct a sequence of quasi-invariant trees $\mathcal{T}_n \subseteq \D$ for a quasi post-critically finite sequence $f_n \in \BP_d$ capturing all interesting dynamical features.

\begin{theorem}\label{thm:qit}
Let $f_n \in \BP_d$ be quasi post-critically finite. 
After passing to a subsequence, there exists a constant $K > 0$, a simplicial map 
$$
f: (\mathcal{T}, \p) \longrightarrow (\mathcal{T}, \p)
$$ 
on a pointed finite ribbon tree with vertex set $\mathcal{V}$, a sequence of pointed finite ribbon trees $(\mathcal{T}_n, \p_n)$ with $\p_n \in \mathcal{T}_n \subseteq \D$, and a sequence of isomorphisms
$$
\phi_n: (\mathcal{T}, \p) \longrightarrow (\mathcal{T}_n, \p_n)
$$
such that
\begin{itemize}
\item (Degenerating tree.) If $v_1\neq v_2 \in \mathcal{V}$, then 
$$
d_{\D}(\phi_n(v_1), \phi_n(v_2)) \to \infty.
$$
\item (Geodesic edge.) If $E =[v_1,v_2] \subseteq \mathcal{T}$ is an edge, then the corresponding edge $\phi_n(E) \subseteq \mathcal{T}_n$ is a hyperbolic geodesic segment connecting $\phi_n(v_1)$ and $\phi_n(v_2)$.
\item (Critically approximating.) Any critical point of $f_n$ is within $K$ hyperbolic distance from the vertex set $\mathcal{V}_n:= \phi_n(\mathcal{V})$ of $\mathcal{T}_n$.
\item (Quasi-invariance on vertices.)
If $v\in \mathcal{V}$, then
$$
d_{\D}(f_n(\phi_n(v)), \phi_n(f(v))) \leq K \text{ for all } n.
$$

\item (Quasi-invariance on edges.) If $E\subseteq \mathcal{T}$ is an edge and $x_n \in \phi_n(E)$, then there exists $y_n \in \phi_n(f(E))$ so that
$$
d_{\D}(f_n(x_n), y_n) \leq K \text{ for all } n.
$$
If $E$ is a periodic edge of period $q$, then 
$$
d_{\D}(f_n^q(x_n), x_n) \leq K \text{ for all } n.
$$
\end{itemize}
\end{theorem}

To prepare ourselves with the construction, we will first review some results for degenerations of rational maps and geometric bounds for Blaschke products.

\subsection*{Degeneration of rational maps}
Let $\Rat_d$ denote the space of rational maps of degree $d$.
By fixing a coordinate system of $\hat \C \cong \Proj^1(\C)$, a rational map can be expressed as a ratio of two homogeneous polynomials $f(z:w) = (P(z,w):Q(z,w))$, where $P$ and $Q$ have degree $d = \deg(f)$ with no common divisors.
Thus, using the coefficients of $P$ and $Q$ as parameters, we have
$$
\Rat_d = \Proj^{2d+1} \setminus V(\Res),
$$ 
where $\Res$ is the resultant of the two polynomials $P$ and $Q$, and $V(\mathrm{Res})$ is the hypersurface for which the resultant vanishes.
This embedding gives a natural compactification $\overline{\Rat_d} = \Proj^{2d+1}$, which will be called the {\em algebraic compactification}.
A map $f\in \overline{\Rat_d}$ can be written as
$$
f = (P:Q) = (Hp: Hq),
$$
where $H = \gcd(P, Q)$.
We set $$\varphi_f := (p:q),$$ which is a rational map of degree at most $d$.
The zeroes of $H$ in $\Proj^1$ are called the {\em holes} of $f$, and the set of holes of $f$ is denoted by $\mathcal{H}(f)$.

If the coefficients of $f_n \in \Rat_d$ converge to those of $f\in \overline{\Rat_d}$, we say that $f$ is the {\em algebraic limit} of the sequence $f_n$.
We also say $f_n$ {\em converges algebraically} to $f$. 
The limit $f$ is said to have degree $k$ if $\varphi_f$ has degree $k$.
Abusing notations, sometimes we shall refer to $\varphi_f$ as the algebraic limit of $f_n$. 
The following is useful in analyzing the limiting dynamics.

\begin{lem}[\cite{DeM05} Lemma 4.2]\label{lem:ac}
If $f_n \in \Rat_d$ converges to $f$ algebraically, then $f_n$ converges to $\varphi_f$ uniformly on compact subsets of $\widehat{\C}-\mathcal{H}(f)$.
\end{lem}

The following statement is also useful in many situations:
\begin{lem}\label{lem:critc}
Let $f_n\in \Rat_d$ converge to $f$ algebraically with $\deg(f) \geq 1$.
If $x\in \mathcal{H}(f)$, then there exists a sequence of critical points $c_n$ for $f_n$ with $c_n \to x$.
\end{lem}
\begin{proof}
Suppose for contradiction that this is not true.
Let $x\in U$ be a small open neighborhood containing no critical points of $f_n$ for all large $n$ and let $C = \partial U$.
Modifying $U$ and changing the role of $0$ and $\infty$ if necessary, we may assume that $f$ is univalent on $U$, $f(C)$ is a simple closed curve and the image $U'$ does not contain $\infty$.
Since $f_n$ converges uniformly to $f$ on a neighborhood of $C$ by Lemma \ref{lem:ac}, there exists a component $C_n \subseteq f_n^{-1}(f(C))$ that converges in Hausdorff topology to $C$.
Let $U_n$ be the component of $\hat\C-C_n$ that contains $x$.
So for large $n$, $U_n$ contains no critical points of $f_n$.
Note that $f_n(\partial U_n) = f_n(C_n) = f(C) = \partial U'$, so $f_n: U_n \longrightarrow U'$ is univalent.
This is a contradiction as there exists a sequence of poles of $f_n$ converging to the hole $x$, but $\infty \notin U'$.
\end{proof}

\subsection*{Geometric bounds for Blaschke products}
The normalization $f(0) = 0$ imposed on maps $f\in \BP_d$ gives the following compactness result:
\begin{prop}\label{prop:ac1}
Let $K \geq 0$, and let $f_n: \D\longrightarrow \D$ be a sequence of proper holomorphic map of degree $d$ with $d_{\D}(0, f_n(0)) \leq K$.
Then after passing to a subsequence, $f_n$ converges compactly on $\D$ to a proper holomorphic map $f: \D \longrightarrow \D$ (of potentially lower degree).

By Schwarz reflection, $f, f_n$ are rational maps. As rational maps, $f_n$ converges algebraically to a map $f$ of degree $\geq 1$ with holes contained in $\mathbb{S}^1$.
\end{prop}
\begin{proof}
If $f_n \in \BP_d$, then the statement follows from \cite[Proposition 10.2]{McM09}.
The statement for general sequences can be derived from the above by post-composing with a bounded sequence $M_n \in \Isom(\D)$ so that $M_n\circ f_n \in \BP_d$.
\end{proof}

We remark that the rational map point of view in Proposition \ref{prop:ac1} gives more information.
By Lemma \ref{lem:ac}, we know where the convergence of $f_n \to f$ fails to be uniform on $\mathbb{S}^1$.

By Schwarz lemma, any holomorphic map $f: \D \longrightarrow \D$ is distance non-increasing with respect to the hyperbolic metric.
We will frequently use the following which gives a bound in the other direction:
\begin{theorem}[\cite{McM09}, Theorem 10.11]\label{thm:almostisometry}
There is a constant $R>0$ such that for any holomorphic map $f: \D \longrightarrow \D$ of degree $d$:
\begin{enumerate}
\item If $d_{\D}([a,b], C(f)) > R$, then $d_{\D}(f(a), f(b)) = d_{\D}(a,b) + O(1)$; and
\item If $d_{\D}([a,b], f(C(f))) > R$, then $d_{\D}(f^{-1}(a), f^{-1}(b)) = d_{\D}(a,b) + O(1)$.
\end{enumerate}
Here $C(f)$ denotes the critical set of $f$, $[a,b]$ is the hyperbolic geodesic segment that connects $a$ and $b$, and $f^{-1}$ is any branch of the inverse map that is continuous along $[a,b]$.
\end{theorem}

\subsection*{Quasi-invariant tree}
Let $f_n \in \BP_d$ be a quasi post-critically finite sequence.
We now explain how a pointed quasi-invariant tree can be constructed after possibly passing to a subsequence.
This pointed quasi-invariant tree plays a similar role as the Hubbard tree for post-critically finite polynomials.

The construction is in the same spirit as the ribbon $\R$-tree introduced in \cite{McM09}, see also \cite{L19a, L19}.
One of the key differences is that two sequences $x_n, y_n \in \D$ are identified in the ribbon $\R$-tree if $\lim d_{\D}(x_n, y_n)/ R_n = 0$, where $R_n$ is a given rescaling factor; while they are identified in the quasi-invariant tree if they are of uniform bounded distance apart.
The more restrictive identification allows us to better control the dynamics and construct rescaling limits.
Another problem with using the ribbon $\R$-tree is that it only sees the dynamics on the scale of $R_n$.
To deal with incompatible escaping rates of the critical points, we thus abandon the rescaling construction as in \cite{McM09}.

In the following, we shall use $(a_n)$ to denote a sequence, and $a_n$ as the $n$-th term of the sequence $(a_n)$. If there are multiple subindices, we shall use $(-)_n$ to emphasize the index for the sequence is $n$.

Let $\p_n = 0$ be the unique attracting fixed point of $f_n$ in $\D$.
Let 
$$
\widetilde{\mathcal{P}}:=\{(f_n^j(c_{i,n}))_n: i=1,..., d-1, j=0,..., l_i+q_i-1\} \cup \{(\p_n)\}.
$$
Note that an element of $\widetilde{\mathcal{P}}$ is a sequence of points in $\D$ indexed by $n$, and the cardinality of $\widetilde{\mathcal{P}}$ is $1+\sum_{i=1}^{d-1} l_i+q_i$.

Let $(v_n) \in \widetilde{\mathcal{P}}$.
After passing to a subsequence for $f_n$, we assume that the local degree $\deg_{v_n} f_n$ of $f_n$ at $v_n$ is constant in $n$ and the limit
$$
\lim_{n\to\infty} d_{\D}(v_n, w_n)
$$
exists for any pair of elements $(v_n), (w_n) \in \widetilde{\mathcal{P}}$ (which can possibly be $\infty$).
This defines an equivalence relation: 
$$
(v_n) \sim (w_n) \text{ if } \lim_{n\to\infty} d_{\D}(v_n, w_n) < \infty.
$$
An equivalence class $\mathcal{C}\in \widetilde{\mathcal{P}} / \sim$ will be called a {\em cluster set}.
We define the degree of $\mathcal{C}$ by
$$
\deg(\mathcal{C}) = 1+ \sum_{(v_n) \in \mathcal{C}} (\deg_{v_n} f_n - 1).
$$

Let $\mathcal{C}\in \widetilde{\mathcal{P}} / \sim$ be an equivalence class.
Note that $\mathcal{C}$ is a finite union of sequences of points in $\D$.
It is convenient to choose a representative $(v_n(\mathcal{C})) \in\mathcal{C}$ with the convention that $(v_n(\mathcal{C})) = (p_{n})$ if $\mathcal{C} = [(p_{n})]$.
We denote 
$$
\mathcal{P}:= \{(v_n(\mathcal{C})): \mathcal{C} \in \widetilde{\mathcal{P}} / \sim\}.
$$ 
Note that an element in $\mathcal{P}$ is a sequence of points in $\D$ indexed by $n$.
Since $\widetilde{\mathcal{P}}$ is a finite set, $\mathcal{P}$ is a finite set as well.

We also set 
$$
\mathcal{P}_n = \{ v_n(\mathcal{C}): \mathcal{C} \in \widetilde{\mathcal{P}} / \sim\},
$$ 
i.e., $\mathcal{P}_n$ consists of the $n$-th term of the elements of $\mathcal{P}$.
Note that $\mathcal{P}_n$ is a finite collection of points of $\D$.
If $(v_n) \neq (w_n) \in \mathcal{P}$, then $d_{\D}(v_n, w_n) \to \infty$.
Thus, after passing to a subsequence, we can assume that for any pair of distinct elements $(v_n), (w_n) \in \mathcal{P}$, $v_n \neq w_n$ for all $n$.
Thus, each point $v_n \in \mathcal{P}_n$ uniquely determines an element $(v_n) \in \mathcal{P}$, which gives a canonical identification of $\mathcal{P}_n$ with $\mathcal{P}$.
We define the degree of a point $v_n \in \mathcal{P}_n$ by
$$
\deg(v_n) = \deg(\mathcal{C})
$$
if $(v_n) \in \mathcal{P}$ represents $\mathcal{C} \in \widetilde{\mathcal{P}} / \sim$.

\begin{figure}[ht]
  \centering
  \resizebox{0.4\linewidth}{!}{
    \def\svgwidth{\columnwidth}
\begingroup%
  \makeatletter%
  \providecommand\color[2][]{%
    \errmessage{(Inkscape) Color is used for the text in Inkscape, but the package 'color.sty' is not loaded}%
    \renewcommand\color[2][]{}%
  }%
  \providecommand\transparent[1]{%
    \errmessage{(Inkscape) Transparency is used (non-zero) for the text in Inkscape, but the package 'transparent.sty' is not loaded}%
    \renewcommand\transparent[1]{}%
  }%
  \providecommand\rotatebox[2]{#2}%
  \newcommand*\fsize{\dimexpr\f@size pt\relax}%
  \newcommand*\lineheight[1]{\fontsize{\fsize}{#1\fsize}\selectfont}%
  \ifx\svgwidth\undefined%
    \setlength{\unitlength}{792bp}%
    \ifx\svgscale\undefined%
      \relax%
    \else%
      \setlength{\unitlength}{\unitlength * \real{\svgscale}}%
    \fi%
  \else%
    \setlength{\unitlength}{\svgwidth}%
  \fi%
  \global\let\svgwidth\undefined%
  \global\let\svgscale\undefined%
  \makeatother%
  \begin{picture}(1,0.77272727)%
    \lineheight{1}%
    \setlength\tabcolsep{0pt}%
    \put(0,0){\includegraphics[width=\unitlength,page=1]{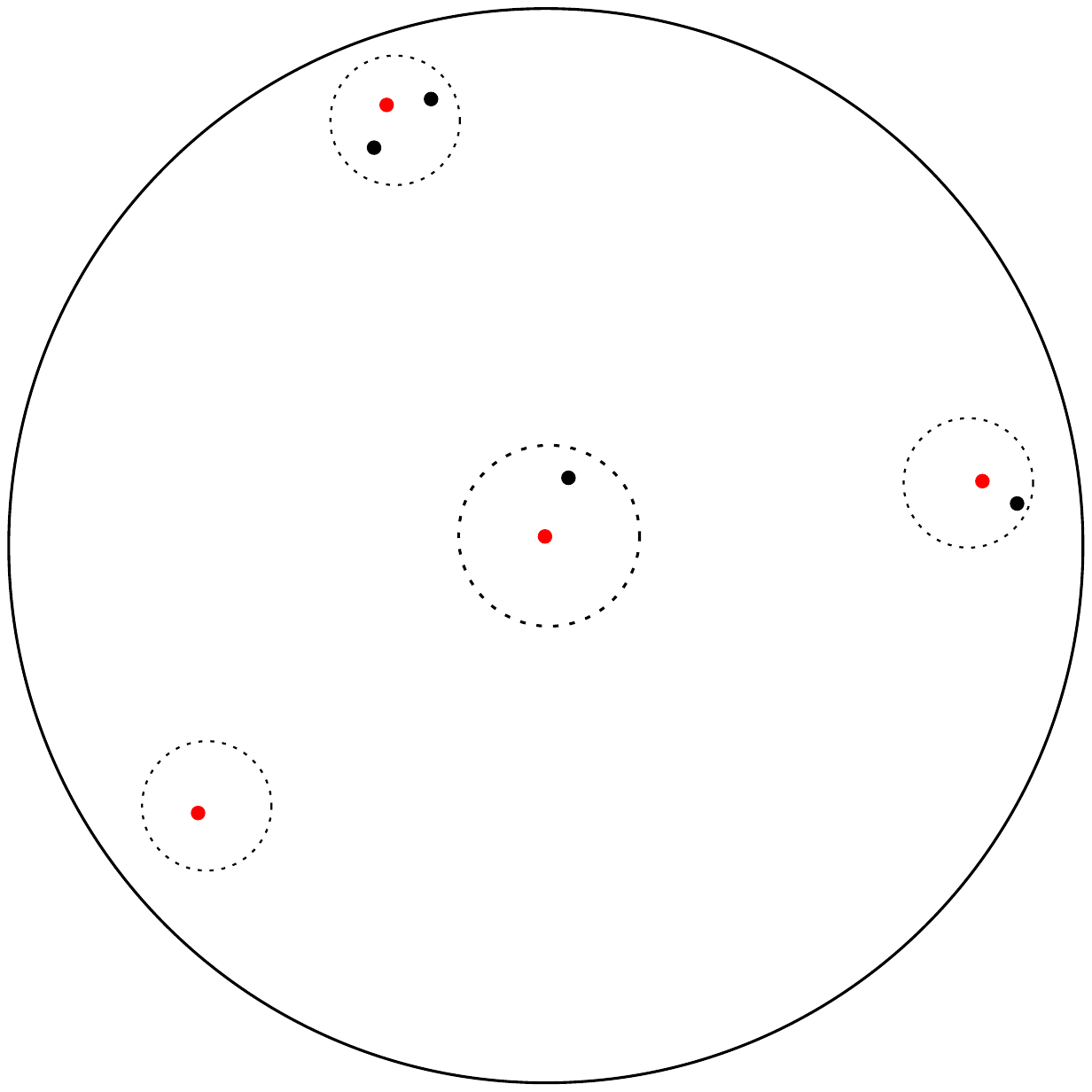}}%
    \put(0.73985778,0.66537185){\color[rgb]{0,0,0}\makebox(0,0)[lt]{\lineheight{1.25}\smash{\begin{tabular}[t]{l}{\Huge $\D$}\end{tabular}}}}%
    \put(0.5135168,0.45155565){\color[rgb]{0,0,0}\makebox(0,0)[lt]{\lineheight{1.25}\smash{\begin{tabular}[t]{l}{\Large $\mathcal{C}$}\end{tabular}}}}%
    \put(0.4535183,0.36724683){\color[rgb]{0,0,0}\makebox(0,0)[lt]{\lineheight{1.25}\smash{\begin{tabular}[t]{l}{\Large $v_n(\mathcal{C})$}\end{tabular}}}}%
    \put(0,0){\includegraphics[width=\unitlength,page=2]{Cluster.pdf}}%
  \end{picture}%
\endgroup%

  }
  \caption{A schematic illustration of cluster sets. We choose a representative (red) for each cluster. As $n\to \infty$, the distance is bounded between points in the same cluster, while goes to infinity between points in different clusters.}
  \label{fig:Cluster}
\end{figure}

Index the elements of $\mathcal{P} = \{(b_{0,n})_n, (b_{1,n})_n,..., (b_{m,n})_n\}$ with $(b_{0,n}) = (\p_{n})$.
The sequence of quasi-invariant trees $\mathcal{T}_n$ is the `spine' of the degenerating hyperbolic polygon $\chull(\mathcal{P}_{n})$ and is constructed inductively as follows:

As the base case, we define $\mathcal{T}^0_n = \{b_{0,n}\}$ with vertex set $\mathcal{V}^0_n=\{b_{0,n}\}$.

Assume that $\mathcal{T}^i_n$ is constructed with vertex set $\mathcal{V}^i_n:= \{v_{1,n},..., v_{m_i,n}\}$ containing $\{b_{0,n},..., b_{i,n}\}$. 
Assume as induction hypotheses that 
\begin{enumerate}
\item $\lim_{n\to\infty}d_{\D}(v_{k,n}, v_{k',n}) = \infty$ for all $k\neq k' \leq m_i$;
\item $\min_{k=i+1,..., m} d_{\D}(b_{k,n}, \chull(\mathcal{V}^{i}_n)) \to \infty$ as $n\to\infty$;
\item each edge of $\mathcal{T}^i_{n}$ is a hyperbolic geodesic segment.
\end{enumerate}

After passing to a subsequence and changing indices, we assume for all $n$,
$$
d_{\D}(b_{i+1,n}, \chull(\mathcal{V}^{i}_n)) = \min_{k=i+1,..., m} d_{\D}(b_{k,n}, \chull(\mathcal{V}^{i}_n)).
$$
Let $a_{i+1,n}$ be the projection of $b_{i+1,n}$ to the convex hull $\chull(\mathcal{V}^i_n)$.
After passing to a subsequence, we assume
$\lim_{n\to\infty} d_{\D}(a_{i+1,n}, v_{k,n})$
exist (which can possibly be $\infty$) for all $k$ (see Figure \ref{fig:Construction}).

Since $a_{i+1,n}$ is on the boundary of $\chull(\mathcal{V}^i_n)$, there exist $u_n, w_n \in \mathcal{V}^{i}_n$ so that $a_{i+1,n}$ is on the hyperbolic geodesic segment $[u_n, w_n]$.
Since $\mathcal{T}^i_n$ is a tree and each edge is a hyperbolic geodesic segment, there is a unique piecewise geodesic path in $\mathcal{T}^i_n$ that connects $u_n$ and $w_n$. We denote it by
$$
[v_{j_1,n} = u_n, v_{j_2,n}] \cup [v_{j_2,n}, v_{j_3,n}] \cup ... \cup [v_{j_{l-1}, n}, v_{j_l,n} = w_n],
$$
where $[v_{j_i,n}, v_{j_{i+1},n}]$ is an edge of $\mathcal{T}^i_n$.
After passing to a subsequence, we assume that the subindices $j_i$ does not depend on $n$.

\begin{figure}[ht]
  \centering
  \resizebox{0.9\linewidth}{!}{
    \def\svgwidth{\columnwidth}
    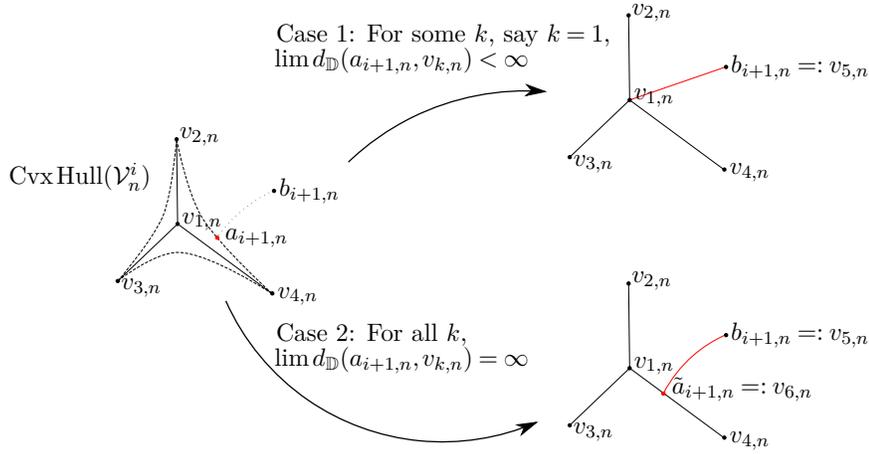

  }
  \caption{An illustration of the inductive construction of $\mathcal{T}_{s,n}$.}
  \label{fig:Construction}
\end{figure}

We have two cases. 
If $a_{i+1,n}$ stay at a bounded distance from some vertex $v_{k,n} \in \mathcal{V}^i_n$, then we define
$$
\mathcal{T}^{i+1}_n := \mathcal{T}^{i}_n \cup [v_{k,n}, b_{i+1,n}],
$$
and set $\mathcal{V}^{i+1}_n:=\mathcal{V}^i_n \cup \{b_{i+1,n}\}$ (see Figure \ref{fig:Construction}).

Otherwise, since the hyperbolic triangles are thin, subdividing the polygon $v_{j_1,n},..., v_{j_k,n}$ into finitely many triangles, there exist 
$$
\tilde{a}_{i+1,n} \in [v_{j_1,n}, v_{j_2,n}] \cup ... \cup [v_{j_{l-1}, n}, v_{j_l,n}] \subseteq \mathcal{T}^i_n
$$ 
which stay at a bounded distance from $a_{i+1,n}$. 
Note that  the bound depends only on the number of vertices.
We define
$$
\mathcal{T}^{i+1}_n := \mathcal{T}^{i}_n \cup [\tilde{a}_{i+1,n}, b_{i+1,n}],
$$
and set the vertex set $\mathcal{V}^{i+1}_n:=\mathcal{V}^i_n \cup \{\tilde{a}_{i+1,n}, b_{i+1,n}\}$ (see Figure \ref{fig:Construction}).

By our construction, it is easy to verify that the 3 induction hypotheses are satisfied.
Thus, by induction, let $\mathcal{T}_n = \mathcal{T}^m_n$ be the finite tree when all $m$ points in $\mathcal{P}_n$ are added.

By construction, $\mathcal{P}_n \subseteq \mathcal{V}_n$, and any point in $\mathcal{V}_n - \mathcal{P}_n$ is a branch point for $\mathcal{T}_n$. 
Moreover, $\lim_{n\to\infty} d_{\D}(v_{k,n}, v_{k',n}) = \infty$ for all $k\neq k'$ and the edges are hyperbolic geodesic segments.

After passing to a further subsequence, we may assume that $\mathcal{T}_n$ are all canonically isomorphic with the same ribbon structure.
Thus, we denote by $\mathcal{T}$ the underlying ribbon finite tree, with the isomorphisms 
$$
\phi_n: \mathcal{T}\longrightarrow \mathcal{T}_n.
$$
We denote the vertex set as $\mathcal{V}$.
By our construction, $\phi_n^{-1}(\mathcal{P}_n) \subseteq \mathcal{V}$ is canonically identified with $\mathcal{P}$ and the identification does not depend on $n$.
In this way, we shall abuse the notation and regard $\mathcal{P} \subseteq \mathcal{V}$.

The local degree at a vertex is defined by $\delta(v) := \deg (\phi_n(v))$ if $v\in \mathcal{P}$, and $\delta(v) := 1$ if $v\notin \mathcal{P}$.
There exists a unique vertex corresponding to the attracting fixed point for $f_n$, and we shall denote it by $\p \in \mathcal{V}$.
Note that by construction, $\delta(v) -1$ equals to the number of critical points of $f_n$ counted with multiplicity that stay at a bounded distance from $\phi_n(v)$.

We first show that the vertex set is quasi forward invariant.
\begin{lem}\label{lem:qfiv}
There exists an induced map $f:\mathcal{V} \longrightarrow \mathcal{V}$ and a constant $K$ such that for all $v\in \mathcal{V}$ and for all $n$,
$$
d_{\D}(f_n(\phi_n(v)), \phi_n(f(v))) \leq K.
$$
\end{lem}
\begin{proof}
First consider the case $v\in \mathcal{P}$. 
Note that there exists an induced map $\tilde f: \widetilde{\mathcal{P}} \longrightarrow \widetilde{\mathcal{P}}$ defined as follows.
Let $(v_n) \in \widetilde{\mathcal{P}}$.
\begin{itemize}
\item If $(v_n) = (\p_n)$, then $\tilde f ((v_n)) = (v_n)$.
\item If $(v_n) = (f_n^j(c_{i,n}))$ with $j \leq l_i+q_i-2$, then $\tilde f ((v_n)) = (f_n^{j+1}(c_{i,n}))$.
\item If $(v_n) = (f_n^j(c_{i,n}))$ with $j = l_i+q_i-1$, then $\tilde f ((v_n)) = (f_n^{l_i}(c_{i,n}))$.
\end{itemize}
Note that in the first two cases, $\tilde f ((v_n)) = (f_n(v_n))$.
In the third case, $\tilde f ((v_n)) = (f_n^{l_i}(c_{i,n}))$ stay at a bounded distance from $(f_n(v_n)) = (f_n^{l_i+q_i}(c_{i,n}))$ as $f_n$ is quasi post-critically finite.

If $(v_n) \sim (w_n)$, then by Schwarz lemma, $f_n(v_n)$ stay at a bounded distance from $f_n(w_n)$, so $\tilde f((v_n)) \sim \tilde f ((w_n))$.
Thus the map  $\tilde f: \widetilde{\mathcal{P}} \longrightarrow \widetilde{\mathcal{P}}$ descends to a map
$f: \mathcal{P} \longrightarrow \mathcal{P}$.
By our construction, there exists a constant $K$ so that
$$
d_{\D}(f_n(\phi_n(v)), \phi_n(f(v))) \leq K.
$$

Now suppose $v\in \mathcal{V} - \mathcal{P}$. Then $v$ is a branch point. 
Let $b_1, b_2, b_3$ be the closest points to $v$ on $\mathcal{P}$ so that the convex hull of them in $\mathcal{T}$ forms a `tripod' with $v$ as the center.
Then the angles $\angle \phi_n(b_i)\phi_n(v)\phi_n(b_j)$ are uniformly bounded away from $0$ for $i\neq j$ by construction.
Thus $\phi_n(v)$ is of a uniform bounded distance from the hyperbolic geodesic segment $[\phi_n(b_i), \phi_n(b_j)]$ that connects $\phi_n(b_i)$ and $\phi_n(b_j)$.
By construction, the geodesic $[\phi_n(b_i), \phi_n(b_j)]$ is close to critical points (in the sense of the bound in Theorem \ref{thm:almostisometry}) only possibly near its end points.
Thus by Theorem \ref{thm:almostisometry}, $d_{\D}(f_n(\phi_n(b_i)), f_n(\phi_n(b_j))) = d(\phi_n(b_i), \phi_n(b_j)) + O(1)$.
Therefore, the angles $\angle f_n(\phi_n(b_i)) f_n(\phi_n(v)) f_n(\phi_n(b_j))$ are all uniformly bounded away from $0$ for $i\neq j$.
Thus, $f_n(\phi_n(v))$ is of a uniform bounded distance from some branch point $\phi_n(v')$ of $\mathcal{T}_n$.
Therefore, we define $f(v) = v'$ in this case, and the lemma follows.
\end{proof}

The above lemma allows us to define an induced map 
$f: (\mathcal{T}, \p) \longrightarrow (\mathcal{T}, \p)$ 
by extending $f$ continuously on any edge $[v, w]$ to the arc $[f(v), f(w)]$.
Consider the finite subtree $\mathcal{T}^P \subseteq \mathcal{T}$ as the convex hull of the periodic vertices of $f$. 
We first show 
\begin{lem}\label{lem:simplicial}
Any vertex $v\in \mathcal{T}^P$ is periodic and $f: \mathcal{T}^P \longrightarrow \mathcal{T}^P$ is a simplicial map, i.e, if $[v,w]$ is an edge of $\mathcal{T}^P$, then $[f(v), f(w)]$ is also an edge of $\mathcal{T}^P$.
\end{lem}
\begin{proof}
If $v$ is a vertex in $\mathcal{T}^P$, then there are two periodic vertices $w_1, w_2$ so that $v\in [w_1, w_2]$.
Let $q$ be the least common multiplier of the periods of $w_1$ and $w_2$.
By Lemma \ref{lem:qfiv}, 
$d_{\D}(\phi_n(w_i), f_n^q(\phi_n(w_i))) = O(1)$. 
By Schwarz lemma, $d_{\D}(f_n^q(\phi_n(v)), f_n^q(\phi_n(w_i))) \leq d_{\D}(\phi_n(v), \phi_n(w_i))$, thus we have 
$$
d_{\D}(f_n^q(\phi_n(v)), \phi_n(w_i)) \leq d_{\D}(\phi_n(v), \phi_n(w_i)) + O(1).
$$ 
Since the angle $\angle \phi_n(w_1)\phi_n(v)\phi_n(w_2)$ is uniformly bounded away from $0$, 
$$
d_{\D}(\phi_n(w_1), \phi_n(w_2)) = d_{\D}(\phi_n(v), \phi_n(w_1)) + d_{\D}(\phi_n(v), \phi_n(w_2)) + O(1).
$$
A standard hyperbolic geometry estimate using the horocycles gives that 
$$
d_{\D}(f_n^q(\phi_n(v)),\phi_n(v)) = O(1).
$$ 
This means that $v$ is periodic.
Thus $\mathcal{T}^P$ is invariant under $f$.

A similar proof also shows $f$ is simplicial.
Indeed, suppose for contradiction that $f$ is not simplicial. Then there exists an edge $[w_1,w_2] \subseteq \mathcal{T}^P$ so that $[f(w_1), f(w_2)]$ is not an edge.
Let $v\in (f(w_1), f(w_2))$ be a vertex.
Then $v$ is also periodic by the previous paragraph.
Let $q$ be the least common multiplier of the periods of $w_1, w_2, v$.
Note that
$$
d_{\D}(\phi_n(f(w_1)), \phi_n(f(w_2))) = d_{\D}(\phi_n(f(w_1)), \phi_n(v)) + d_{\D}(\phi_n(v), \phi_n(f(w_2))) + O(1).
$$
Thus apply Lemma \ref{lem:qfiv} to $f^{q-1}$,
$$
d_{\D}(\phi_n(w_1), \phi_n(w_2)) = d_{\D}(\phi_n(w_1), \phi_n(f^{q-1}(v))) + d_{\D}(\phi_n(f^{q-1}(v))), \phi_n(w_2)) + O(1).
$$
Therefore, $f^{q-1}(v) \in (w_1,w_2)$ which is a contradiction.
\end{proof}

Since the vertices of $\mathcal{T}_n$ are uniformly quasi-invariant, we show $f_n$ is quasi invariant on $\mathcal{T}_n$ with dynamics modeled by $f: (\mathcal{T}, p) \longrightarrow (\mathcal{T}, p)$.
\begin{prop}\label{prop:qi}
There exists a constant $K$ such that 
\begin{itemize}
\item if $[v_1, v_2]\subseteq \mathcal{T}$ is an edge, then for any point $x_n\in [\phi_n(v_1), \phi_n(v_2)]$, there exists $y_n \in \phi_n(f([v_1,v_2]))$ so that
$$
d_{\D}(f_n(x_n), y_n) \leq K \text{ for all $n$};
$$

\item if $[v_1, v_2]\subseteq \mathcal{T}$ is a periodic edge of period $q$, then for any point $x_n\in [\phi_n(v_1), \phi_n(v_2)]$,
$$
d_{\D}(f_n^q(x_n), x_n) \leq K \text{ for all $n$}.
$$
\end{itemize}
\end{prop}
\begin{proof}
The proof uses the same estimates as in the proof of Lemma \ref{lem:simplicial}.
By Schwarz lemma,
$$
d_{\D}(f_n(x_n), f_n(\phi_n(v_i)))\leq d_{\D}(x_n, \phi_n(v_i)).
$$ 
Thus, by Lemma \ref{lem:qfiv}, 
$$
d_{\D}(f_n(x_n), \phi_n(f(v_i)))\leq d_{\D}(x_n, \phi_n(v_i)) + O(1).
$$ 
By Theorem \ref{thm:almostisometry} and Lemma \ref{lem:qfiv},
\begin{align*}
d_{\D}(\phi_n(f(v_1)), \phi_n(f(v_2))) &= d_{\D}(\phi_n(v_1), \phi_n(v_2)) + O(1)\\
&= d_{\D}(x_n, \phi_n(v_1)) + d_{\D}(x_n, \phi_n(v_2)) + O(1)
\end{align*}
Therefore, there exists a constant $K$ so that $d_{\D}(f_n(x_n), \phi_n(f([v,w]))) \leq K$.

A similar argument proves the part on periodic edges.
\end{proof}

\begin{cor}\label{cor:ne}
The map $f$ is non-expanding. More precisely, let $E$ be an open edge of $\mathcal{T}$. 
If there exists $k$ with $f^k(E) \cap E \neq \emptyset$, then $f^k: E \longrightarrow E$ is a homeomorphism.
\end{cor}
\begin{proof}
By Schwarz lemma, $f_n$ is distance decreasing. By Proposition \ref{prop:qi}, the dynamics of $f^k$ on $E$ is approximated by $f_n^k$ on $\phi_n(E)$. Thus $f$ is non-expanding.
\end{proof}

\begin{cor}\label{cor: sim}
We can add finitely many vertices on $\mathcal{T}$ so that $f: (\mathcal{T}, p) \longrightarrow (\mathcal{T}, p)$ is simplicial.
\end{cor}
\begin{proof}
By Corollary \ref{cor:ne}, each recurring edge is periodic.
Thus, for each edge $E$, there exists $n$ so that $f^n(E) \subseteq \mathcal{T}^P$.
We call the smallest such integer the {\em generation} of $E$.
By Lemma \ref{lem:simplicial}, $f$ is simplicial on generation $0$ edges.
If $E$ is generation $1$, then $f(E)$ is a finite union of generation $0$ edges.
By adding finitely many vertices on $E$, we may assume $f$ is simplicial on $E$.
The corollary now follows by induction.
\end{proof}

\begin{proof}[Proof of Theorem \ref{thm:qit}]
Let $f: (\mathcal{T}, \p)\rightarrow (\mathcal{T},\p)$ with isomorphisms $\phi_n (\mathcal{T}, \p)\rightarrow (\mathcal{T}_n,\p_n)$ constructed as above.
Then Corollary \ref{cor: sim} shows $f$ is simplicial.
By construction of $\mathcal{T}_n$, the first $3$ conditions are satisfied.
By Lemma \ref{lem:qfiv} and Proposition \ref{prop:qi}, the map is quasi-invariant.
\end{proof}

\subsection*{Rescaling limits}
Let $v\in \mathcal{V}$. 
We define a {\em normalization at $v$} or a {\em coordinate at $v$} as a sequence $M_{v,n}\in \Isom(\D)$ so that $M_{v,n}(0) = \phi_n(v)$.
Note that there are many choices for $M_{v,n}$, and they differ by pre-composing with rotations that fix $0$.

We shall think of the sequence $M_{v,n}$ as giving `coordinates' at $v$ which gives an associated limiting disk $\D_v$ and a limiting circle $\mathbb{S}^1_v$ for the vertex $v$.
More precisely, we say a sequence $z_n \in \overline{\D}$ converges to $z\in \overline{\D}_v$ in $v$-coordinate, denoted by $z_n \to_v z$ or $z = \lim_v z_n$ if
$$
\lim_{n\to\infty} M^{-1}_{v,n} (z_n) = z.
$$
Here the subindices in $\D_v$ and $\mathbb{S}^1_v$ are used to distinguish different coordinates at different vertices.

By Proposition \ref{prop:qi}, there exists a constant $K$ such that
$$
d_{\D}(0,M_{f(v),n}^{-1}\circ f_n \circ M_{v,n}(0)) = d_{\D}(M_{f(v),n}(0),f_n\circ M_{v,n}(0)) \leq K.
$$
Thus, by Proposition \ref{prop:ac1}, after passing to subsequences, the sequence
$$
M_{f(v),n}^{-1}\circ f_n \circ M_{v,n}
$$
converges compactly on $\D$ to a proper holomorphic map $F_v$ on the unit disk.
We call $F_v$ the {\em rescaling limit} between $v$ and $f(v)$.

We remark that it is important to keep track of the changing coordinates, thus one shall really think of $F_v$ as a map 
$$
F_v : \D_v \longrightarrow \D_{f(v)}.
$$
To emphasize this, we sometimes also write $F_v = F_{v\to f(v)}$.

We also remark that $M_{f(v),n}^{-1}\circ f_n \circ M_{v,n}$ and $F_v$ extend to rational maps on $\hat\C$ of degree $\geq 1$, and $M_{f(v),n}^{-1}\circ f_n \circ M_{v,n}$ converges algebraically to $F_v$.
This perspective is useful as it gives more information on where the convergence $M_{f(v),n}^{-1}\circ f_n \circ M_{v,n}$ to $F_v$ fails to be uniform on $\mathbb{S}^1$.

More generally,
$M_{f^k(v),n}^{-1}\circ f^k_n \circ M_{v,n}$
converges compactly on $\D$ to $F_{f^{k-1}(v)}\circ F_{f^{k-2}(v)}\circ...\circ F_{v}$.
We shall denote this composition by
$$
F^k_{v}:=F_{f^{k-1}(v)}\circ F_{f^{k-2}(v)}\circ...\circ F_{v}.
$$
In particular, if $v$ is a periodic point of period $q$, then
$$
M_{v,n}^{-1}\circ f^q_n \circ M_{v,n}
$$
converges compactly on $\D$ to $F^q_v$.
We shall call this map $F^q_v$ the {\em first return rescaling limit} at $v$.
Similar as before, all these maps extend to rational maps on $\hat\C$, and the convergences are algebraic.

Recall that $\delta(v)$ is the local degree of a vertex $v \in \mathcal{V}$, and $\delta(v)-1$ equals to the number of critical points of $f_n$ counted with multiplicity that stay at a bounded distance from $\phi_n(v)$.
Thus, the rescaling limit $F_v$ at $v$  has degree $\delta(v)$.
Also recall that $\eta_n: \R/\Z \longrightarrow \mathbb{S}^1$ is the marking for $f_n \in \BP_d$.
So $\eta_n(0)$ is the marked repelling fixed point of $f_n$ on $\mathbb{S}^1$.

By precomposing the normalizations with rotations, we make the following {\em anchored }convention on the normalizations. Throughout this paper, the normalizations will always be assumed to be anchored (see Figure \ref{fig:RL}).
\begin{defn}\label{defn:anchored}
Let $f_n \in \BP_d$ be quasi post-critically finite, with simplicial tree map $f: (\mathcal{T},\p) \longrightarrow (\mathcal{T}, \p)$. 
The normalizations $M_{v, n}$ at vertices $v\in \mathcal{V}$ are said to be {\em anchored} if they satisfy the following conditions.
\begin{itemize}
\item If $v = \p$ and $\delta(\vp) = 1$, then $M_{\vp,n}$ is chosen so that $\lim_\vp \eta_n(0)=1\in \mathbb{S}^1_{\vp}$.
\item If $v = \p$ and $\delta(\vp) \geq 2$, then $M_{\vp,n}$ is chosen so that $1 \in \mathbb{S}^1_{\vp}$ is the nearest fixed point of $F_\vp$ to $\lim_\vp \eta_n(0)$ counterclockwise.
\item If $v\neq \p$, then $M_{v,n}$ is chosen so that $\lim_v \phi_n(\p) = 1 \in \mathbb{S}^1_v$.
\end{itemize}
\end{defn}

To illustrate the second case in Definition \ref{defn:anchored}, consider the sequence $f_n(z) = z^3 \frac{z-(1-1/n)}{1-(1-1/n)z}$. One can verify that it is quasi post-critically finite.
Note that $L_{p,n}(z) = z$ is a normalization at $p$. 
Under this normalization, $\eta_n(0) \to_p 1 \in \mathbb{S}^1_p$ as $\eta_n(0) = 1$ is the marked fixed point for $f_n$ for all $n$.
This normalization $L_{p,n}$ does not satisfy the anchored convention as the rescaling limit 
$$
F_p(z) = \lim L_{p,n}^{-1} \circ f_n \circ L_{p,n}(z) = \lim f_n(z) =  -z^3
$$ 
has fixed points at $\pm i$.
Since the fixed point $i$ is closer to $1 = \lim_n \eta_n(0)$ counterclockwise, one can verify that the sequence $M_{p,n}(z) = iz$ is a normalization at $p$, and it satisfies the anchored convention.

\begin{figure}[ht]
  \centering
  \resizebox{0.9\linewidth}{!}{
    \def\svgwidth{\columnwidth}
    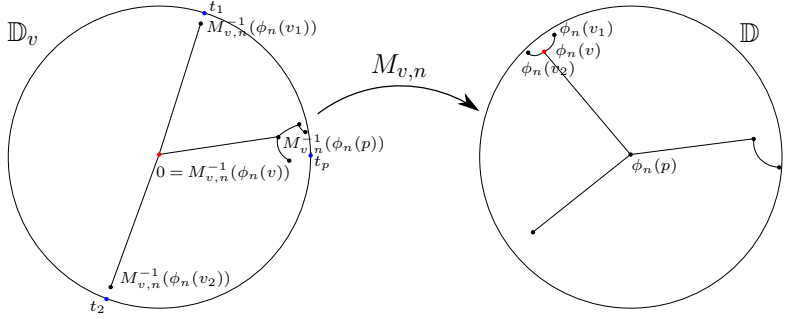

  }
  \caption{An illustration of $M_{v,n}$ for $v\neq p$, and the marked points on $\mathbb{S}^1_v$. This normalization satisfies the anchored convention as $\lim_v\phi_n(p) = t_p = 1 \in \mathbb{S}^1_v$.}
  \label{fig:RL}
\end{figure}

\subsection*{A marking of $\mathbb{S}^1$ at a vertex}
Let $v_1,..., v_l$ be the list of adjacent vertices of $v$ in $\mathcal{T}$. 
Then $\{v_1,..., v_l\}$ can be identified with the tangent space $T_v\mathcal{T}$ of $\mathcal{T}$ at $v$.
By our construction, after passing to a subsequence, there exist $l$ distinct points $t_1,..., t_l \in \mathbb{S}^1_v$ with
$t_i = \lim_v \phi_n(v_i)$ (see Figure \ref{fig:RL}).
We denote this correspondence by the map
$$
\xi_v: T_v\mathcal{T} \longrightarrow \mathbb{S}^1_v.
$$
Since the angles between different edges at a vertex in $\mathcal{T}_n$ are uniformly bounded from below by construction, in the $v$-coordinate, $\mathcal{T}_n$ converges to the union of geodesic rays connecting $0$ and points in $\xi_v (T_v\mathcal{T})$.
In particular, if $w\neq v$ is in the same component of $\mathcal{T}-\{v\}$ as $v_i$, then $\phi_n(w) \to_v t_i$.

Since $f$ is injective on each edge, we have a well-defined tangent map $Df:T_v\mathcal{T} \longrightarrow T_{f(v)}\mathcal{T}$. 
On the other hand, the rescaling limit $F_{v}$ gives a maps from $\mathbb{S}^1_v$ to $\mathbb{S}^1_{f(v)}$.
It is easy to check that this marking of $\mathbb{S}^1$ at a vertex is compatible with the dynamics of the tangent map:
\begin{lem}\label{lem:ct}
Let $v\in \mathcal{V}$. Then the holes of the rescaling limit $F_v$ are contained in $\xi_v(T_v\mathcal{T}) \subseteq \mathbb{S}^1_v$. Moreover,
$$
F_{v} \circ \xi_v = \xi_{f(v)} \circ Df.
$$
\end{lem}
\begin{proof}
The first statement follows immediately from Lemma \ref{lem:critc}.

To prove the second part, we let $w$ be adjacent to $v$ with $\phi_n(w)\to_v t$.
Thus, $[\phi_n(v), \phi_n(w)] \to_v [0,t]$.
Let $t' := \xi_{f(v)} \circ Df (w) \in \mathbb{S}^1_{f(v)}$. 
Then $[\phi_n(f(v)), \phi_n(f(w))] \to_{f(v)} [0, t']$.
Suppose for contradiction that $F_{v}(t) \neq t'$. Note that $M_{f(v),n}^{-1}\circ f_n \circ M_{v,n}$ converges to $F_{v}$ uniformly away from finitely many holes.
Let $K$ be the constant in Proposition \ref{prop:qi}. 
Then we conclude by the uniform convergence that there exists $x\in [0,t]$ such that
$$
d_{\D}([0,t'], M_{f(v),n}^{-1}\circ f_n \circ M_{v,n}(x)) > 2K.
$$
This is a contradiction.
\end{proof}

\begin{cor}
Let $v\neq \p \in \mathcal{V}$ be periodic of period $q$. Then $F^q_v$ has a non-repelling fixed point at $1$.
\end{cor}
\begin{proof}
Let $a \in T_v \mathcal{T}$ be the direction associated to $\p$.
Since $f$ is simplicial on $\mathcal{T}$ and $f(\p) = \p$, $Df^q(a) = a$.
Since the normalizations are anchored, the direction $a$ corresponds to $1\in \mathbb{S}^1_v$. 
By Lemma \ref{lem:ct}, $F_v(1) = 1$.
It is non-repelling by Schwarz lemma and the fact that the fixed point $\phi_n(p) \to_v 1$.
\end{proof}
\begin{defn}\label{defn:mcpc}
We introduce an {\em anchored marking} to anchored normalizations $M_{v,n}$ by associating a point $t_v \in \mathbb{S}^1_v$ with the following rules:
\begin{itemize}
\item if $v$ is periodic, $t_v$ is the periodic point $1\in \mathbb{S}^1_v$;
\item inductively, we choose $t_v$ on strictly pre-periodic points as the nearest $t$ to $1$ in counterclockwise orientation with $F(t) = t_{f(v)}$.
\end{itemize}
\end{defn}

\subsection*{Landing of pre-periodic points}
The pre-periodic points for $f_n$ are marked by pre-periodic points of $m_d$ on $\R/\Z$ using $\eta_n:=\eta_{f_n}$.
Let $x \in \Q/\Z$ be a rational angle. 
Then $\eta_n(x)$ is a pre-periodic point for $f_n$.

Denote the projection of $\eta_n(x)$ to the convex hull $\chull(\mathcal{T}_n)$ by $\proj_{\mathcal{T}_n}(\eta_n(x))$.
Since there are only countably many pre-periodic points, by a standard diagonal argument and passing to subsequences, we may assume that
$$
\lim_{n\to\infty} d_{\D}(\proj_{\mathcal{T}_n}(\eta_n(x)), \phi_n(v))
$$
exist (which can possibly be $\infty$) for all $v\in \mathcal{V}$ and for all pre-periodic point $x$.
We say $x$ {\em lands} at a vertex $v\in \mathcal{V}$ if
$$
\lim_{n\to\infty} d_{\D}(\proj_{\mathcal{T}_n}(\eta_n(x)), \phi_n(v)) < \infty.
$$
We say $x$ {\em lands} on $\mathcal{V}$ if $x$ lands at $v$ for some $v\in \mathcal{V}$.
The collection of $x\in \Q/\Z$ that lands at $v$ is called the {\em landing angles} at $v$.

After passing to a subsequence, we may assume the limits $\lim_v \eta_n(x)$ exist in $\mathbb{S}^1_v$ for all $v$ and all pre-periodic points $x$.
Another way to define the landing angles at $v$ is to look at the rescaling limits.
It can be checked easily from the definition that
\begin{lem}\label{lem:pland}
A pre-periodic point $x$ lands at $v\in \mathcal{V}$ if and only if
$$
\lim_v \eta_n(x) \notin \xi_v(T_v \mathcal{T}).
$$
\end{lem}

If $x$ does not land on $\mathcal{V}$, then there exists an edge $E = [v,w] \in \mathcal{T}$ so that the projection $\proj_{\mathcal{T}_n}(\eta_n(x))$ is within bounded distance to $[\phi_n(v),\phi_n(w)]$.
We say $x$ lands on the edge $E$ in this case.

The degree of a cycle $C \subseteq \mathbb{S}^1$ for $m_d$ is the least $e\geq 1$ so that $m_d|_C$ extends to a covering of the circle of degree $e$.
A cycle $C$ is said to be {\em simple} if the degree of $C$ is $1$ (see \cite{McM10} for more details).
We now show that almost all pre-periodic points land on $\mathcal{V}$:
\begin{prop}\label{prop:al}
All but finitely many pre-periodic points land on $\mathcal{V}$.

Moreover, let $\mathcal{C}=\cup C_i$ be the set of periodic points that do not land on $\mathcal{V}$. Then 
\begin{itemize}
\item every cycle $C_i$ is simple with the same rotation number; 
\item $m_d|_\mathcal{C}$ preserves the cyclic ordering of $\mathcal{C}$.
\end{itemize}
Any strictly pre-periodic point that does not land on $\mathcal{V}$ is eventually mapped to $\mathcal{C}$.
\end{prop}
\begin{proof}
Let $x$ be a periodic point of period $q$ for $m_d$. Suppose that $x$ does not land on $\mathcal{V}$.
We first show that the multiplier at $\eta_n(x)$ converges to $1$.
Let $a_n \in \mathcal{T}_n$ be of uniform bounded distance away from $\proj_{\mathcal{T}_n}(\eta_n(x)) \in \chull(\mathcal{T}_n)$.
Note that $d_{\D} (a_n, \phi_n(v)) \to \infty$ for any vertex $v\in \mathcal{V}$.

Let $M_{a,n} \in \Isom(\D)$ so that $M_{a,n}(0) = a_n$.
Similar to the case of vertices of $\mathcal{T}_n$, we can define $a$-limit in this setting.
In this $a$-coordinate, after passing to a subsequence, $\mathcal{T}_n$ converges to a geodesic passing through $0$.
Thus, there exist two distinct points $t_1, t_2 = -t_1 \in \mathbb{S}^1_a$ corresponding to the $a$-limits of vertices.
After passing to a subsequence, we assume $\eta_n(x) \to_a t \in \mathbb{S}^1_a$.
Since $a_n$ and $\proj_{\mathcal{T}_n}(\eta_n(x))$ are uniformly bounded apart, $t \neq t_i$.

After passing to a subsequence, by Proposition \ref{prop:qi} and Proposition \ref{prop:ac1}, $M_{a,n}^{-1} \circ f_n^q \circ M_{a,n}$
converges algebraically to a degree $1$ map $F$.
A similar argument as in Lemma \ref{lem:ct} gives that $F(t_i) = t_i$.
Since $\eta_n(x)$ is fixed by $f_n^q$ and $t$ is not a hole for $F$, $F(t) = t$.
Thus, $F = \mathrm{id}$ as it fixes $3$ points.
Therefore, $(f_n^q)'(\eta_n(x))\to 1$,
in particular, $\log |(f_n^q)'(\eta_n(x))| \leq \log 2$ for all sufficiently large $n$.
By \cite[Theorem 1.1]{McM10}, these cycles $C_i$ are all simple with the same rotation number, and $m_d|_\mathcal{C}$ preserves the cyclic ordering of $\mathcal{C}$.

Since there are only finitely many simple cycles with the same rotation number (see \cite[\S 2]{McM10}), all but finitely many periodic points land on $\mathcal{V}$.

There are only finitely many edges in $\mathcal{T}$ that are mapped to the edges landed by $\mathcal{C}$.
Thus by pulling back, there are only finitely many strictly pre-periodic points landing at edges, and they all come from backward orbits of $\mathcal{C}$.
The proposition now follows.
\end{proof}

A vertex $v\in \mathcal{V}$ is said to be {\em simple} if $\delta(v) = 1$, and is called {\em critical} if $\delta(v) \geq 2$.
The same proof also gives
\begin{prop}\label{prop:sl}
Let $v\in \mathcal{V}$. If $f^k(v)$ is simple for all $k\geq 0$, then there are only finitely many pre-periodic points landing at $v$.
\end{prop}

We immediately have the following, which will be used to construct the dual lamination later.
\begin{cor}\label{cor:am1}
There are at most $2$ pre-periodic points landing on an edge $E$. More precisely, at most $1$ from each side of the edge.
\end{cor}
\begin{proof}
Suppose for contradiction that there are 2 pre-periodic points landing on the same side of the edge $E$, then the ribbon structure will give infinitely many periodic points landing on $E$, contradicting Proposition \ref{prop:al}.
\end{proof}

\subsection*{Pullback of quasi-invariant tree}
Given the quasi-invariant tree $\mathcal{T}_n$ for $f_n$ modeled by $f:\mathcal{T} \longrightarrow \mathcal{T}$, we can construct the pullback quasi-invariant tree $\mathcal{T}^1_n$ as follows.

Let $v\in \mathcal{V}$. 
Choose an ordering of the $d$-preimages of $\phi_n(v)$ under $f_n$, and denote them by $w_{1,n}(v),..., w_{d,n}(v)$.
Define 
$$
\widetilde{\mathcal{P}}^1 :=\{ (w_{i,n}(v)): i = 1,..., d, v \in \mathcal{V}\}.
$$
Note that an element of $\widetilde{\mathcal{P}}^1$ is a sequence of points in $\D$ indexed by $n$ and it is a finite set.
After passing to a subsequence, we assume
$$
\lim_{n\to\infty} d_{\D} (w_{i,n}(v), w_{j,n}(u)) \text{ and } \lim_{n\to\infty} d_{\D} (w_{i,n}(v), \phi_n(u))
$$
exist for all $i,j \in \{1,..., d\}$ and for all $v, u \in \mathcal{V}$.
This defines an equivalence relation on $\widetilde{\mathcal{P}}^1$, and an equivalence class is called a {\em cluster set}.
A cluster set $[(w_n)]$ is said to be {\em new} if $\lim_{n\to\infty} d_{\D} (w_n, \phi_n(\mathcal{V})) = \infty$.
We choose a representative $(w_n)$ for each new cluster set.

Note that by construction, $(f_n(w_n)) = (\phi_n(v))$ for some $v\in \mathcal{V}$.
If $a_n \in \mathcal{T}_n$ are such that $d_{\D}(a_n, \phi_n(\mathcal{V})) \to \infty$, then $d_{\D}(f_n(a_n), \phi_n(\mathcal{V})) \to \infty$ by Proposition \ref{prop:qi}.
Thus, 
\begin{align}\label{eqn:nedge}
\lim_{n\to\infty} d_{\D} (w_n, \mathcal{T}_n) = \infty \text{ for any new cluster }[(w_n)].
\end{align}
So we can apply the same inductive method in the construction of $\mathcal{T}_n$ to add these new cluster sets, and get a new quasi-invariant tree $\mathcal{T}^1_n \supseteq \mathcal{T}_n$.

After passing to a subsequence, we may assume $\mathcal{T}^1_n$ are all isomorphic with the same ribbon structure.
We denote $\mathcal{T}^1$ as the underlying ribbon finite tree, with the isomorphisms
$$
\phi_n: \mathcal{T}^1 \longrightarrow \mathcal{T}^1_n.
$$
A simplicial model of the dynamics can be constructed for the pullback and is denoted by 
$$
f: (\mathcal{T}^1, \p)\longrightarrow (\mathcal{T}^0, \p) := (\mathcal{T}, \p)\subseteq (\mathcal{T}^1, \p).
$$ 

The rescaling limits are defined similarly and the same proof of Lemma \ref{lem:ct} shows the compatibility of the local dynamics with the tangent map.
We remark that the new vertices are all simple, so the rescaling limits are defined by degree $1$ maps.
Each vertex $v\in \mathcal{V}^0$ of $\mathcal{T}^0$ has $d$ preimages in $\mathcal{T}^1$ counted with multiplicity.
By Equation \ref{eqn:nedge}, if $E$ is an edge of $\mathcal{T}^0$, then $E$ is also an edge of $\mathcal{T}^1$, i.e., the inclusion map from $\mathcal{T}^0$ to $\mathcal{T}^1$ is simplicial.

We also remark that a priori, there are many choices in this construction.
They all give the same simplicial model $f: (\mathcal{T}^1, \p)\longrightarrow (\mathcal{T}^0, \p)$, as they are equal to the unique one constructed combinatorially in \S \ref{sec:atm}.

The pullback can be iterated, and we denote the $k$-th pullback simplicial model by 
$f: (\mathcal{T}^k, \p) \longrightarrow (\mathcal{T}^{k-1}, \p) \subseteq (\mathcal{T}^k, \p)$.
We also denote $f: (\mathcal{T}^\infty, \p) \longrightarrow (\mathcal{T}^\infty, \p)$ as the union of these maps.

\subsection*{Dual Lamination of $\mathcal{T}$}
Let $A_v \subseteq \Q/\Z$ be the set of rational angles landing at $v$.
Since $\eta_n : \R/\Z\longrightarrow \mathbb{S}^1$ is a homeomorphism, and $\mathcal{T}_n$ is a ribbon tree, for any $v \neq w \in \mathcal{V}$, $A_v$ and $A_w$ are {\em unlinked}, i.e., there exists an interval $I \subseteq \R/\Z$ so that $A_v \subseteq I$ and $A_w \subseteq \R/\Z - I$.

Two intervals $J_1, J_2 \subseteq \R/\Z$ are said to be essentially the same if $\Int(J_1) = \Int(J_2)$ and $\overline{J_1} = \overline{J_2}$.
Let $E = [v,w]$ be an edge.
We claim that there exists an essentially unique interval $I$ with $A_v \subseteq I $ and $A_w \subseteq \R/\Z - I$.
Otherwise, there are infinitely many rational angles that are separated by $A_v$ and $A_w$.
Since $[v,w]$ is an edge, these angles must land on $E$, which is a contradiction to Corollary \ref{cor:am1}.
We denote $\partial I = \{t^\pm_E\}$, and call $t^\pm_E$ the {\em dual angles} for $E$.
We remark that it is possible that $I$ is a degenerate interval consisting of a single point. In this case, $\partial I$ consists of a single point. 

An equivalent way to compute $t^\pm_E$ is as follows.
Let $x_E \in T_v \mathcal{T}$ be the tangent vector at $v$ associated to $E$.
We consider the set
$$
I_E(v):= \{t \in \R/\Z: \lim_v \eta_n(t) = \xi_v(x_E) \in \mathbb{S}^1_v\} \subseteq \R/\Z.
$$
Since $\eta_{n}$ is a homeomorphism, $I_E(v)$ is an interval. 
By Lemma \ref{lem:pland}, one can verify that $\partial I_E(v) = t^\pm_E$.
Note that if the other boundary point $w \in \partial E$ is chosen, then $I_E(w)$ is essentially the complement of $I_E(v)$ in $\R/\Z$, so the boundary points are the same.

\begin{figure}[ht]
  \centering
  \resizebox{0.8\linewidth}{!}{
    \def\svgwidth{\columnwidth}
    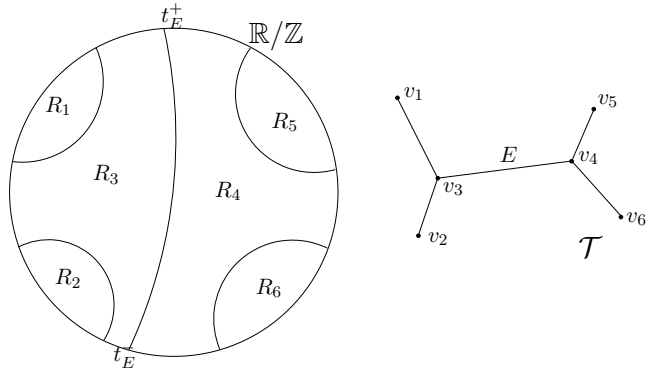

  }
  \caption{An illustration of the dual lamination. The rational angles in $\partial R_i \cap \R/\Z$ land at $v_i$. The dual angles $t_E^\pm$ for the edge $E$ are labeled on the figure.}
  \label{fig:DL}
\end{figure}

It can be verified that these angles are compatible with the dynamics:
\begin{prop}
Let $E$ be an edge of $\mathcal{T}$ with dual angles $t^\pm_E$.
Then $m_d(t^\pm_E)$ are the dual angels for $f(E)$.
\end{prop}

A {\em lamination} $\mathcal{L}$ is a family of disjoint hyperbolic geodesics in $\D$ together with the two endpoints in $\mathbb{S}^1\cong \R/\Z$, whose union $|\mathcal{L}| :=\bigcup\mathcal{L}$ is closed.
An element of the lamination is called a {\em leaf} of the lamination.

We define the leaf associated with the edge $E$ as the hyperbolic geodesics in $\D$ connecting $t^\pm_E \in \R/\Z \cong \partial \D$.
It is easy to check that leaves for different edges have disjoint interiors.

The {\em dual finite lamination} for $f: (\mathcal{T},\p) \longrightarrow (\mathcal{T}, \p)$ is defined as the union of all leaves for edges of $\mathcal{T}$, and is denoted by $\mathcal{L}^F_\mathcal{T}$.
The leaves can be constructed for edges of any pullbacks of $\mathcal{T}$.
We call the closure of the union of leaves for edges in $\mathcal{T}^\infty$ the {\em dual lamination} of $\mathcal{T}$, and is denoted by $\mathcal{L}_\mathcal{T}$.

We remark that our lamination $\mathcal{L}_\mathcal{T}$ is defined abstractly using dual angles in $\R/\Z$. The leaves of $\mathcal{L}_\mathcal{T}$ do not lie in the same disk $\D$ where $f_n \in \BP_d$ or $\mathcal{T}_n$ is defined.

The lamination $\mathcal{L}_\mathcal{T}$ gives an {\em equivalence relation} $\sim_\mathcal{T}$ on $\R/\Z$: $a\sim_\mathcal{T} b$ if there exists a finite chain of leaves connecting $a$ and $b$.
Note that different laminations may generate the same equivalence relation.

\subsection*{Periodic Fatou and Julia points}
For reasons that will appear apparent later in \S \ref{sec:BPH}, we define
\begin{defn}
Let $v\in \mathcal{V}$. It is called {\em a Fatou point} if it is eventually mapped to a critical periodic orbit, and is called a {\em Julia point} otherwise.
\end{defn}

Let $a,b\in \mathcal{V}$. We use $[a,b], (a,b), [a,b)$ and $(a,b]$ to denote the paths in $\mathcal{T}$ that connects $a,b$, with appropriate boundary points removed.
\begin{lem}\label{lem:h}
If $v\neq \p$ is a periodic point, and $[\p, v)$ contains a periodic Fatou point, then the first return rescaling limit $F$ at $v$ has an attracting fixed point on $\mathbb{S}^1_v$.
\end{lem}
\begin{proof}
After passing to an iterate, we may assume that $f$ fixes $[\p,v]$.
Let $t_p^v := \lim_v\phi_n(p)$. 
By Lemma \ref{lem:ct}, $t_p^v$ is a fixed point of the rescaling limit $F_v$.
Since $f_n$ fixes $\phi_n(p)$, by Schwarz lemma, $M_{v,n}^{-1} \circ f_n \circ M_{v,n}$ moves points towards $M_{v,n}^{-1}(\phi_n(p))$.
Since $M_{v,n}^{-1} \circ f_n \circ M_{v,n}$ converges to $F_v$ and $M_{v,n}^{-1}(\phi_n(p)) \to t_p^v$, the fixed point $t_p^v$ is non-repelling.

Suppose for contradiction that $t_p^v$ is parabolic. 
Then $d_{\D}(x, F(x))$ can be made arbitrarily small\footnote{Note that if $t_p$ is attracting, the hyperbolic distance $d_{\D}(x, F(x))$ is uniformly bounded away from $0$.} by making $x$ close to $t_p^v$. 
Therefore, for any $\epsilon> 0$, there exists $x_n$ close to $t_p^v$ in $v$-coordinate so that $d_{\D}(x_n, f_n(x_n)) < \epsilon$ for all sufficiently large $n$.

On the other hand, let $c$ be the fixed Fatou point on $[\p, v)$.
Let $F_c$ be the rescaling limit at $c$.
Note that $\deg F_c = \delta(c) \geq 2$, as $c$ is a Fatou point.

We claim that $t_v^c:=\lim_c\phi_n(v)$ is a repelling fixed point of $F_c$.
There are two cases. If $c = \p$, then $F_c = F_\p$ has a fixed point at $0$, which is necessarily attracting as $\deg F_c \geq 2$. Thus, $t_v^c$ is repelling.
If $c \neq \p$, then by the same argument as for $t_p^v$, the point $t_p^c:= \lim_c\phi_n(\p) \in \mathbb{S}^1_c$ is a non-repelling fixed point of $F_c$.
Since $t_v^c \neq t_p^c$, it is repelling. This proves the claim.

Therefore for $y\in \D_c$ near $t_v^c$, $d_{\D}(y, F_c(y))\geq K$ for some $K$ depending on the multiplier of the repelling fixed point $t_v^c$ under $F_c$ and the angle $\angle F_c(y)yt_v^c$ is greater or equal to $\pi/2$.
Thus, for all sufficiently large $n$, there exists $y_n$ close to $t_v^c$ in $c$-coordinate so that $d_{\D}(y_n, f_n(y_n)) \geq K$, and the angle $\angle f(y_n)y_nx_n$ is greater or equal to $\pi/2$.
By choosing $\epsilon$ small enough, we have $d_{\D}(f_n(x_n), f_n(y_n)) > d_{\D}(x_n, y_n)$ which is a contradiction to Schwarz lemma.
\end{proof}

\begin{cor}\label{cor:nb}
If $v\neq \p$ is a periodic Julia point, and $[\p, v)$ contains a periodic Fatou point, then $v$ is not a branch point of $\mathcal{T}$.
\end{cor}
\begin{proof}
Suppose for contradiction that $v$ is a branch point, then the first return rescaling limit $F$ at $v$ is the identity map as it has degree $1$ and fixes three points on the circle.
This is a contradiction as $F$ has an attracting fixed point on the circle by Lemma \ref{lem:h}.
\end{proof}

We call a periodic Fatou point $v$ {\em parabolic} or {\em boundary-hyperbolic} if the first return rescaling limit is parabolic or boundary-hyperbolic.
If $\p$ is critical, then $[\p, v)$ contains a periodic Fatou point for any $v\neq \p$, so we have
\begin{cor}
If $\delta(\p) \geq 2$, then 
\begin{itemize}
\item every periodic branch point is a Fatou point;
\item every periodic Fatou point other than $\p$ is boundary-hyperbolic.
\end{itemize}
\end{cor}

\section{Angled tree map}\label{sec:atm}
In this section, we introduce abstract angled tree maps that give combinatorial descriptions of the pointed quasi-invariant trees for $f_n\in \BP_d$.
The construction is similar to the angled Hubbard tree introduced in \cite{Poirier93} with two major differences:
\begin{itemize}
\item To work with quasi post-critically finite degenerations in $\BP_d$, the angled tree maps in our setting are simplicial and marked;
\item To capture the dynamics of parabolic or boundary-hyperbolic rescaling limits, the angles are allowed to be $0^\pm$.
\end{itemize}
We remark that unlike the angled Hubbard trees, the angle $0$ plays a special role in our setting.

Let $(\mathcal{T}, \p)$ be a pointed finite tree with a ribbon structure and vertex set $\mathcal{V}$.
Let $f: (\mathcal{T}, \p) \longrightarrow (\mathcal{T}, \p)$ be a simplicial map.
The tangent space at a vertex $v$ is identified with the set of incident edges to $v$ and is denoted by $T_v\mathcal{T}$.
We define the local degree function $\delta: \mathcal{V} \longrightarrow \Z_{\geq 1}$ which assigns an integer $\delta(v) \geq 1$ to each vertex $v\in \mathcal{V}$.
A vertex $v$ is said to be {\em critical} if $\delta(v) \geq 2$ and {\em simple} otherwise.
The degree of the map $f$ is defined by
$$
\deg(f):= 1+ \sum_{v\in \mathcal{V}} (\delta(v)-1),
$$
and we always assume $\deg(f) \geq 2$.
We also assume that $\mathcal{T}$ is {\em non-trivial}, i.e. it contains more than one point.
We say $f: (\mathcal{T}, \p) \longrightarrow (\mathcal{T}, \p)$ is {\em minimal} if $\mathcal{T}$ is the convex hull of critical orbits and $\mathcal{V}$ is the smallest set containing critical orbits such that $f$ is simplicial.

\subsection*{Angle structure on $\mathcal{T}$}
We identify $\mathbb{S}^1 = \R/\Z$, and $m_d: \mathbb{S}^1\longrightarrow \mathbb{S}^1$ is the multiplication by $d$ map.
By our convention, $m_1$ is the identity map.
For $d\geq 2$, $m_d$ gives a topological model of the dynamics on the Julia set of a degree $d$ hyperbolic and a doubly parabolic Blaschke product.

To set up a framework that also works for singly parabolic or boundary-hyperbolic Blaschke products uniformly, we consider an extended circle $\mathbb{S}^{1}_{d}$, which is naturally regarded as {\em cyclically ordered set} (see \cite[\S 2]{McM09}).
As a set, 
$\mathbb{S}^{1}_{d}$ is constructed from $\mathbb{S}^1$ by adding (formal symbols) $x^-, x^+$ for any point $x$ in the backward orbit of $0$ under $m_d$ for $d\geq 2$.
The cyclic ordering on $\mathbb{S}^{1}_{d}$ is defined so that $x^-$ (or $x^+$) is regarded as a point infinitesimally smaller than $x$ (or bigger than $x$ respectively) in the standard identification of $\mathbb{S}^1 = \R/\Z$.
We use the convention that $\mathbb{S}^1_1 = \mathbb{S}^1$.

Given any integer $k \geq 1$, the map $m_k$ naturally extends to
$m_k: \mathbb{S}^1_{d} \longrightarrow \mathbb{S}^1_{d}$ 
by setting
$m_k (x^\pm) = m_k(x)^\pm$.
This is well-defined as if $x$ is in the backward orbit of $0$ under $m_d$, $m_k(x)$ is also in the backward orbit of $0$ under $m_d$.

If $f$ is a degree $d$ boundary-hyperbolic Blaschke product, i.e., $f$ has an attracting fixed point $a$ on the circle, 
the Julia set $J$ of $f$ is a Cantor set on $\mathbb{S}^1$.
The complement $\mathbb{S}^1 - J$ consists of countably many intervals, which are all eventually mapped to the unique interval $I \subseteq \mathbb{S}^1$ that contains the attracting fixed point $a$. The boundary $\partial I$ consists of two repelling fixed points of $f$.
Let $\mathcal{O}(a)$ be the backward orbit of the attracting fixed point $a$.
Then there exists bijective map $\eta_f: \mathbb{S}^{1}_{d} \longrightarrow J(f) \cup \mathcal{O}(a)$ which preserves the cyclic ordering so that
$$
f \circ \eta_f = \eta_f \circ m_d.
$$
Note that $\eta_f(0) = a$, and $\eta_f(0^\pm) = \partial I$.

\begin{figure}[ht]
  \centering
  \resizebox{0.45\linewidth}{!}{
    \def\svgwidth{\columnwidth}
    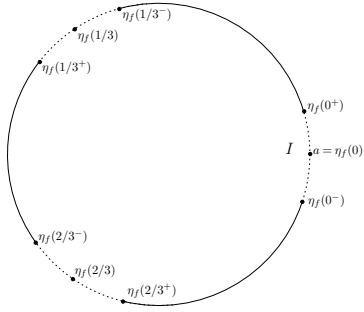

  }
  \caption{An illustration of the conjugacy $\eta_f$ for a degree $3$ boundary-hyperbolic Blaschke product. The Julia set is a Cantor set, constructed by removing the backward orbits of the interval $I$.}
  \label{fig:BHM}
\end{figure}

To model the dynamics of the pullbacks, we construct $\mathbb{S}^{1}_{d,D}$ by adding $x^-, x^+$ to $\mathbb{S}^1$ if $m_D(x)$ is in the backward orbit of $0$ under $m_d$, and the cyclic ordering is constructed in the same way.
Note that by this construction, $m_D:\mathbb{S}^{1}_{d,DD'} \longrightarrow \mathbb{S}^{1}_{d,D'}$ is a degree $D$ covering between cyclically ordered sets (see \cite[\S 2]{McM09} for detailed definitions).
Note that $\mathbb{S}^{1}_{d,1} = \mathbb{S}^{1}_{d}$.
We remark that the intervals can be defined naturally for $\mathbb{S}^{1}_{d,D}$, and we denote them by $[a,b], (a,b], [a,b), (a,b)$ with appropriate boundary points removed.

If $v \in \mathcal{V}$ has pre-period $l$ and period $q$.
We define the {\em cumulative degree}
$$
\Delta(v) = \delta(f^l(v))\delta(f^{l+1}(v))...\delta(f^{l+q-1}(v)),
$$
and the {\em cumulative pre-periodic degree}
$$
\Delta_{pre}(v) = \delta(v)\delta(f(v))...\delta(f^{l-1}(v)).
$$
We use the convention that $\Delta_{pre}(v) = 1$ for all periodic vertices.

We attach an extended circle $\mathbb{S}^1_{\Delta(v), \Delta_{pre}(v)}$ to a vertex $v \in \mathcal{V}$. 
We define an {\em angle function} $\alpha$ at $v$ as an injective map
$$
\alpha_v: T_v\mathcal{T}\hookrightarrow \mathbb{S}^1_{\Delta(v), \Delta_{pre}(v)}.
$$
We say $\alpha$ is {\em regular} at $v$ if $\alpha_v(T_v\mathcal{T}) \subseteq \mathbb{S}^1 \subseteq \mathbb{S}^1_{\Delta(v), \Delta_{pre}(v)}$.

\begin{defn}\label{defn:comp}
We say an angle function $\alpha$ is {\em compatible} if for any $v \in \mathcal{V}$,
\begin{enumerate}
\item $\alpha_v$ is {\em cyclically compatible}: if $x_1, x_2, x_3 \in T_v\mathcal{T}$ are clockwise oriented, then $\alpha_v(x_1), \alpha_v(x_2), \alpha_v(x_3)$ are also clockwise oriented; 
\item $\alpha_v$ is {\em dynamically compatible}:
\begin{itemize}
\item if $v= \p$ and $\delta(\p) = 1$, then there exists a rigid rotation $R$, which is necessarily a rational rotation, so that
$
R \circ \alpha_{\p} = \alpha_{\p} \circ Df|_{T_\p\mathcal{T}}
$;
\item otherwise, 
$
m_{\delta(v)} \circ \alpha_v = \alpha_{f(v)} \circ Df|_{T_v\mathcal{T}}
$;
\end{itemize}

\item $\alpha_v$ is {\em $\p$-compatible}: 
if $v \neq \p$ is periodic and $x \in T_v\mathcal{T}$ is the tangent vector in the direction of $\p$, then $\alpha_v(x) = 0$.
\end{enumerate}
\end{defn}

We remark that if $v\neq \p$ is a periodic point of period $q$, and $x\in T_v\mathcal{T}$ is in the direction of $\p$, then $D_vf^q(x) = x$ as $f$ is simplicial.
Thus condition (3) is compatible with the dynamics.

\begin{defn}
An {\em angled tree map} is a triple 
$$
(f: (\mathcal{T}, \p) \rightarrow (\mathcal{T}, \p), \delta, \alpha = \{\alpha_v\})
$$ 
of a simplicial map on a pointed finite ribbon tree together with a local degree function $\delta$ and a compatible angle function $\alpha$ which is regular at $p$.
\end{defn}

We shall use $f: (\mathcal{T}, \p) \rightarrow (\mathcal{T}, \p)$ or simply $\mathcal{T}$ to denote an angled tree map if the dynamics, local degree function and angle function are not ambiguous.

\subsection*{Pullback of an angled tree map}
Given an angled tree map $f: (\mathcal{T}, \p) \rightarrow (\mathcal{T}, \p)$, one can naturally construct a new angled tree map by pulling back by the dynamics.
More precisely, take $v\in \mathcal{V}$ and a tangent direction $a\in T_{f(v)}\mathcal{T}$.
Let $S_a$ be the component of $\mathcal{T}-\{f(v)\}$ corresponding to the direction $a\in T_{f(v)}\mathcal{T}$.
Let $B \subseteq \mathbb{S}^1_{\Delta(v), \Delta_{pre}(v)}$ be the preimage of $\alpha(a) \in \mathbb{S}^1_{\Delta(f(v)),\Delta_{pre}(f(v))}$ under the corresponding map from $\mathbb{S}^1_{\Delta(v), \Delta_{pre}(v)}$ to $\mathbb{S}^1_{\Delta(f(v)),\Delta_{pre}(f(v))}$.
We attach a copy of $S_a$ at every point in $B- \alpha(T_v\mathcal{T})$.
The dynamics extend naturally to the new copies by identifications, and the angular structures are defined by pulling back with the identity map $m_1$.

Let $\mathcal{T}^1$ be the angled tree constructed from $\mathcal{T}$ by running the above algorithm for all vertices $v$ and all tangent directions $a\in T_w\mathcal{T}$.
We call 
$$
f:(\mathcal{T}^1, \p) \longrightarrow (\mathcal{T}^0, \p) := (\mathcal{T}, \p)\subseteq (\mathcal{T}^1, \p)
$$ 
the {\em (first) pullback} of $f: (\mathcal{T}, \p) \longrightarrow (\mathcal{T}, \p)$.

Note that each vertex $v \in \mathcal{T}^0$ has exactly $d$ preimages in $\mathcal{T}^1$ counted with multiplicity, and the inclusion map $i: \mathcal{T}^0 \longrightarrow \mathcal{T}^1$ is simplicial.
Also note that the map $f: (\mathcal{T}^1, \p) \longrightarrow (\mathcal{T}^1, \p)$ is no longer minimal.

This construction can be iterated.
We denote the $k$-th pullback by 
$f:(\mathcal{T}^k, \p) \longrightarrow (\mathcal{T}^{k}, \p)$, and the union of the pullbacks by $f:(\mathcal{T}^\infty, \p) \longrightarrow (\mathcal{T}^\infty, \p)$.

\subsection*{External angles, markings and anchored conditions}
Similar to the abstract Hubbard trees (see \cite[Chapter III \S 4]{Poirier93}), external angles can be defined using the dynamics on $\mathcal{T}$.
They can be described as follows.

Let $d = \deg(f)$ be the degree of $f: \mathcal{T} \longrightarrow \mathcal{T}$.
Let $\mathcal{T}^\infty$ be the union of pullbacks of $\mathcal{T}$.
Let $\epsilon(\mathcal{T}^\infty)$ be the set of ends of $\mathcal{T}^\infty$.
Note that the ribbon structure of $\mathcal{T}^\infty$ makes $\epsilon(\mathcal{T}^\infty)$ a cyclically ordered set.
Since $\mathcal{T}$ is non-trivial, the set of ends $\epsilon(\mathcal{T}^\infty)$ is infinite.
The simplicial map $f: \mathcal{T}^\infty \longrightarrow \mathcal{T}^\infty$ induces a map $f_*: \epsilon(\mathcal{T}^\infty) \longrightarrow \epsilon(\mathcal{T}^\infty)$ (cf. \cite[\S 3]{McM09}).
There is a natural cyclical order-preserving map
$$
\phi: \epsilon(\mathcal{T}^\infty) \longrightarrow \mathbb{S}^1 = \R/\Z
$$
that transports the dynamics of $f_*$ on $\epsilon(\mathcal{T}^\infty)$ to the dynamics of $m_d$ on $\R/\Z$ (cf. \cite[\S 6]{McM09}).
We call $\phi$ an {\em external angle marking}, or simply a {\em marking}.
Since the image $\phi(\epsilon(\mathcal{T}^\infty))$ is invariant under $m_d^{-1}$, it is necessarily dense.
We remark that there are $d-1$ choices of markings, and any two markings are related by post-composition with an element of the automorphism group $\Z/(d-1)$ of $m_d$.
An angled tree map together with a marking is called a {\em marked angled tree map}.

Note that the $\p$-compatible condition in Definition \ref{defn:comp} gives the unique normalization of angle functions at a periodic vertex $v \neq \p$.
Given a marking $\phi$ on an angled tree map $\mathcal{T}$, we can normalize the angle function at all other vertices with the following anchored convention, which is an analogue of Definition \ref{defn:anchored} and Definition \ref{defn:mcpc}.

We remark that $0$ may not be in the image $\phi(\epsilon(\mathcal{T}^\infty))$.
But since $\phi(\epsilon(\mathcal{T}^\infty))$ is dense in $\R/\Z$, there exists a sequence $\gamma_n \in \epsilon(\mathcal{T}^\infty)$ with $\phi(\gamma_n) \to 0^-$, i.e., $\phi(\gamma_n)$ approaches $0$ from below.
We represent each end $\gamma_n$ by a path in $\mathcal{T}^\infty$ that starts at $p$.
Since $\alpha$ is regular at $p$, so $\alpha_p$ extends to
$$
\alpha_\p: T_\p \mathcal{T}^\infty \longrightarrow \mathbb{S}^1 \subseteq \mathbb{S}^1_{\Delta(p)}.
$$
Since each path $\gamma_n$ gives a tangent vector $x_n \in T_\p \mathcal{T}^\infty$, we can associate an internal angle $t_n := \alpha_\p(x_n)$ to $\gamma_n$.
With the standard topology on $\mathbb{S}^1$, one can verify that the limit of $t_n$ exists and is independent of the sequence $\gamma_n$ that we choose.
We call $t_p:= \lim t_n \in \mathbb{S}^1$ the internal angle at $p$ with respect to the marking (see Figure \ref{fig:PullBack}).

\begin{figure}[ht]
  \centering
  \resizebox{0.8\linewidth}{!}{
    \def\svgwidth{\columnwidth}
    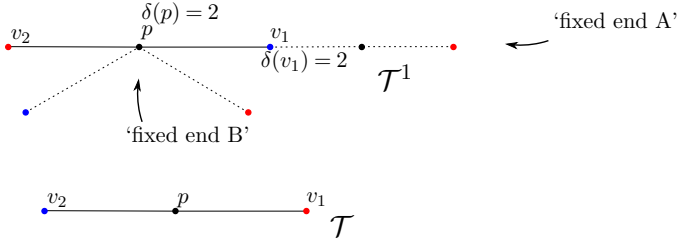

  }
  \caption{Vertices in $\mathcal{T}^1$ is sent to vertices in $\mathcal{T} = \mathcal{T}^0$ with the same color. This degree $3$ angled tree map has $2$ markings $\phi_A, \phi_B$. The corresponding ends of $0$-external angle are illustrated by the two arrows. Note that $0\in \R/\Z$ is in the image of $\phi_A$ but not in the image of $\phi_B$. The corresponding internal angles at $p$ with respect to $\phi_A$ and $\phi_B$ are $t_{p,A} = 1/3$ and $t_{p, B} = 0$.
   }
  \label{fig:PullBack}
\end{figure}

\begin{defn}\label{defn:angt}
Let $(f: (\mathcal{T}, \p) \rightarrow (\mathcal{T}, \p), \delta, \alpha)$ be a marked angled tree map.
Let $t_p$ be the internal angle at $p$ with respect to the marking.
It is said to be {\em anchored} if the followings hold:
\begin{itemize}
\item if $\delta(\p) = 1$, then $t_p = 0$;
\item if $\delta(\p) \geq 2$, then $t_p \in (-\frac{1}{\Delta(\p)-1},0]$.
\item if $v \neq \p$ is strictly pre-periodic and $x \in T_v\mathcal{T}$ is the tangent vector in the direction of $\p$, then $\alpha_v(x) \in (-\frac{1}{\delta(v)},0]$.
\end{itemize}
\end{defn}

We remark that by post-composing the angled functions with rotations that keep the compatibility condition, one can always make the angled tree map anchored.
Throughout this paper, we shall always assume that an angled tree map is marked and the angle functions are anchored.

\subsection*{Realizing angled tree maps}
Let $f_n\in \BP_d$ be a quasi post-critically finite sequence. 
By the discussion in \S \ref{sec: gfb}, we have an induced (marked) simplicial map $f: (\mathcal{T}, \p) \longrightarrow (\mathcal{T}, \p)$ with a local degree function $\delta$ modeling the dynamics of the critical orbits.
By construction, $f$ is minimal. 
The cumulative degree and cumulative pre-periodic degree for $v$ are defined similarly. 

To define the angle functions, we assign the angle $0 \in \mathbb{S}^1_{\Delta(v), \Delta_{pre}(v)}$ to the marked point $t_v \in \mathbb{S}^1_v$ (see Definition \ref{defn:mcpc}).
The angles are then determined by the conjugacy of the tangent map with the rescaling limit (see Lemma \ref{lem:ct}).
We remark that the angles assigned at periodic Julia points $v$ are rather artificial.
We use the convention to assign the angles $\frac{i}{\nu}, i=0,..., \nu-1$ according to their cyclic order where $\nu$ is the valence at $v$.

The above construction gives an angled tree map.
An angled tree map $(f: (\mathcal{T}, \p) \rightarrow (\mathcal{T}, \p), \delta, \alpha)$ that arises in this way is said to be {\em realized} by $f_n \in \BP_d$.
Let $\mathcal{T}_n$ be the sequence of quasi-invariant trees for $f_n$.
Since we will be working with pullbacks, we shall say $\mathcal{T}_n$ {\em realizes} the angled tree map $\mathcal{T}$ or the angled tree map $\mathcal{T}$ is {\em realized} by $\mathcal{T}_n$ when the underlying maps are not ambiguous.
We can associate an angled tree map for the $k$-th pullback $\mathcal{T}^k_n$ of the quasi-invariant tree $\mathcal{T}_n$.
One can verify that this angled tree map is the $k$-th pullback of the angled tree map $\mathcal{T}$.
We shall say $\mathcal{T}^k_n$ realizes $\mathcal{T}^k$.

\subsection*{Admissible angled tree maps}
Let $f:(\mathcal{T}, \p) \longrightarrow (\mathcal{T}, \p)$ be an angled tree map, we now describe a sufficient condition for realization. 
A periodic vertex $v$ is said to be {\em attached to $p$} if $[\p, v)$ contains no Fatou point.
Here $[\p,v)$ is the path in $\mathcal{T}$ that connects $\p$ and $v$ with the boundary point $v$ removed.
The {\em core} $\mathcal{T}^{C}$ of an angled tree map is defined as the convex hull of all periodic vertices attached to $p$.
Since $f$ is simplicial, any vertex in $\mathcal{T}^C$ is periodic.
Note that if $\delta(\p) \geq 2$, then $\mathcal{T}^C = \{\p\}$.

A vertex is said to be an end if the valence at $v$ is $1$.
We say a tree $T$ is {\em star-shaped} if there exists a unique vertex in $T$ that is not an end.
The unique vertex is called the {\em center} for the star-shaped tree.
For $k\geq 2$, we call a star-shaped tree with $k$ ends a {\em $k$-star}.
By our definition, a tree with a single vertex is also star-shaped with no end.

\begin{defn}
The core $\mathcal{T}^C$ is said to be {\em critically star-shaped} if
\begin{itemize}
\item $\mathcal{T}^C$ is star-shaped with center $\p$;
\item every end of $\mathcal{T}^C$ is a periodic Fatou point;
\item for any vertex $v \in \mathcal{T}^C$, the angle function $\alpha_v$ is regular at $v$.
\end{itemize}
\end{defn}

If $\delta(\p) \geq 2$, then $\mathcal{T}^C = \{\p\}$ and the conditions are trivially satisfied.
If $\delta(\p) =1$, these conditions give a way to `normalize' the dynamics at $\p$.

\begin{defn}\label{defn:adm}
An angled tree map $f:(\mathcal{T},\p) \longrightarrow (\mathcal{T}, \p)$ is said to be {\em admissible} if
\begin{itemize}
\item the core $\mathcal{T}^C$ is critically star-shaped;
\item every periodic branch point other than $\p$ is a Fatou point.
\end{itemize}
\end{defn}

Assume $\mathcal{T}^C$ is critically star-shaped, by Lemma \ref{cor:nb}, the second condition is necessary for realization.
Thus for $\delta(\p) \geq 2$, the admissible condition is a necessary condition for realization.
In \S \ref{sec:BPH}, we introduce a notion of admissible splitting of a simplicial pointed Hubbard tree.
We shall see that this produces an admissible angled tree map (see Proposition \ref{prop:eqaht}).

\section{Realizing admissible angled tree map}\label{sec:raatm}
In this section, we shall prove the following theorem:
\begin{theorem}\label{thm:ar}
If a minimal angled tree map $f:(\mathcal{T},\p) \longrightarrow (\mathcal{T}, \p)$ is admissible, then it is realizable.
\end{theorem}

The proof is by induction.
We start by showing any degree $2$ admissible angled tree map is realizable.
We then show we can build a degree $d+1$ admissible angled tree map from a degree $d$ one.
By the induction hypothesis, the degree $d$ admissible angled tree map is realizable by a sequence of $g_n \in \BP_d$.
The induction step is proved by carefully adding a zero of the Blaschke product. 
More precisely, we construct a sequence $f_n = e^{i\theta_n}\frac{z-a_n}{1-\overline{a_n} z}\cdot g_n$ where $\theta_n$ and $a_n$ are chosen carefully, and we show the corresponding angled tree map is the desired one.

For our purposes, it is convenient to define the distance $d_\mathcal{T}(v,w)$ between two vertices in $\mathcal{T}$ as the number of edges in the shortest path between $v$ and $w$.
Since $f$ is simplicial, $f$ is distance non-increasing with respect to this metric $d_\mathcal{T}$.

\subsection*{More precise statement}
A degree $d$ proper holomorphic map $f: \D\longrightarrow \D$ is said to be 
\begin{itemize}
\item {\bf \em 1-anchored} if $f(1) = 1$;
\item {\bf \em fixed point centered} if $f(0) = 0$;
\item {\bf \em zeros centered} if the sum
$$
a_1+... + a_d
$$
of the points of $f^{-1}(0)$ (counted with multiplicity) is equal to $0$;
\item {\bf \em $M$-uni-critical} if the critical points are all within $M$ hyperbolic distance from $0$.
\end{itemize}

Take $f\in \BP_d$.
By conjugating with a rotation fixing $0$ that sends $\eta_f(0)$ to $1$, we can canonically identify $f$ with a 1-anchored, fixed point centered Blaschke product.
The following definition can be found in \cite[\S 4]{Milnor12}:
\begin{defn}\label{defn:nms}
Let $\mathcal{S} \subseteq \mathcal{V}$ be an invariant subset. 
We associate a copy of the disk $\D_v$ for each $v\in \mathcal{S}$, and we define the {\em normalized mapping scheme} $\mathcal{F}$ for $\mathcal{S}$ as a collection of 1-anchored proper holomorphic maps 
$$
\mathcal{F}_v: \D_v\longrightarrow \D_{f(v)}
$$ 
of degree $\delta(v)$, which are either fixed point centered or zeros centered according as $v$ is periodic or strictly pre-periodic under $f: \mathcal{S} \rightarrow \mathcal{S}$.
\end{defn}

Denote $\mathcal{S}_\p \subseteq \mathcal{V}$ as the backward orbit of $\p$.
If $\delta(\p) \geq 2$, the rescaling limits give a normalized mapping scheme on $\mathcal{S}_\p$.
Since the rescaling limits are marked by the anchored convention (see Definition \ref{defn:mcpc}),
we say conjugacies $\phi_v$ between the rescaling limits and mapping schemes have {\em compatible markings} if $\phi_v(t_v) = 1$.
We prove the following more precise and technical statement which immediately implies Theorem \ref{thm:ar}:
\begin{prop}\label{prop:rt}
Let $f:(\mathcal{T},\p) \longrightarrow (\mathcal{T}, \p)$ be a minimal admissible angled tree map of degree $d$.
In the case when $\delta(\p) \geq 2$, we let $\mathcal{F}$ be a normalized mapping scheme on $\mathcal{S}_\p$.
There exists a $K$-quasi post-critically finite sequence $f_n \in \BP_d$ realizing the angled tree map such that
\begin{enumerate}
\item the rescaling limits on $\mathcal{S}_\p$ are conjugate to $\mathcal{F}$ with compatible markings;
\item there exists a constant $M$ depending only on the tree map $f$ (and thus independent of $\mathcal{F}$) so that 
\begin{enumerate}
\item[(a)] for any periodic cycle $\mathcal{C}$ other than $\p$, there exists a periodic point $v \in \mathcal{C}$ so that the first return rescaling limit 
$$
F^q_v: \D_v\longrightarrow \D_v
$$ 
is $M$-uni-critical, where $q$ is the period of $v$;
\item[(b)] for any strictly pre-periodic vertex $w \notin \mathcal{S}_\p$, the rescaling limit $F_w: \D_w \longrightarrow \D_{f(w)}$ is $M$-uni-critical and the critical values are within $M$ hyperbolic distance from $0 \in \D_{f(w)}$.
\end{enumerate}
\end{enumerate}
\end{prop}

Note that condition (1) is vacuously satisfied when $\delta(p) = 1$.
We remark that 
as we degenerate $\mathcal{F}$, $K$ may go to infinity.
Since $M$ is independent of $\mathcal{F}$, the last condition says that nevertheless, the critical points for each vertex $v$ stay in a uniformly bounded distance.
Thus, the first return rescaling limit at $v$ is almost uni-critical.
This provides compactness that will be used in the successive degenerations (see Lemma \ref{lem:sep} and Proposition \ref{lem:Muc}).

We also remark that the $M$-uni-critical condition may not hold for all periodic points in the periodic cycle:
if there are two or more critical vertices in the periodic cycle, the fact that the first return map is $M$-uni-critical at one periodic vertex does not mean the first return map is $M$-uni-critical at other periodic vertices in this cycle.

\subsection*{Estimates in hyperbolic geometry}
We start with some useful estimates that will be used frequently in the proof.
The following lemma follows from the Schwarz lemma and Koebe distortion theorem.
\begin{lem}\label{lem:hme}
Let $U$ be a simply connected domain in $\C$ with hyperbolic metric $\rho_U |dz|$. Then for $z\in U$,
$$
\frac{1}{2d_{\R^2}(z, \partial U)} \leq \rho_U(z) \leq \frac{2}{d_{\R^2}(z, \partial U)}.
$$
\end{lem}

The following fact in hyperbolic geometry is very useful.
\begin{lem}\label{lem:he}
Let $z, w\in \D$ with $|z| = 1-\delta$ and $|w| = 1- \delta^s$.\\
If $0\leq s < 1$, then
$$
sd_{\D}(0,z) \leq d_{\D}(0, w) \leq sd_{\D}(0,z) + \log 2.
$$
If $s> 1$, then
$$
s(d_{\D}(0,z) -\log 2) \leq d_{\D}(0, w) \leq sd_{\D}(0,z).
$$
\end{lem}
\begin{proof}
Note that $d_{\D}(0,z)=\log\frac{2-\delta}{\delta}$ and $d_{\D}(0,w)=\log\frac{2-\delta^s}{\delta^s}$. So 
$$
d_{\D}(0, w)- sd_{\D}(0, z) = \log (2-\delta^s) - \log (2-\delta)^s.
$$

If $s<1$, then $(2-\delta)^s \leq 2 - \delta^s$, so $d_{\D}(0, w)- sd_{\D}(0, z) \in [0,\log2]$.

If $s>1$, then $(2-\delta)^s \geq 2 - \delta^s$, so $d_{\D}(0, w)- sd_{\D}(0, z) \in [-s\log 2, 0]$.
\end{proof}

Given $a\in \D$, we denote $\hat{a} = \frac{a}{|a|} \in \mathbb{S}^1$, $\delta_a = 1-|a|$ and $\rho_a= d_{\D}(0, a)$.
\begin{lem}\label{lem:em}
Let $M_a(z) = \frac{z-a}{1-\bar{a}{z}}$. Then for $0<s<1$,
$$
|M_a(z) + \hat{a}| \leq 2\delta_a^{1-s}
$$
for all $z\in B(0, 1-\delta_a^s)\supseteq B_{\Hyp^2}(0, s\rho_a)$, where $B_{\Hyp^2}(0, s\rho_a)$ is the hyperbolic ball centered at $0$ with radius $s\rho_a$.
\end{lem}
\begin{proof}
Note that
\begin{align*}
|M_a(z) + \hat{a}|
&=|\frac{z-a+\hat{a}-\hat{a}\bar{a}z}{1-\bar{a}{z}}| =|\frac{\delta_az+\hat{a}\delta_a}{1-\bar{a}{z}}|.
\end{align*}
Since $|z|<1$ and $|\hat{a}|=1$, it follows that $|\delta_az+\hat{a}\delta_a| \leq 2\delta_a$.
Since $|z| \leq 1-\delta_a^s$, we have $|1-\bar{a}{z}| \geq \delta_a^s$.
Thus,
$|M_a(z) + \hat{a}| \leq \frac{2\delta_a}{\delta_a^s} = 2\delta_a^{1-s}$.
\end{proof}

More generally, we have the following estimate in terms of hyperbolic geometry:
\begin{lem}\label{lem:aem}
Let $\theta > 0$. There exists a constant $C = C(\theta)$ so that the following holds.
Let $M_a(z) = \frac{z-a}{1-\bar{a}{z}}$. 
Let $z\in \D$ so that the angle $\angle 0za$ between the hyperbolic geodesic segments $[0,z]$ and $[z, a]$ satisfies $\angle 0za \geq \theta$.
Then
$$
|M_a(z) + \hat{a}| \leq C\delta,
$$
where $\rho := d_{\D}(a,z)$ and $\delta$ is the positive number so that the hyperbolic ball $B_{\Hyp^2}(0, \rho)$ has Euclidean radius $1-\delta$.
\end{lem}
\begin{proof}
Since the angle between $[a,z]$ and $[z, 0]$ is bounded below by $\theta$, there exists a constant $C_1 = C_1(\theta)$ so that $\rho_a \geq \rho_z + \rho - C_1$, where $\rho_a = d_{\D}(0,a)$ and $\rho_z = d_{\D}(0,z)$.
Thus, there exists a constant $C= C(\theta)$ so that $\delta_a \leq \frac{C}{2} \delta_z \delta$, where $\delta_a = 1-|a|$ and $\delta_z = 1-|z|$.
By the same computation as in Lemma \ref{lem:em}, we have
$|M_a(z) + \hat{a}| \leq \frac{2\delta_a}{\delta_z} \leq C \delta$.
\end{proof}

We shall also use the following estimate for critical points:
\begin{lem}\label{lem:cl}
Let $C, \theta > 0$. 
Let $x_1,..., x_k, z\in \D$ be $k+1$ points with angles $\angle x_i z x_j \geq \theta$ for any pairs $i,j$.
Let $f: \D\longrightarrow \D$ be a proper map of degree $d$ such that $f(z) = a$, $f(x_i) = b$ for all $i$. Suppose that
$$
d_{\D}(z, x_i) \leq d_{\D}(a,b) + C.
$$
Then there exists a constant $R = R(C, \theta, d)$ such that there are $k-1$ critical points (counted with multiplicity) of $f$ in $B_{\Hyp^2}(z, R)$.
\end{lem}
\begin{proof}
Let $L = d_{\D}(a,b)$, and $z_{i,t}$ and $a_t$ be the points on the geodesic segments $[z, x_i]$ and $[a,b]$ with $d_{\D}(z, z_{i,t}) = d_{\D}(a, a_t) = t$.
Since $d_{\D}(z_{i,t}, x_i) \leq L+C-t$, by Schwarz lemma, we have 
$$
f(z_{i,t})\in B_{\Hyp^2}(a, t) \cap B_{\Hyp^2}(b, L+C-t).
$$
Thus, there exists a constant $R_1 = R_1(C)$ so that $d_{\D}(f(z_{i,t}), a_t) \leq R_1$.
Thus there exists a constant $R_2 = R_2(R_1,d)$ so that 
$d_{\D}(z_{i,t}, f^{-1}(a_t)) \leq R_2$ (see \cite[Corollary 10.3]{McM09}).
Since the angles $\angle x_i z x_j$ are bounded from below, there exists a constant $R_3 = R_3(\theta)$ so that the balls $B_{\Hyp^2}(z_{i,t}, R_2)$ are disjoint for all $t\geq R_3$.
Thus there are $k$ different preimages of $a_{R_3}$ in the ball $B_{\Hyp^2}(z, R_2+R_3)$.
We have $f(B_{\Hyp^2}(z, R_2+R_3)) \subseteq B_{\Hyp^2}(a, R_2+R_3)$ by Schwarz Lemma.
Let $U$ be the component of $f^{-1}(B_{\Hyp^2}(a, R_2+R_3))$ that contains $z$.
Since $f:\D \longrightarrow \D$ is proper, $U$ is simply connected.
By Theorem \ref{thm:almostisometry}, there exists a constant $R = R(R_2, R_3, d)=R(C, \theta, d)$ so that $U \subseteq B_{\Hyp^2}(z, R)$. 
The degree of $f: U \longrightarrow B_{\Hyp^2}(a, R_2+R_3)$ is at least $k$, so there are at least $k-1$ critical points in $U \subseteq B_{\Hyp^2}(z, R)$ proving the lemma.
\end{proof}

The following uniform separation property will also be used:
\begin{lem}\label{lem:sep}
Let $f:\D\longrightarrow \D$ be a proper map of degree $d$ that is $M$-uni-critical.
Then there exists $\theta = \theta(M) >0$ so that for any $t \in \mathbb{S}^1$, any two points in $f^{-1}(t)$ are separated by $\theta$ in the standard Euclidean metric on $\mathbb{S}^1$.
\end{lem}
\begin{proof}
Since post-composing by $\Isom(\D)$ does not change the statement of the lemma, we may assume $f\in \BP_d$.
Since the critical set is contained in $B_{\Hyp^2}(0, M)$, $f$ lives in a compact set of $\BP_d$, and the conclusion follows.
\end{proof}

\subsection*{Realization for degree 2 angled tree maps}
Let $f:(\mathcal{T}, \p) \longrightarrow (\mathcal{T}, \p)$ be a minimal admissible angled tree map. 
If $\delta(\p) = 2$, then $\mathcal{T} = \{\p\}$ and we can simply take the constant sequence $f_n(z) = \mathcal{F}_{\p}(z)$.
Therefore we may assume $\delta(\p) = 1$.
By the compatibility condition, there exists a rigid rotation $R$ so that $R\circ \alpha_{\p} = \alpha_{\p} \circ Df|_{T_v\mathcal{T}}$.
Note that here $R$ is necessarily a rational rotation.
By the admissible condition, the angled tree map is uniquely determined by this rotation number: 
$$
\mathcal{T} = \cup_{k=0}^{q-1} f^k([\p,c])
$$ 
where $c$ is the unique critical vertex of $\mathcal{T}$.
The realization for such angled tree maps has been established in \cite[\S 14]{McM09} using `strong convergence'.
We give a different proof here, as the induction step uses the same idea.

\begin{lem}\label{lem:r2}
Proposition \ref{prop:rt} holds in degree $2$.
\end{lem}
\begin{proof}
By the above discussion in the previous paragraph, we may assume that $\delta(\p) = 1$ with rotation number $r/q$.
Thus conditions $(1)$ and $(2)$ in Proposition \ref{prop:rt} are trivially satisfied, and it suffices to construct a quasi post-critically finite $f_n \in \BP_2$ realizing $\mathcal{T}$ of rotation number $r/q$.

Let $R(z) = e^{i2\pi r/q}z$ be the rigid rotation with rotation number $r/q$.
Let $M_n(z) = \frac{z-a_n}{1-\overline{a_n}z}$ where $a_n \in \R^+$ and $d_{\D}(0, a_n) = n$.
Note $M_n(z) \to -z$ algebraically.

We consider the sequence $f_n(z) = -R(z)M_n(z)$ which is conjugate to a map in $\BP_2$ by some rotation.
Abusing the notations, we shall not distinguish this difference and show that $f_n$ is quasi post-critically finite and realizes angled tree map with rotation $r/q$.

Let $b_n \in \R^+$ with $d_{\D}(0, b_n) = n/2$ be the hyperbolic midpoint of $[0, a_n]$.
Since the zeros of $f_n$ are $0$ and $a_n$, by Lemma \ref{lem:cl}, there exists some constant $K_1$ such that
\begin{align}\label{eqn:bc}
d_{\D}(b_n,c_n) \leq K_1
\end{align}
where $c_n$ is the critical point of $f_n$.
Since $d_{\D}(0, b_n) = \frac{1}{2}d_{\D}(0, a_n)$, we have $1-|b_n| = 1-|R(b_n)| \geq \delta_{a_n}^{1/2}$ by Lemma \ref{lem:he}.
Thus, by Lemma \ref{lem:em}, the error term $|f_n(b_n)-R(b_n)|$ satisfies
\begin{align}\label{eqn:error}
|f_n(b_n)-R(b_n)| = |R(b_n)||M_n(b_n)+1| \leq 2\delta_{a_n}^{1/2}
\end{align}
where $\delta_{a_n} = 1-|a_n|$.

By Schwarz lemma, $d_{\D}(0, f_n(b_n)) \leq n/2$, so $1-|f_n(b_n)| \geq \delta_{a_n}^{1/2}$ as well. 
Therefore, by Lemma \ref{lem:hme}, the hyperbolic metric along the Euclidean segment $[f_n(b_n), R(b_n)]_{\R^2}$ is bounded above by 
\begin{align}\label{eqn:hmc}
\rho_{\Hyp^2}(z)|dz| \leq \frac{2}{\delta_{a_n}^{1/2}}|dz|.
\end{align}
Thus, by Equations \ref{eqn:error} and \ref{eqn:hmc}
$$
d_{\D}(f_n(b_n),R(b_n))\leq 2\delta_{a_n}^{1/2}\cdot \frac{2}{\delta_{a_n}^{1/2}} = 4.
$$
A similar proof and Schwarz lemma give that for $k=1,..., q$,
$$
d_{\D}(f_n^k(b_n), R^k(b_n)) \leq K_2
$$
for some constant $K_2$.
By Schwarz lemma and Equation \ref{eqn:bc}, for $k=1,..., q$, 
$$
d_{\D}(f_n^k(c_n), R^k(b_n)) \leq K_3
$$ 
for some constant $K_3$.
Therefore, $f_n$ is a quasi post-critically finite sequence and realizes the angled tree map with rotation number $r/q$.
\end{proof}

\subsection*{Construction of the reduction}
Let $f:(\mathcal{T}, \p) \longrightarrow (\mathcal{T}, \p)$ be a minimal admissible angled tree map of degree $d+1$.
We shall construct a new admissible angled tree map of degree $d$ by removing a `furthest' critical point.

Let $\mathcal{T}^1$ be the first pullback of $\mathcal{T}$.
Then $\p$ has exactly $d+1$ preimages in $\mathcal{T}^1$ counted with multiplicity.
Since $f$ is simplicial, $f:\mathcal{T}^1 \longrightarrow \mathcal{T}^1$ is distance non-increasing with respect to the edge metric $d_{\mathcal{T}^1}$.
Since the inclusion map $i: \mathcal{T} \longrightarrow \mathcal{T}^1$ is simplicial, $d_{\mathcal{T}^1}$ agrees with $d_\mathcal{T}$ on $\mathcal{T}$.

It is convenient to introduce a radius function $r: \mathcal{V}^1 \longrightarrow \N$ by setting $r(v) = d_{\mathcal{T}^1}(v, \p)$.
Denote $r_f = \max_{v\in f^{-1}(\p)} r(v)$.
We say $w\in f^{-1}(\p) \subseteq \mathcal{T}^1$ is a {\em furthest preimage} of $\p$ if 
$r(w) = r_f$.

\begin{lem}\label{lem:rb}
Let $v\in \mathcal{V} \subseteq \mathcal{V}^1$. 
Then
$$
r(v) + r(f(v)) \leq r_f.
$$
\end{lem}
\begin{proof}
First assume $v$ is a critical vertex.
Suppose for contradiction that $r(v) + r(f(v)) > r_f$.
Note that $f(v) \neq \p$, as otherwise, $v$ gives a further preimage of $\p$, contracting to the definition of $r_f$.

Let $a\in T_{f(v)}\mathcal{T} \subseteq T_{f(v)}\mathcal{T}^1$ be the tangent vector associated to the component containing $\p$.
By construction of $\mathcal{T}^1$, the tangent vector $a$ has $\delta(v) \geq 2$ preimages in $T_v\mathcal{T}^1$ under the tangent map $D f: T_v\mathcal{T}^1 \longrightarrow T_{f(v)}\mathcal{T} \subseteq T_{f(v)}\mathcal{T}^1$.
At least one direction will increase the radius function.
Thus, by considering the pullback of $[\p, f(v)]$ in the increasing direction, we can find a preimage $x$ of $\p$ with radius $r(x) = r(v)+r(f(v)) > r_f$, which is a contradiction.

If $v$ is post-critical, let $v= f^l(c)$ where $c$ is critical.
Since $f$ is distance non-increasing, $r(f(v)) \leq r(v) \leq r(f(c)) \leq r(c)$, so $r(v)+r(f(v)) \leq r_f$.

If $v$ is neither critical nor post-critical, since $f$ is assumed to be minimal on $\mathcal{T}$, $v$ is not an endpoint. 
Let $w_1,w_2$ be two critical or post-critical vertices such that $v\in [w_1,w_2]$ and there are no other critical or post-critical points in $(w_1, w_2)$.
Thus, $f$ restricts to an isomorphism from $[w_1,w_2]$ to $[f(w_1), f(w_2)]$.
Note that at least one of the two points, say $w_1$, has larger radius than $v$.
Write $r(v) = r(w_1) - l$ for some positive integer $l$.
Since $f$ is an isomorphism from $[w_1,w_2]$ to $[f(w_1), f(w_2)]$, we have $r(f(v)) \leq r(f(w_1)) + l$.
Thus, $r(v) + r(f(v)) \leq r(w_1) + r(f(w_1)) \leq r_f$.
\end{proof}

Let $w$ be a preimage of $\p$ in $\mathcal{T}^1$.
Then there exists at least one critical point on the path $[\p, w] \subseteq \mathcal{T}^1$.
A critical point $c_w \in\mathcal{V}$ on the path $[\p, w]$ is said to be {\em furthest} if $r(c_w) \geq r(c)$ for any critical point $c \in \mathcal{V}$ on the path $[\p,w]$.
We call $c_w$ the critical point associated with $w$.

\begin{lem}\label{lem:fe}
If $w \in \mathcal{V}^1$ is a furthest preimage of $\p$, then
\begin{enumerate}
\item $r_f = r(c_w) + r(f(c_w))$;
\item $\mathcal{T} \cap [\p, w] = [\p, c_w]$.
\end{enumerate}
\end{lem}
\begin{proof}
Since there are no critical points in $(c_w, w]$, we have
$$
d_{\mathcal{T}^1}(c_w, w) = d_{\mathcal{T}^1}(f(c_w), p) = r(f(c_w))
$$ 
and the first statement follows.

Suppose for contradiction that the second statement does not hold.
Since $f$ is minimal, there exists a post-critical vertex $v'$ on $(c_w, w]$.
Note that $v'$ has a preimage $v$ in $\mathcal{T}$ as $v'$ is post-critical.
Since $f$ is distance non-increasing, we have $r(f(c_w)) \leq r(c_w) < r(v') \leq r(v)$, so $r(v)+r(v') = r(v)+r(f(v)) > r_f$ which is a contradiction to Lemma \ref{lem:rb}.
\end{proof}

\begin{defn}
A preimage $w \in \mathcal{V}^1$ of $\p$ is said to be {\em critically furthest} if
\begin{itemize}
\item $r(w) = r_f$;
\item $r(c_w) = \max \{r(c_v): f(v) = \p \text{ and } r(v) = r_f\}$.
\end{itemize}
\end{defn}
We remark that $c_w$ may not be a furthest critical point. It is only furthest among the critical points associated to furthest preimages of $\p$.

Let $w$ be a critically furthest preimage of $\p$.
Note that $[\p, w]$ is a path in the pullback $\mathcal{T}^1$.
The following lemma says that the associated critical point $c_w \in [\p, w]$ is an end point of the original tree $\mathcal{T}$.

\begin{lem}\label{lem:end}
If $w \in \mathcal{V}^1$ is a critically furthest preimage of $\p$, then $c_w$ is an end point of $\mathcal{T}$.
\end{lem}
\begin{proof}
Suppose for contradiction that $c_w$ is not an end point of $\mathcal{T}$.
There is at least one direction $a\in T_{c_w}\mathcal{T}$ for which the radius function is increasing.
By Lemma \ref{lem:rb}, $Df(a) \in T_{f(c_w)}\mathcal{T}$ corresponds to the component containing $\p$.
Thus, by considering the pullback of $[\p, f(c_w)]$ in this direction, we find another furthest preimage $v$ of $\p$.
Since $[\p, c_w]$ is strictly contained in $[\p, v) \cap \mathcal{T}$, by the second conclusion of Lemma \ref{lem:fe}, $c_w$ is not the furthest critical point on the path $[\p, v]$.
Let $c_v$ be the associated critical point for $v$. 
Then $r(c_v) > r(c_w)$.
Thus, $w$ is not critically furthest, which is a contradiction.
\end{proof}

Let $w$ be a critically furthest preimage of $\p$.
We first define
$$
\tilde f = f: (\mathcal{T},\p) \longrightarrow (\mathcal{T}, \p)
$$
with the same degree function and angle function except at $c:= c_w$, where $\tilde\delta(c) = \delta(c) - 1$.
Since $c$ is an endpoint by Lemma \ref{lem:end}, it is easy to check the above definition gives an admissible angled tree map of degree $d$.

Let $\widetilde{\mathcal{T}}\subseteq \mathcal{T}$ be the angled subtree where the dynamics $\tilde f$ is minimal.
We call
$$
(\tilde f: (\widetilde{\mathcal{T}},\p) \longrightarrow (\widetilde{\mathcal{T}}, \p), \tilde \delta, \tilde \alpha)
$$
the {\em reduction} of $f$.

Note that if $\delta(c) \geq 3$, then the dynamics $\tilde f$ on $\mathcal{T}$ is minimal, so $\widetilde{\mathcal{T}}=\mathcal{T}$.
If $\delta(c) = 2$, then the dynamics $\tilde f$ may or may not be minimal on $\mathcal{T}$ (depending on whether $c$ is in post-critical set of other critical points or not).
In any case, $\mathcal{T}$ is the convex hull of $\widetilde{\mathcal{T}}$ and the orbit of $c$.

\subsection*{Induction step for realization}
Let $f:(\mathcal{T},\p) \longrightarrow (\mathcal{T}, \p)$ be a minimal admissible angled tree map of degree $d+1$. 
If $\delta(\p) \geq 2$, let $\mathcal{F}$ be a marked and normalized mapping scheme on $\mathcal{S}_\p$.

If $\delta(\p) = d+1$, then Proposition \ref{prop:rt} is vacuously satisfied by taking the constant sequence $f_n(z) = \mathcal{F}_{\p}(z)$.
Thus, we assume $\delta(\p) \leq d$.
Note that in this case, $\tilde{\delta}(p) = \delta(p)$ and $\widetilde{\mathcal{F}}_p = \mathcal{F}_p$.

We will break up the proof of Proposition \ref{prop:rt} into three steps:
\begin{itemize}
\item Using induction, we first show the quasi-invariant tree $f:(\mathcal{T},\p) \longrightarrow (\mathcal{T}, \p)$ is realized (Lemma \ref{lem:rt}), which already implies Theorem \ref{thm:ar};
\item We then show the conditions for periodic rescaling limits in Proposition \ref{prop:rt} are satisfied (see IH \ref{ih:3} below and Lemma \ref{lem:iF});
\item Finally, we perform surgery so that the conditions for strictly pre-periodic rescaling limits in Proposition \ref{prop:rt} are satisfied (Lemma \ref{lem:mfp}).
\end{itemize}

Let $(\tilde f: (\widetilde{\mathcal{T}},\p) \longrightarrow (\widetilde{\mathcal{T}}, \p), \tilde \delta, \tilde \alpha)$ be the reduction of $f$.
We assume the following technical and auxiliary induction hypotheses for the reduction. It is easy to verify that all these induction hypotheses are satisfied for the degree $2$ base case.
\begin{ih}\label{ih:m}
There exists a quasi post-critically finite sequence $\tilde f_n$ of degree $d$ realizing $\tilde f:(\widetilde{\mathcal{T}},\tilde\p) \rightarrow (\widetilde{\mathcal{T}}, \tilde\p)$, with isomorphisms $\tilde\phi_n:\widetilde{\mathcal{T}} \rightarrow \widetilde{\mathcal{T}}_n$, such that
the rescaling limit $\widetilde{F}_p$ at $p$ is conjugate to $\widetilde{\mathcal{F}}_p$ with compatible marking.
\end{ih}
\begin{ih}\label{ih:1}
There exists a sequence $R_n\to \infty$ such that
$$
d_{\D}(\tilde\phi_n(v), \tilde\phi_n(w)) = d_{\widetilde{\mathcal{T}}}(v, w)R_n + O(1),
$$ 
where $O(1)$ depends on $\tilde f$ and $\widetilde{\mathcal{F}}$.
\end{ih}

\begin{ih}\label{ih:2}
Let $v\in \widetilde{\mathcal{V}}$ with $f(v) \neq \p$.
Let $s_0 = \lim_{f(v)} \tilde\phi_n(\p) \in \mathbb{S}^1_{f(v)}$ and $\widetilde{F}_{v}^{-1}(s_0) = \{t_1,..., t_{\tilde{\delta}(\tilde v)}\} \subseteq \mathbb{S}^1_{v}$. Then there exists $\tilde \theta = \tilde \theta(\tilde f)$ so that
$$
|t_i-t_j| \geq \tilde \theta \text{ for all } i\neq j.
$$  
Let $w_{i,n}$ be the nearest zero of $\tilde f_n$ with
$\lim_{v} w_{i,n} = t_i$. Then
$$
d_{\D}(\tilde \phi_n(v), w_{i,n}) = d_{\D}(\tilde \phi_n(v), w_{j,n}) \text{ for all } i\neq j.
$$
\end{ih}
\begin{ih}\label{ih:3}
There exists a constant $\widetilde{L} = \widetilde{L}(\tilde f)$ so that for any periodic cycle $\mathcal{C}$ other than $\p$, we can order the cycle by $\mathcal{C} = \{v_1,..., v_q\}$ so that
$$
\lim_{n\to\infty} d_{\D}(\tilde \phi_n(v_{i+1}), \tilde f_n(\tilde \phi_n(v_i))) \leq \widetilde{L} \text{ for all } i =1,..., q-1.
$$
\end{ih}

We remark that IH \ref{ih:2} is vacuously satisfied if $\delta(\tilde v) = 1$.
The nearest zero in IH \ref{ih:2} exists as $\widetilde{F}_{v}(t_i) = s_0$.
By Lemma \ref{lem:cl}, IH \ref{ih:2} gives a constant $\widetilde{R}$ depending only $\tilde f$, so that $B_{\Hyp^2}(\phi_n(v), \widetilde{R})$ contains $\delta(v)-1$ critical points of $\tilde f_n$.
Thus IH \ref{ih:3} gives that the rescaling limit
$$
\widetilde{F}_{v_1}^q: \D_{v_1} \longrightarrow \D_{v_1}
$$
is $\widetilde{M}$-uni-critical for some constant $\widetilde{M} = \widetilde{M}(\tilde f)$ depending only on $\tilde f$.

We also remark that the estimate in IH \ref{ih:3} may not hold for $i=q$, as the hyperbolic distance between a critical point and its image under the first return map $\tilde F_{v_1}^q$ may depend on $\mathcal{F}$.

\begin{lem}\label{lem:rt}
There exists a quasi post-critically finite sequence $f_n \in \BP_d$ which realizes $f:(\mathcal{T}, \p) \longrightarrow (\mathcal{T}, \p)$, and satisfies IH \ref{ih:m}, IH \ref{ih:1} and IH \ref{ih:2}.
\end{lem}
\begin{proof}
We consider two cases.

{\bf Case (1): $\widetilde{\mathcal{T}} = \mathcal{T}$.}

We set $\phi_n(v) = \tilde \phi_n(v)$ for $v\in \mathcal{V}$.
Let $M_{c,n}, M_{\phi(c), n} \in \Isom(\D)$ be the coordinate at $c$ and $\phi(c)$ respectively.
Let $\widetilde{F}_{c}$ be the rescaling limit at $c$.
Since $\delta(\p) \leq d$, $c\neq \p$.

If $f(c) = \p$, set $w_n = \phi_n(c)$.

Otherwise, let
$s_0 = \lim_{f(c)}\phi_n(\p) \in \mathbb{S}^1_{f(c)}$.
Let $t_1,..., t_{\tilde \delta(c)} \in \mathbb{S}^1_{c}$ be the preimages of $s_0$ under $\widetilde{F}_{c}$.
Let $w_{i,n}$ be the nearest zero of $\tilde f_n$ to $c_n$ with
$\lim_{c} w_{i,n} = t_i$.

By IH \ref{ih:2}, we choose a point $t \neq t_i \in \mathbb{S}^1_{c}$ so that the $t, t_1,..., t_{\tilde \delta(c)}$ are separated by at least $\theta$, where $\theta$ depends only on $f$.
Let $w_n\in \D$ be such that
\begin{itemize}
\item $M_{c,n}(w_n)$ lies in the geodesic ray $[0, t)$; and
\item $d_{\D}(\phi_n(c), w_n) = d_{\D}(\phi_n(c), w_{i,n})$ for any $i$.
\end{itemize}
Let $\hat{w}_n = \frac{w_n}{|w_n|} \in \mathbb{S}^1$, and $A_n(z) = \frac{z-w_n}{1-\overline{w_n}z}$.

Consider the sequence 
$$
f_n(z) = \frac{-1}{\hat{w}_n} A_n(z) \tilde f_n(z),
$$
which is conjugate to a map in $\BP_d$ by some rotation. 
Again, abusing the notations, we shall not distinguish $f_n$ and its conjugate in $\BP_d$ in this proof.

Throughout this proof, if not specified, the constants and $O(1)$ depend on the tree map $f$ and the mapping scheme $\mathcal{F}$.
Let $v\in \mathcal{V}$. 
We first claim that there exists a constant $K_1$ with
$$
d_{\D}(f_n( \phi_n(v)), \tilde f_n(\phi_n(v))) \leq K_1.
$$

Denote $\delta_n:= 1-|w_n|$ and $\rho_n:= d_{\D}(0, w_n)$, then $\rho_n = r(w)R_n + O(1)$.
Let $s = \frac{r(v)}{r(w)}$.
By Lemma \ref{lem:aem}, there exists some constant $K_2$ such that the error term satisifes
\begin{align}\label{eqn:del}
|f_n(\phi_n(v)) - \tilde f_n(\phi_n(v))| = |\tilde f_n(\phi_n(v))||A_n(\phi_n(v)) + \hat{w}_n| \leq K_2 \delta_n^{1-s}.
\end{align}
By Lemma \ref{lem:rb}, we have $r(v) + r(f(v)) \leq r_f = r(w)$.
Thus 
$$
d_{\D}(0, \tilde f_n(\phi_n(v))) = r(f(v)) R_n + O(1) \leq (1-s) \rho_n + O(1).
$$
Since $|\frac{-1}{\hat{w}_n} A_n(z)| < 1$ for $z\in \D$,
$$
d_{\D}(0, f_n (\phi_n(v))) \leq d_{\D}(0, \tilde f_n(\phi_n(v))) \leq (1-s) \rho_n + O(1).
$$
Thus, by Lemma \ref{lem:hme}, the hyperbolic metric at $z$ along the Euclidean segment $[f_n (\phi_n(v)), \tilde f_n (\phi_n(v))]_{\R^2}$ satisfies
$$
\rho_{\Hyp^2}(z) |dz| \leq K_3 \frac{1}{\delta_n^{1-s}} |dz|.
$$
Together with Equation \ref{eqn:del}, we conclude that
$$
d_{\D}(f_n( \phi_n(v)), \tilde f_n(\phi_n(v))) \leq K_1,
$$
for some constant $K_1$.
Thus, the vertices for $\mathcal{T}_n$ are $K_1$ quasi-invariant.
The same proof as in Proposition \ref{prop:qi} shows that $\mathcal{T}_n$ is quasi-invariant under $f_n$.

Note that $f_n$ and $\tilde f_n$ have the same set of zeros except for $w_n$.
If $f(c) = \p$ i.e. $c = w$, then $w_n = \phi_n(c)$, so $\delta(c) = \tilde\delta(c) +1$.

Otherwise, by IH \ref{ih:2} and Lemma \ref{lem:cl}, there are $\delta(c)$ critical points within a bounded distance away from $\phi_n(c)$, so $\delta(c) = \tilde\delta(c) +1$.
The same argument also shows $\delta(v) = \tilde \delta(v)$ for all $v\neq c$.
A similar estimate on the first pullback tree $\mathcal{T}^1$ allows us to verify we have the correct marking for rescaling limits.
So $f_n$ realizes the angled tree map $f: (\mathcal{T}, \p) \rightarrow (\mathcal{T}, \p)$.

By construction, IH \ref{ih:m}, IH \ref{ih:1} and IH \ref{ih:2} are satisfied.
This proves the Case (1).

{\bf Case (2): $\widetilde{\mathcal{T}}\subsetneqq\mathcal{T}$.}

We show we can reduce to the first case.
During the reduction, we will encounter angled tree maps that are no longer minimal.
The notion of realization generalizes naturally to angle tree maps that are not minimal.

If $c \in \mathcal{T}$ is eventually mapped into $\widetilde{\mathcal{T}}$ by $\tilde f$, then there exists $k \geq 1$ with $\mathcal{T} \subseteq \widetilde{\mathcal{T}}^k$ for some pullback of $\widetilde{\mathcal{T}}$, which is realized by quasi-invariant tree $\widetilde{\mathcal{T}}^k_n$ for $\tilde {f}_n$.
It's easy to see $\widetilde{\mathcal{T}}^k$ satisfies all the induction hypotheses as all the additional vertices have degree $1$.
Applying the argument in Case (1) to $\widetilde{\mathcal{T}}^k$, we conclude Lemma \ref{lem:rt}.

Thus, we assume the orbit of $c$ does not intersect $\widetilde{\mathcal{T}}$.

Suppose $c$ is periodic with period $q$.
Let $a = \proj_{\widetilde{\mathcal{T}}}(c)$ be the projection of $c$ to $\widetilde{\mathcal{T}}$.
Since $c$ is periodic, and $f$ is distance non-increasing, $a$ is periodic as well. 
Since there are no periodic Julia branch points other than $\p$, $c$ is adjacent to $a$ and the period of $a$ divides $q$.

If $a$ is a periodic Fatou point, then the first return rescaling limit $\widetilde{F}$ at $a$ has degree $\geq 2$.
Let $t\in \mathbb{S}^1_{\Delta(a), \Delta_{pre}(a)}$ be the angle associated to the direction of $c$. Then the corresponding repelling periodic point $s\in \mathbb{S}^1_a$ for $\widetilde{F}$ is not a hole for $\widetilde{F}$ by Lemma \ref{lem:ct}.

Let $M_{a,n} \in \Isom(\D)$ be the local coordinate at $a$.
We define $\tilde{\phi}_n(c)$ so that
\begin{itemize}
\item $M_{a,n}(\tilde{\phi}_n(c))$ lies in the geodesic ray $[0, s)$; and 
\item $d_{\D}(0, M_n(\tilde{\phi}_n(c))) = R_n= d_{\mathcal{T}}(a, c) R_n$.
\end{itemize}
We define $\tilde\phi_n(\tilde f^k(c)) = \tilde f_n^{k}(\tilde{\phi}_n(c))$ for $k=1,..., p-1$.
We construct 
$$
\mathcal{T}_n = \widetilde{\mathcal{T}}_n \cup \bigcup_{k=0}^{q-1} [\tilde\phi_n(\tilde f^k(a)), \tilde\phi_n(\tilde f^k(c))].
$$
It is easy to verify that the map $\tilde f_n$ on $\mathcal{T}_n$ realizes the (non-minimal) angled tree map $\tilde f :(\mathcal{T}, \p) \longrightarrow (\mathcal{T}, \p)$ and satisfies all the induction hypotheses as all the additional vertices have degree $1$. 
Thus the lemma follows from the argument in Case (1).

If $a$ is a periodic Julia point and $a\neq \p$, 
since $\mathcal{T}$ is admissible, $a$ is not a branch point for $\mathcal{T}$.
Thus $a$ is an end point of $\widetilde{\mathcal{T}}$.
Therefore, if $t\in \mathbb{S}^1_{\Delta(a), \Delta_{pre}(a)}$ is the angle at $a$ associated to the direction of $c$, it corresponds to the unique repelling fixed point $s\in \mathbb{S}^1_a$ for the degree $1$ first return rescaling limit $\widetilde{F}$ at $a$.
The proof is similar to the previous case.

If $a$ is a periodic Julia point and $a= \p$, then $\mathcal{T}$ is a star-shaped tree. 
The lemma then follows directly from a similar argument as in Lemma \ref{lem:r2} (see also \cite[\S 14]{McM09}).

Finally suppose $c$ is strictly pre-periodic. 
Let $b = \tilde f^l(c)$ where $l$ is the pre-period, and $a= \proj_{\widetilde{\mathcal{T}}}(b)$.
The same argument as in the case when $c$ is periodic shows that the convex hull $\mathcal{T}' \subseteq \mathcal{T}$ of $\widetilde{\mathcal{T}}$ and the orbit of $b$ is realized by $\tilde{f}_n$. 
Then a similar argument as in the case when $c$ is mapped into $\mathcal{T}$ shows that by pulling back, $\tilde f :(\mathcal{T}, \p) \longrightarrow (\mathcal{T}, \p)$ is realized by $\tilde{f}_n$ and satisfies all the induction hypotheses.
Thus, the argument in Case (1) gives that $f :(\mathcal{T}, \p) \longrightarrow (\mathcal{T}, \p)$ is also realizable. This proves the Case (2).
\end{proof}

\subsection*{Rescaling limits for periodic orbits}
We now show the bound for periodic orbits is independent of $\mathcal{F}$.
\begin{lem}\label{lem:iF}
The realization $f_n$ for $f: (\mathcal{T},\p) \rightarrow (\mathcal{T},\p)$ satisfies IH \ref{ih:3}.
\end{lem}
\begin{proof}
Recall that we have two cases after the reduction.
We consider Case (1) for $\widetilde{\mathcal{T}} = \mathcal{T}$ here. 
The same modification as in Lemma \ref{lem:rt} can be used to prove Case (2) for $\widetilde{\mathcal{T}}\subsetneqq\mathcal{T}$.

After passing to a subsequence, we may assume all limits exist in the following discussion.
Let $\{v_1,..., v_q\}$ be a periodic cycle.
Changing the ordering if necessary, by IH \ref{ih:3}, there exists $M_1 = M_1(\tilde{f})$ so that for $i=1,..., q-1$, 
\begin{align}\label{eqn:M1}
\lim_{n\to\infty} d_{\D}(\phi_n(v_{i+1}), \tilde f_n(\phi_n(v_i))) \leq M_1.
\end{align}

We claim that after passing to a subsequence, there exists a constant $M_2 = M_2(f)$ so that for all $i=1,..., q$,
\begin{align}\label{eqn:M2}
\lim_{n\to \infty} d_{\D}(f_n( \phi_n(v_i)), \tilde f_n(\phi_n(v_i))) \leq M_2.
\end{align}

Let $w$ be the new pre-image of $\p$.
Then $r(w) = r_f$. 
Since $v_i \in \mathcal{V}$, by Lemma \ref{lem:rb}, $r(v_i) + r(f(v_i)) \leq r_f$.
Since $v_i$ is periodic, $r(f(v_i)) = r(v_i)$.
Thus $r(v_i) \leq r_f/2$.

Suppose we have strict inequality $r(v_i) < r_f/2$. 
Then the error term is
\begin{align*}
|f_n(\phi_n(v_i)) - \tilde f_n(\phi_n(v_i))| \leq K_1 \delta_n^{1-\frac{r(v_i)}{r_f}},
\end{align*}
where $K_1$ is a constant depending on $f$ and $\mathcal{F}$.
The hyperbolic metric at $\tilde f_n(\phi_n(v_i))$ and $f_n(\phi_n(v_i))$ is
$$
\rho_{\Hyp^2}(f_n(\phi_n(v_i)))|dz| \leq \rho_{\Hyp^2}(\tilde f_n(\phi_n(v_i)))|dz| \leq K_2 \delta_n^{-\frac{r(f(v_i))}{r_f}}|dz|,
$$
where $K_2$ is a constant depending on $f$ and $\mathcal{F}$.
Since $\delta_n \to 0$, and 
$$
1-{(r(v_i)+r(f(v_i)))}/{r_f} = 1-2r(v_i)/r_f > 0,
$$ 
we have
$$
\lim_{n\to \infty} d_{\D}(f_n( \phi_n(v_i)), \tilde f_n(\phi_n(v_i))) \leq \lim_{n\to \infty} K_1K_2 \delta_n^{1-\frac{r(v_i)+r(f(v_i))}{r_f}} = 0.
$$ 
Thus, after passing to a subsequence, the claim follows by simply taking the constant $M_2 = 1$.

Therefore, we only need to consider the case $r(v_i) = \frac{r_f}{2}$.
Suppose $v_i$ is not the midpoint of $[\p, w]$. Then a similar estimate as above would also give
$$
\lim_{n\to \infty} d_{\D}(f_n( \phi_n(v_i)), \tilde f_n(\phi_n(v_i))) = 0,
$$ 
so the claim follows in this case as well.

Finally, suppose $v_i$ is the midpoint of $[\p, w]$. 
By IH \ref{ih:2} and Lemma \ref{lem:sep}, there exists a constant $\theta = \theta(f)$ such that for any $n$, we have
$$
\angle \phi_n(w) \phi_n(v_i) \phi_n(\p)\geq \theta.
$$ 
By Schwarz lemma, 
$$
d_{\D}(0, \tilde f_n(\phi_n(v_i))) \leq d_{\D}(\phi_n(w), \phi_n(v_i)).
$$
Let $\epsilon_n > 0$ be so that $|\tilde f_n(\phi_n(v_i))| = 1-\epsilon_n$.
Recall that $f_n(z) = \frac{-1}{\hat{w}_n} A_n(z) \tilde f_n(z)$, where $A_n(z) = \frac{z-\phi_n(w)}{1-\overline{\phi_n(w)}z}$, and $\hat{w}_n = \frac{\phi_n(w)}{|\phi_n(w)|} \in \mathbb{S}^1$.
By Lemma \ref{lem:aem}, there exists a constant $K_3 = K_3(\theta)$ so that for all $n$, $|A_n(\phi_n(v_i)) + \hat w_n| \leq K_3 \epsilon_n$.
Therefore,
\begin{align*}
|f_n(\phi_n(v_i)) - \tilde f_n(\phi_n(v_i))| = |\tilde f_n(\phi_n(v_i)||1+\frac{1}{\hat w_n} A_n(\phi_n(v_i))| \leq K_3 \epsilon_n.
\end{align*}
By Lemma \ref{lem:he},
$$
\rho_{\Hyp^2}(f_n(\phi_n(v_i)))|dz|\leq \rho_{\Hyp^2}(\tilde f_n(\phi_n(v_i)))|dz| \leq \frac{4}{\epsilon_n}|dz|.
$$
Therefore, 
$$
d_{\D}(f_n( \phi_n(v_i)), \tilde f_n(\phi_n(v_i))) \leq 4K_3.
$$
Since $K_3$ depends only on $f$, the claim follows.

Combining Equations \ref{eqn:M1} and \ref{eqn:M2}, there exists $M_3 = M_3(f)$ so that 
$$
\lim_{n\to\infty} d_{\D}(\phi_n(v_{i+1}), f_n(\phi_n(v_i))) \leq M_3
$$ 
for all $i=1,..., q-1$.
\end{proof}

\subsection*{Rescaling limits for strictly pre-periodic orbits}
Let $f_n \in \BP_d$ be the quasi post-critically finite sequence realizing $\mathcal{T}$ constructed as above.
Using a standard quasi-conformal surgery argument (cf. \cite[Theorem 5.7]{Milnor12} or \cite{BrannerFagella14}), we show:
\begin{lem}\label{lem:mfp}
We can modify $f_n$ so that it satisfies the Proposition \ref{prop:rt}.
\end{lem}
\begin{proof}
Let $f_n \in \BP_d$ be constructed as above.
Then $f_n$ has the correct rescaling limit at the periodic orbits, we need to modify the dynamics on the strictly pre-periodic vertices to get the desired rescaling limits while keeping the rescaling limits on periodic points unchanged.
Assume $\delta(\p) \geq 2$, and consider the case where $v$ is a strictly pre-periodic point with $f(v) = \p$.
The other strictly pre-periodic points can be treated using the same argument.

If $\delta(v) = 1$ then we can modify the marking $\phi_n$ so that $F_{v}(0) = 0$, and thus marked and normalized.

If $\delta(v) \geq 2$, we choose a large ball $B(0, r) \subseteq \D_p$ containing all critical values of the mapping scheme $\mathcal{F}_{v}$ and $F_{v}$.
We choose a larger ball $B(0, s) \subseteq \D_p$, 
so that $\mathcal{F}_{v}^{-1}(B(0,r)) \subseteq F_{v}^{-1}(B(0,s))$.
Let $U = \mathcal{F}_{v}^{-1}(B(0,r) \subseteq \D_v$ and $V=F_{v}^{-1}(B(0,s))\subseteq\D_v$.
Since $f_n \circ M_{v,n}$ converges compactly on $\D$ to $F_{v}$, there exists a component
$V_n$ of $(f_n \circ M_{v,n})^{-1}(B(0,s))$ approximating $V$.
For sufficiently large $n$, we define 
\[
\tilde g_n (z)=
\begin{cases}
f_n(z), & z \notin M_{v,n}(V_n) \\
\mathcal{F}_{v} \circ M^{-1}_{v,n}(z), & z \in M_{v,n}(\overline{U}) \\
H_n(z), & z \in M_{v,n}(V_n - \overline{U})
\end{cases}
\]
where $H_n(z)$ is a $K$ quasi-regular degree $\delta(v)$ covering between annuli $M_{v,n}(V_n - \overline{U})$ and $B(0, s) - \overline{B(0,r)}$ interpolating boundary values. 
Note $K$ can be chosen to be independent of $n$.
Let $\mu'_n$ be the Beltrami differential on $\D$ associated to $\tilde g_n$.
Then $\mu'_n = 0$ away from the annulus $M_{v,n}(V_n - \overline{U})$.
For sufficiently large $n$, orbits of $\tilde{g}_n$ can pass through $M_{v,n}(V_n - \overline{U})$ at most once.
We extend $\mu'_n$ to a Beltrami differential on $\hat\C$ by reflecting along $\mathbb{S}^1$.
Thus, we can construct a $\tilde g_n$ invariant Beltrami differential $\mu_n$ which has bounded dilatation. 
By measurable Riemann mapping theorem, there exists a $K$-quasiconformal map $\psi_n$ and a Blaschke product $g_n$ so that 
$$
\tilde g_n = \psi_n^{-1}\circ g_n \circ \psi_n.
$$
Let $L_{v,n}\in \Isom(\D)$ with $L_{v, n}(0) = \psi_n(v_n)$, and set $\psi_{v,n}=L_{v,n}^{-1} \circ \psi_n\circ M_{v,n}$.
We have the following commutative diagram.
\[ \begin{tikzcd}
(\D, 0) \arrow{r}{M_{v,n}} \arrow[swap]{d}{\psi_{v,n}} & (\D, v_n) \arrow{r}{\tilde g_n} \arrow{d}{\psi_n} & (\D, \tilde g_n(v_n)) \arrow{d}{\psi_n}\\%
(\D, 0) \arrow{r}{L_{v,n}} & (\D, \psi_n(v_n)) \arrow{r}{g_n} & (\D, g_n(\psi_n(v_n)))
\end{tikzcd}
\]

After passing to a subsequence, $M_{v,n} \circ \tilde g_n$ converges compactly on $\D$ to a proper map of degree $\delta(v)$ and $\psi_n$, $\psi_{v,n}$ converge to $K$ quasiconformal maps that preserve the circle.
Thus,  $L_{v,n} \circ g_n$ converges to a proper map of degree $\delta(v)$.
We denote this rescaling limit by $G_{v}$.
Since for sufficiently large $n$, $\psi_{v,n}$ is conformal on $U$ where $\tilde g_n \circ M_{v,n} = \mathcal{F}_{v}$ and $U$ contains all critical points of $\mathcal{F}_{v}$, $G_{v}$ is conformally conjugate to $\mathcal{F}_{v}$ (see \cite[Lemma 5.10]{Milnor12}).
Adjusting the coordinate $L_{v,n}$ and interpolating function $H_n$ if necessary, we can assume the rescaling limit $G_{v}$ is normalized and has compatible marking.

We now show the surgery does not change the rescaling limits on periodic points.
Let $w\in \mathcal{V}$ be a periodic point.
Without loss of generality, we assume $w$ is fixed. 
Let $\Omega\subseteq \D_w$ be a compact set.

We claim that the orbit of $z\in M_{w, n}(\Omega)$ under $f_n$ does not pass through $M_{v,n}(V_n)$ for all sufficiently large $n$,
Indeed, since $v$ is strictly pre-periodic, there exists $k_0$ so that for all $k\geq k_0$ and all sufficiently large $n$, 
$$
d_{\D}(\p_n, f_n^{-k}(M_{v,n}(V_n))) \geq 2 d_{\mathcal{T}}(\p, w) R_n.
$$
Therefore by Schwarz lemma, for all sufficiently large $n$ and $z \in M_{w, n}(\Omega)$, $f_n^k(z) \notin M_{v,n}(V_n)$ for all $k \geq k_0$.
On the other hand, since $w$ is fixed, for sufficiently large $n$ and $z\in M_{w, n}(\Omega)$, $f_n^k(z) \notin M_{v,n}(V_n)$ for all $k \leq k_0$.
Therefore, $(M_{w,n}^{-1})^*\mu_n$ converges to $0$ in $L^1$ norm.
Let $L_{w,n}\in \Isom(\D)$ with $L_{w,n}(0) = \psi_n(w_n)$ so that 
$$
\psi_{w,n} := L_{w,n}^{-1} \circ \psi_n\circ M_{w,n}
$$
fixes $0, 1, \infty$.
Thus, $\psi_{w,n}$ converges uniformly to the identity map (see \cite[Proposition 4.7.2]{Hubbard06}).
Denote $g_{w,n} = L_{w,n}^{-1}\circ g_n \circ L_{w,n}$ and $\tilde g_{w,n} = M_{w,n}^{-1}\circ \tilde g_n \circ M_{w,n}$.
We have the following commutative diagram.
\[ \begin{tikzcd}
(\D, 0) \arrow{r}{M_{w,n}} \arrow[swap]{d}{\psi_{w,n}} & (\D, w_n) \arrow{r}{\tilde g_n} \arrow{d}{\psi_n} & (\D, \tilde g_n(w_n)) \arrow{d}{\psi_n} & (\D, \tilde g_{w,n}(0)) \arrow[swap]{l}{M_{w,n}} \arrow{d}{\psi_{w,n}} \\%
(\D, 0) \arrow{r}{L_{w,n}} & (\D, \psi_n(w_n)) \arrow{r}{g_n} & (\D, g_n(\psi_n(w_n))) & (\D, g_{w,n}(0))\arrow[swap]{l}{L_{w,n}}
\end{tikzcd}
\]
Therefore, the rescaling limit $G_{w} :=  \lim_{n\to\infty} L_{w,n}^{-1} \circ g_n \circ L_{w,n}$ equals to $F_{w}$, i.e., we have not changed the rescaling limits at the periodic points.

By choosing the region of modification carefully (see \cite[Lemma 5.9]{Milnor12}), the same argument allows us to modify all strictly pre-periodic points in $\mathcal{V}$ simultaneously.
It is easy to verify that the modified sequence realizes $f:(\mathcal{T}, \p) \rightarrow (\mathcal{T}, \p)$ and satisfies the two conditions in Proposition \ref{prop:rt}.
\end{proof}

\section{Pointed Hubbard trees}\label{sec:pht}
In this section, we will prove the necessary condition for a geometrically finite polynomial to be on the boundary of $\PH_d$:
\begin{prop}\label{prop:NC}
Let $(H,\vp)$ be the pointed Hubbard tree for a geometrically finite polynomial $\hat P \in \overline{\PH_d}$. Then $(H,\vp)$ is iterated-simplicial.
\end{prop}

\subsection*{Pointed Hubbard trees}

Recall that given a monic and centered polynomial $P$ with connected Julia set, there exists a unique B\"ottcher map normalized with derivative $1$ at infinity.
This gives a unique choice of the angle $0$ external ray, and thus a marking on $P$.
In this section, all polynomials considered are monic and centered.

Given a geometrically finite polynomial $\hat P$ with connected Julia set, it can be perturbed into a sub-hyperbolic polynomial with topologically conjugate dynamics on the Julia set \cite{Haissinsky00}.
A quasi-conformal surgery then gives a post-critically finite polynomial $P$ associated to $\hat P$ \cite{HT04}.

For a post-critically finite polynomial $P$, the Hubbard tree $H$, introduced in \cite{DH85}, is defined as the `regulated hull' of the critical and post-critical points in the filled Julia set $K_P$.
More precisely, an arc $I \subseteq K_P$ is called {\em regulated} if its intersection with any bounded Fatou component consists of (at most two) segments of internal rays.
The Hubbard tree is the minimal closed regulated connected subset of $K(P)$ containing the critical and post-critical points (see \cite[\S 1]{Poirier93}).

The polynomial restricts to a map $P:H \longrightarrow H$.
In our setting, the Hubbard tree $H$ is marked by the B\"ottcher map.
We say $P$ is simplicial on $H$ if there exists a finite simplicial structure on $H$ for which $P$ is a simplicial map, i.e., $P$ sends an edge of $H$ to an edge of $H$.
Abusing the notation, we call $H$ a {\em simplicial Hubbard tree} if the map $P: H \longrightarrow H$ is simplicial.

The realization of Hubbard trees has been studied in \cite{Poirier93}.
It is proved that an abstract angled Hubbard tree is realizable by a post-critically finite polynomial if and only if it is expanding.
Thus, we shall not distinguish the Hubbard trees of monic and centered post-critically finite polynomials with the abstract expanding angled Hubbard trees with markings.

\begin{defn}
Let $H$ be a Hubbard tree and $\vp \in H$ be a fixed point.
The pair $(H, \vp)$ is called a {\em pointed Hubbard tree}.
\end{defn}

A point $z\in J$ on the Julia set is said to be a {\em cut point} if $J -\{z\}$ is disconnected. The valence at $z$ is defined as the number of components of $J -\{z\}$.

Let $\hat P \in \partial \PH_d$ be a geometrically finite polynomial. 
There exists a special non-repelling fixed point $\hat p = \lim p_n$ of $\hat P$, where $p_n$ is the attracting fixed point of $P_n \in \PH_d$ and $P_n \to \hat P$.
We remark that $\hat p$ can still be an attracting fixed point for $\hat P$ (see the middle or the right example in Figure \ref{fig:R}).
This happens when some critical points stay at a bounded hyperbolic distance from the attracting fixed point, while some other critical points escape.

Let $P$ be the corresponding post-critically finite polynomial with Hubbard tree $H$.
The special non-repelling fixed point $\hat p$ gives a fixed point $H$ constructed as follows.
We remark that the reason why we need to do some modification is that if the non-repelling fixed point $\hat p$ is on the Julia set, the corresponding fixed point for $P$ may or may not be on the Hubbard tree $H$.

\begin{itemize}
\item If $\hat \vp$ is attracting, then it is contained in a fixed critical Fatou component. We set $\vp$ as the corresponding Fatou fixed point in $H$.

\item If $\hat \vp$ is a parabolic endpoint, then $\hat \vp$ is on the boundary of a unique fixed critical Fatou component.
If the corresponding Julia fixed point is in $H$, we set $\vp$ to be that point; otherwise, we set $\vp$ to be the Fatou fixed point in $H$ corresponding to the fixed critical Fatou component.

\item If $\hat \vp$ is a parabolic cut point, then the corresponding Julia fixed point is contained in $H$. We set $\vp$ to be the corresponding Julia fixed point.
\end{itemize}

We call $(H, \vp)$ the corresponding pointed Hubbard tree for $\hat P$.

For $f\in \BP_d$, any point in $\D$ is mapped towards the attracting fixed point $0$ by Schwarz lemma.
The following key proposition is a manifestation of Schwarz lemma for maps on $\partial \PH_d$.
\begin{prop}\label{prop:bp}
Let $\hat P \in \overline{\PH_d}$ be geometrically finite with the special non-repelling fixed point $\hat \vp$. 
If $\hat v$ is a periodic cut point in $J = J(\hat P)$ with valence $\nu$, then $\hat v$ is parabolic. Moreover, let $K = K(\hat P)$ be the filled Julia set of $\hat P$.
\begin{itemize}
\item If $\hat v = \hat \vp$, then $\hat v$ has exactly $\nu$ attracting basins which are in bijective correspondence with components of $K-\{\hat v\}$.
\item If $\hat v \neq \hat \vp$, then $\hat v$ has exactly $\nu-1$ attracting basins which are in bijective correspondence with components of $K-\{\hat v\}$ that does not contain $\hat \vp$.
\end{itemize}
\end{prop}
\begin{proof}
The fact that $\hat v$ is parabolic follows immediately from the stability of landing rays (see \cite[Lemma B.1]{GM93}).
Note that there are exactly $\nu$ external rays landing at $\hat v$.

If $\hat v = \hat \vp$, then there are $\nu+1$ periodic points ($\nu$ repelling and $1$ attracting) for any approximating polynomial $P_n \in \PH_d$ that converge to $\hat v$. Thus the parabolic multiplicity of $\hat v$ is $\nu+1$, so there are $\nu + 1$ attracting basins.
Since each component of $K-\{\hat v\}$ can correspond to at most $1$ attracting basin of $\hat v$, and different components of $K-\{\hat v\}$ correspond to different attracting basins, the first case follows.

If $\hat v \neq \hat \vp$, then the parabolic multiplicity of $\hat v$ is $\nu$. 

We claim the component of $K-\{\hat v\}$ containing $\hat \vp$ does not give an attracting basin.
Note that the proposition follows from the claim as there are $\nu-1$ attracting basins and $\nu-1$ components of $K-\{\hat v\}$ that does not contain $\hat \vp$.
We shall prove the claim using the corresponding pointed Hubbard tree $(H, \vp)$.
The periodic cut points for $\hat P$ are in correspondence with the periodic Julia cut point in $H$ and there are only finitely many periodic cut points as they are all parabolic.

Suppose $v \in H$ is a closest periodic Julia cut point to $\vp$ that does not satisfy the claim, i.e., all periodic Julia cut points on $[\vp, v)$ satisfy the claim.
If there exist Julia periodic points in $[\vp, v)$, let $w$ be the furthest one to $\vp$.
Let $u_w$ be the adjacent vertex of $\hat w$ on $(w, v)$.
Since $\hat w$ satisfies the claim, $u_w$ is a Fatou periodic point.
If there are no Julia periodic points in $[\vp, v)$, then $\vp$ is a Fatou fixed point and let $u_w = \vp$.
Since $\hat v$ does not satisfy the claim, the adjacent vertex $u_v \in [u_w,v)$ is a Fatou fixed point.
If $u_w \neq u_v$, there exists a Julia periodic point in $(u_w, u_v)$ which is necessarily a cut point, contradicting the assumption that $w$ is the furthest periodic Julia point.
Thus $u_w=u_v$.
Let $\Omega$ be the corresponding Fatou component for $\hat P$. Then $\Omega$ converges to two distinct boundary points under iteration, which is a contradiction.
The claim follows and we conclude the proof of the proposition.
\end{proof}

As an illustration of the second case, consider the geometrically finite polynomial $\hat P \in \partial \mathcal{H}_4$ in Figure \ref{fig:P}.
There is a parabolic fixed point $\hat v \neq \hat p$ with valence $3$.
This fixed point has parabolic multiplicity $3$, and is on the common boundary of 3 Fatou components $U_1, U_2, U_3$.
The Fatou component $U_1$ that contains $\hat p$ gives a repelling direction at $\hat v$.
The attracting basins at $\hat v$ are thus in bijective correspondence with the other two Fatou components.

\subsection*{Pointed simplicial tuning}
If $\delta(\vp) \geq 2$, we define a combinatorial operation called {\em pointed simplicial tuning} on a pointed Hubbard tree $(H, \p)$.

Let $(H_Q, \vp')$ be a marked pointed simplicial Hubbard tree of degree $\delta(p)$ associated to a monic and centered polynomial $Q$.
The marking gives an identification of the incident edges with external rays for $H_Q$.
Let $H_\vp$ be the regulated hull of $H_Q$ with the landing points of those external rays (viewed as an abstract angled tree).
As the first step, we remove $\vp$ and glue the $H_\vp$ to the incident edges at $\vp$.
We remark that to be more precise, we need to remove $T_{\vp}= H \cap U_\vp$, where $U_{\vp}$ is the Fatou component of $\vp$ for the polynomial $P$ associated to the Hubbard tree $H$, and glue back $H_{\vp}$.

Let $w \in H$ be a preimage of $\vp$. 
Suppose that $\delta(w) = 1$.
Note that the incident edges of $w$ corresponds to vertices in $H_\vp$,
we can thus remove $w$ and glue the regulated hull of those corresponding vertices in $H_\vp$ (viewed as an abstract angled tree) to the incident edges at $w$.

If $\delta(w) \geq 2$, we need to specify the pullback map as follows.
Let 
$$
P_w: \C \rightarrow \C
$$ be a polynomial of degree $\delta(w)$ with critical values $z_1,..., z_k$ such that
under $Q$, each $z_i$ is eventually mapped to a periodic critical cycle or a repelling periodic cycle on the boundary of a bounded Fatou component for $Q$.

Let $\tilde{H}_\vp$ be the regulated hull of $H_\vp$ and the orbits of $z_i$ (viewed as an abstract angled tree).
The condition on $z_i$ guarantees that $Q$ is simplicial on $\tilde{H}_\vp$.
Let $H_w$ be the pullback $P_w^{-1}(\tilde{H}_\vp)$.
Similar as before, we remove $w$ and glue a homeomorphic copy of $H_w$, and also replace $H_\vp$ with $\tilde{H}_\vp$.
The dynamics and angle structures are defined naturally.

Inductively, we replace the backward orbits of $\vp$ in the vertex set of $H$ by pullbacks with the above algorithm.
The algorithm terminates as the number of branch points and the critical points in $H$ are finite. 

This algorithm may produce some end points that are not critical or post-critical.
After deleting these end points and the corresponding edges if necessary, we obtain a new expanding pointed angled Hubbard tree $(H', \vp')$, which we call a {\em pointed simplicial tuning} of $(H,\vp)$ (see Figure \ref{fig:ST}).
We remark that depending on how the pointed simplicial tuning is performed, $(H', \vp')$ may or may not be simplicial.

Unlike general tuning, the pointed simplicial tuning can only be performed finitely many times, as the degree of the marked fixed point $\vp'$ is strictly smaller than the degree of $\vp$ after the operation.
We say a pointed Hubbard tree $(H,\vp)$ is {\em iterated-simplicial} if it can be constructed from the trivial pointed Hubbard tree by a sequence of pointed simplicial tunings.

\begin{figure}[ht]
  \centering
  \resizebox{1\linewidth}{!}{
    \def\svgwidth{\columnwidth}
    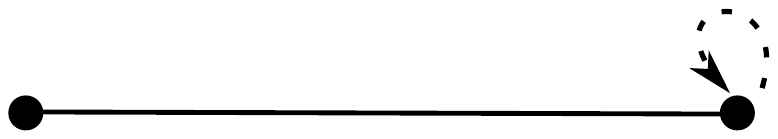

    \def\svgwidth{\columnwidth}
    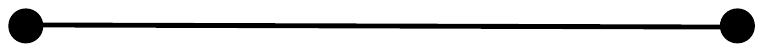

  }
   \resizebox{0.7\linewidth}{!}{
    \def\svgwidth{\columnwidth}
    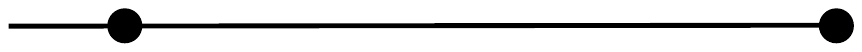

  }
  \caption{Illustrations of pointed simplicial tunings. The angles for edges are labeled on the graph. The resulting pointed Hubbard trees for the first two are not simplicial.}
  \label{fig:ST}
\end{figure}

\subsection*{Pointed simplicial quotient}
An inverse operation for pointed simplicial tuning is {\em pointed simplicial quotient}.
Let us fix a simplicial structure, i.e. a vertex set $V$ for $H$.
We assume that $V$ is forward invariant, and contains all critical points.
A subtree $\vp\in S \subseteq H$ is said to be {\em invariant} if $P(S) \subseteq S$.
An invariant subtree is said to be {\em simplicial} if $P$ is a simplicial map on $S$.
The union of any two simplicial invariant subtrees is again a simplicial invariant subtree.
Thus the maximal simplicial invariant subtree is well-defined.

\begin{lem}\label{lem:sbv}
Let $(H,\vp)$ be the pointed Hubbard tree corresponding to a geometrically finite polynomial $\hat P \in \overline{\PH_d}$.
Let $a\in T_\vp H$ be a periodic tangent direction at $\vp$ of period $q$, and let $E=[\vp,w]$ be the corresponding incident edge.
Then $P^q$ maps $E$ homeomorphically to $E$.
\end{lem}
\begin{proof}
Suppose $\vp$ is a Julia fixed point. Then the adjacent vertices are periodic Fatou points by Proposition \ref{prop:bp}, and $P$ is simplicial on the union of these edges.

Suppose $\vp$ is a Fatou fixed point. Let $\Omega$ be the corresponding fixed Fatou component for $\hat P$.
Let $x\in \partial \Omega$ be the corresponding periodic point in the direction $a$.
If $w=x$, then $P^q$ is a homeomorphism on $[\vp, w]$.
Otherwise, by Proposition \ref{prop:bp}, $w$ is a periodic Fatou point and the corresponding Fatou component is attracted to $\Omega$ at $x$ and the conclusion follows.
\end{proof}

\begin{cor}
Let $(H,\vp)$ be the pointed Hubbard tree corresponding to a geometrically finite polynomial $\hat P \in \overline{\PH_d}$.
If $H$ is non-trivial, i.e., $H \neq \{\vp\}$, then the maximal simplicial invariant subtree $S$ is non-trivial.
\end{cor}

We say $S'$ is a {\em subtree preimage} of $S$ if $S'$ is a subtree (with the simplicial structure given by $V$) and $P(S') \subseteq S$.
The maximal one exists by taking unions.
Let $S$ be the maximal simplicial invariant subtree of $(H,\vp)$.
Note that $S$ itself is a maximal subtree preimage of $S$.
We construct a new tree $H'$ from $H$ by collapsing $S$ and each inductive maximal subtree-preimages to a point and let $\vp' \in H'$ be the point associated with $S$.
Similarly, the new vertex set $V'$ is constructed from $V$ by collapsing, and the map $P$ induces a map $P'$ on $V'$ by $P'([x]) = [P(x)]$ where $[x]$ represents the vertex in $V'$ that $x\in V$ collapses to.
The local degrees can also be recovered by counting the multiplicities of all collapsed critical vertices.

\begin{lem}
The map $P': V' \longrightarrow V'$ satisfies $P'(a) \neq P'(b)$ whenever $[a, b]$ is an edge in $H'$.
\end{lem}
\begin{proof}
Let $[a,b]$ be an edge. Let $x, y$ be vertices for $V$ that lie above $a, b$ respectively.
We may assume that $[x,y]$ is an edge of $H$.
Suppose for contradiction that $P'(x) = P'(y)$, then $[P(x)] = [P(y)]$.
Thus, $P(x)$ and $P(y)$ are contained in some inductive maximal subtree preimage $S'$ of $S$, so $[P(x), P(y)] \subseteq S'$.
Therefore $[x] = [y]$ by maximality of subtree preimage which is a contradiction.
\end{proof}

To construct the angles at the new vertices, we first show the conclusion of Lemma \ref{lem:sbv} also holds for $(H',\vp')$.
\begin{lem}\label{lem:sbv2}
Let $a'\in T_{\vp'} H'$ be a periodic tangent direction of period $q$, and let $E' = [\vp',w']$ be the corresponding incident edge. Then $(P')^q$ maps $E'$ homeomorphically to $E'$.
\end{lem}
\begin{proof}
The tangent direction $a'$ corresponds to an edge $E = [x,y] \subseteq H$ with $x\in S$.
Since $a'$ has period $q$, $P^q([x,y])$ contains $[x,y]$.
Let $I = [x',y'] \subseteq [x,y]$ be the interval closest to $x$ such that $P^q$ maps $I$ homeomorphically to $[x,y]$.
If $x = x'$, then $x$ is a periodic cut point. 
By Proposition \ref{prop:bp}, $[x,y]$ corresponds to an attracting petal for the corresponding parabolic point $\hat x\in J(\hat P)$.
Thus, we can extend $S$ and still get a simplicial action, which is a contradiction to the maximality of $S$.

Therefore, $x \neq x'$.
Suppose for contradiction that $y \neq y'$. 
By pulling back, there exists a periodic point on $(x',y')$, which is not on the boundary of any periodic Fatou component. 
This point gives a repelling periodic cut point for $\hat P$, which is a contradiction to Proposition \ref{prop:bp}.
Thus $y = y'$, and hence $(P')^q$ maps $E'$ homeomorphically to $E'$.
\end{proof}

Let $a'\in T_{\vp'} H'$ be a periodic tangent direction of period $q$ corresponding to the edge $[x,y] \subseteq H$.
As in the proof of Lemma \ref{lem:sbv2}, we have a map $P^q : [x',y] \subseteq [x,y] \longrightarrow [x,y]$.
Thus, there exists a point $w \in [x', y]$ with period $q$, corresponding to a periodic end for the polynomial $Q$ associated with $S$.
We define the angle at the direction $a'$ by the external angle that lands at $t$.
The angles for pre-periodic tangent directions and inductive preimages of $\vp'$ are defined by pullback.
Since $H$ is expanding, it is not hard to verify that $H'$ is expanding as well.
We call $(H', \vp')$ the {\em pointed simplicial quotient} for $(H,\vp)$.
It can be verified that this operation is an inverse for pointed simplicial tuning.
More precisely, if $(H',p')$ is a pointed simplicial quotient of $(H,p)$, then $(H,p)$ can be constructed from $(H',p')$ using pointed simplicial tuning.

This process can be iterated.
The same proof of Lemma \ref{lem:sbv2} using Proposition \ref{prop:bp} gives that unless the Hubbard tree is trivial, the maximal simplicial invariant subtree is non-trivial.
Since pointed simplicial quotient reduces the number of vertices, the process eventually terminates at the trivial Hubbard tree.
This proves Proposition \ref{prop:NC}.

\section{Boundary of $\PH_d$}\label{sec:BPH}
In this section, we prove the other direction of Theorem \ref{thm: eq}.
Let $\hat P \in \overline{\PH_d}$ be a marked geometrically finite polynomial associated with the marked pointed Hubbard tree $(H, \vp)$.
Let $P$ be the corresponding marked post-critically finite polynomial.
For each critical and post-critical Fatou component $\Omega$ of $P$, we choose a point on the boundary $t(\Omega)\in \partial \Omega$ as a marking which satisfies $P(t(\Omega)) = t(P(\Omega))$ (see \cite[\S 5]{Milnor12}).
The boundary marking is chosen using a similar anchored convention as in Definition \ref{defn:anchored} and Definition \ref{defn:mcpc} with respect to the external angle $0$.
The topological conjugacy carries this boundary marking to $\hat P$.
Recall that a Blaschke product is $M$-uni-critical if the critical set is contained in $B_{\Hyp^2}(0, M)$.
A proper holomorphic map $f: U \longrightarrow V$ between two connected and simply connected proper subsets of $\C$ is said to be $M$-uni-critical if there exist uniformizing maps $\phi_U$ and $\phi_V$ so that $\phi_V^{-1} \circ f \circ \phi_U: \D \longrightarrow \D$ is $M$-uni-critical.

Let $\mathcal{S}_\p \subseteq V$ be the set of backward orbits of $\p$ and let $\mathcal{F}$ be a normalized mapping scheme on $\mathcal{S}_\p$ (see Definition \ref{defn:nms}).
We first prove the following (cf. Proposition \ref{prop:rt}):
\begin{prop}\label{prop:rsh}
Let $(H,\vp)$ be a marked simplicial pointed Hubbard tree of degree $d$.
In the case when $\delta(\p) \geq 2$, we let $\mathcal{F}$ be a normalized mapping scheme on $\mathcal{S}_\p$.
Then there exists a sequence $P_n \in \PH_d$ converging to a geometrically finite polynomial $\hat P \in \overline{\PH_d}$ associated to $(H,\vp)$.
Moreover, 
\begin{enumerate}
\item the dynamics of $\hat P$ on the Fatou components corresponding to $\mathcal{S}_\p$ is conjugate to $\mathcal{F}$ (with compatible markings);
\item there exists a constant $M$ depending only on $(H,\vp)$ (and independent of $\mathcal{F}$) so that 
\begin{enumerate}
\item[(a)] for any periodic Fatou point $v\neq\p$ of period $q$, $\hat P^q$ on the corresponding Fatou component $\Omega_v$ is $M$-uni-critical;
\item[(b)] for any strictly pre-periodic Fatou point $w \notin \mathcal{S}_\p$, if $k$ is the smallest positive integer so that $f^k(w)$ is critical, then $\hat P^k: \Omega_w \longrightarrow \Omega_{f^k(w)}$ is $M$-uni-critical and the critical values are within $M$ hyperbolic distance from the critical points in $\Omega_{f^k(w)}$.
\end{enumerate}
\end{enumerate}
\end{prop}

\subsection*{Simplicial pointed Hubbard tree and admissible angled tree map}
Let $(H, \vp)$ be a marked simplicial pointed Hubbard tree.
We first associate an admissible angled tree map to it.
Indeed, the pointed Hubbard tree comes with a local degree function $\delta$.
The dynamics of the polynomial $P$ gives angles between any pair of edges incident at a vertex (see \cite{Poirier93}).
For a periodic Julia vertex, we follow the same convention as in the \S \ref{sec:atm} for realizing angled tree maps:
any two consecutive (in cyclic ordering at the vertex) have $\frac{1}{\nu}$ angle where $\nu$ is valence at the vertex.
To specify the angle 0, we use the anchored convention in Definition \ref{defn:comp} and Definition \ref{defn:angt}.
Therefore, $P:(H, \p) \longrightarrow (H, \p)$ is naturally an angled tree map.

The core of $H$ is critically star-shaped.
Indeed, if $\delta(\p) \geq 2$, then this is vacuously true.
Otherwise, each adjacent vertex $v$ to $\p$ is in a periodic Fatou point.

On the other hand, the Hubbard tree $H$ may contain many periodic Julia branch points.
In the following, we introduce an operation on these branch points, which we call  {\em split modification}, to get an admissible angled tree map.

Let $v \neq \p$ be a periodic Julia branch point.
After passing to an iterate, we may assume that $v$ is fixed.
Let $S$ be the star-shaped subtree consisting of all vertices adjacent to $v$.
Let $a_0$ be the vertex in $S$ corresponding to the direction associated to $\p$, and label the other vertices by $a_1,..., a_m$ in counterclockwise order.
Since $P$ is simplicial on $H$ and fixes $\p$, $a_0$ is fixed, and thus all $a_i$ are fixed.
By Proposition \ref{prop:bp}, each $a_i$ is a fixed Fatou point for $i=1,..., m$.

We modify $H$ locally in $S$ by first removing $v$ and its incident edges.
On the first level, we choose $k_1 \in \{1,..., m\}$ and connect $a_0$ with $a_{k_1}$.
On the second level, we choose $k_{2,1} \in \{1,..., k_1-1\}$ and $k_{2,2} \in \{k_1+1,..., m\}$ and connect $a_{k_1}$ with $a_{k_{2,1}}$ and $a_{k_{2,2}}$.
Inductively, $k_1, k_{2,1}, k_{2,2}$ divide the set $\{1,..., m\}$ into $4$ subsets (some subset may be empty), and we proceed as above for each of the subinterval.
The trees $\tilde{S}$ that can be constructed in this way will be called {\em admissible splitting} (see Figure \ref{fig:M}).
For an admissible splitting, each vertex $a_i$ $i=0,1,..., m$ can be assigned a level, which is the edge distance between $a_i$ and $a_0$. 
An edge connects a level $k$ to a level $k+1$ vertex.

\begin{figure}[ht]
  \centering
  \begin{subfigure}{1\textwidth}
  \centering
  \resizebox{1\linewidth}{!}{
    \def\svgwidth{\columnwidth}
    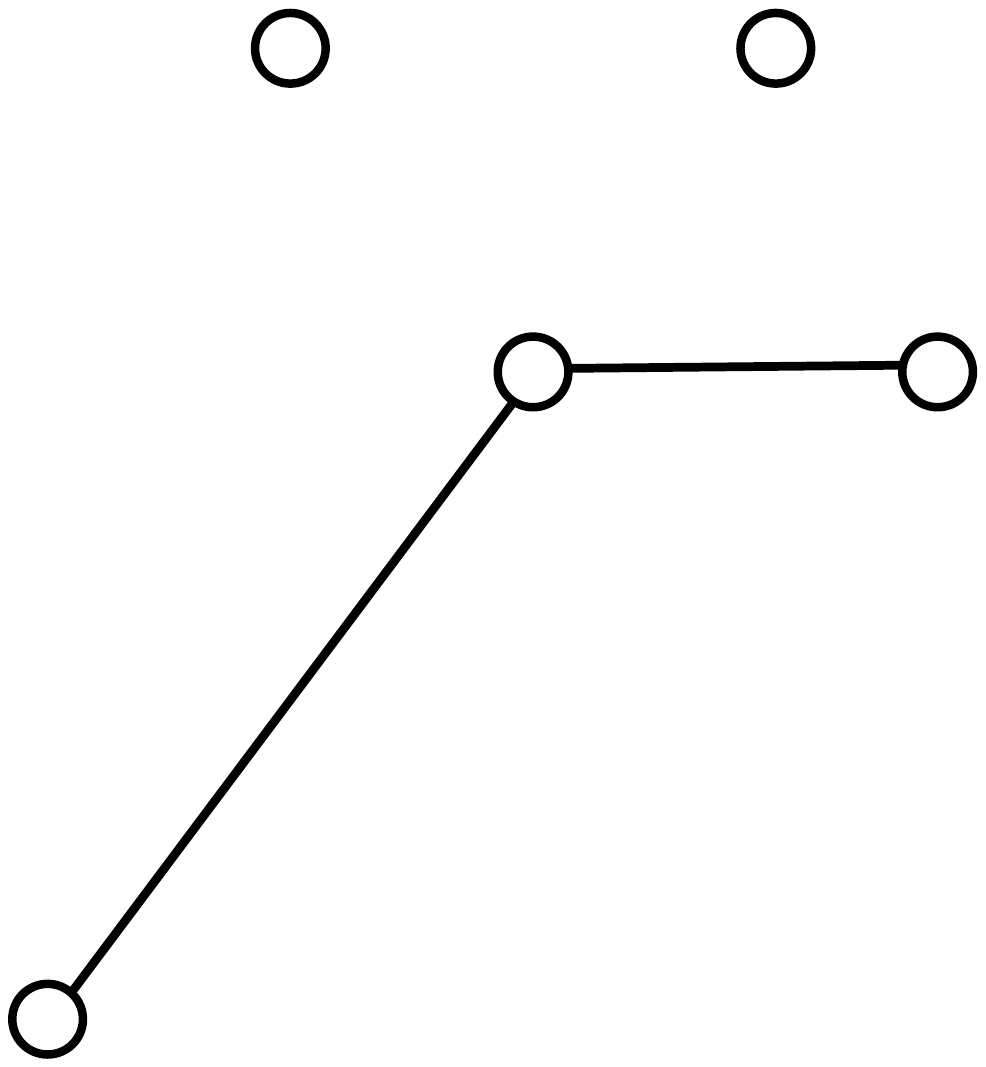

    \def\svgwidth{\columnwidth}
    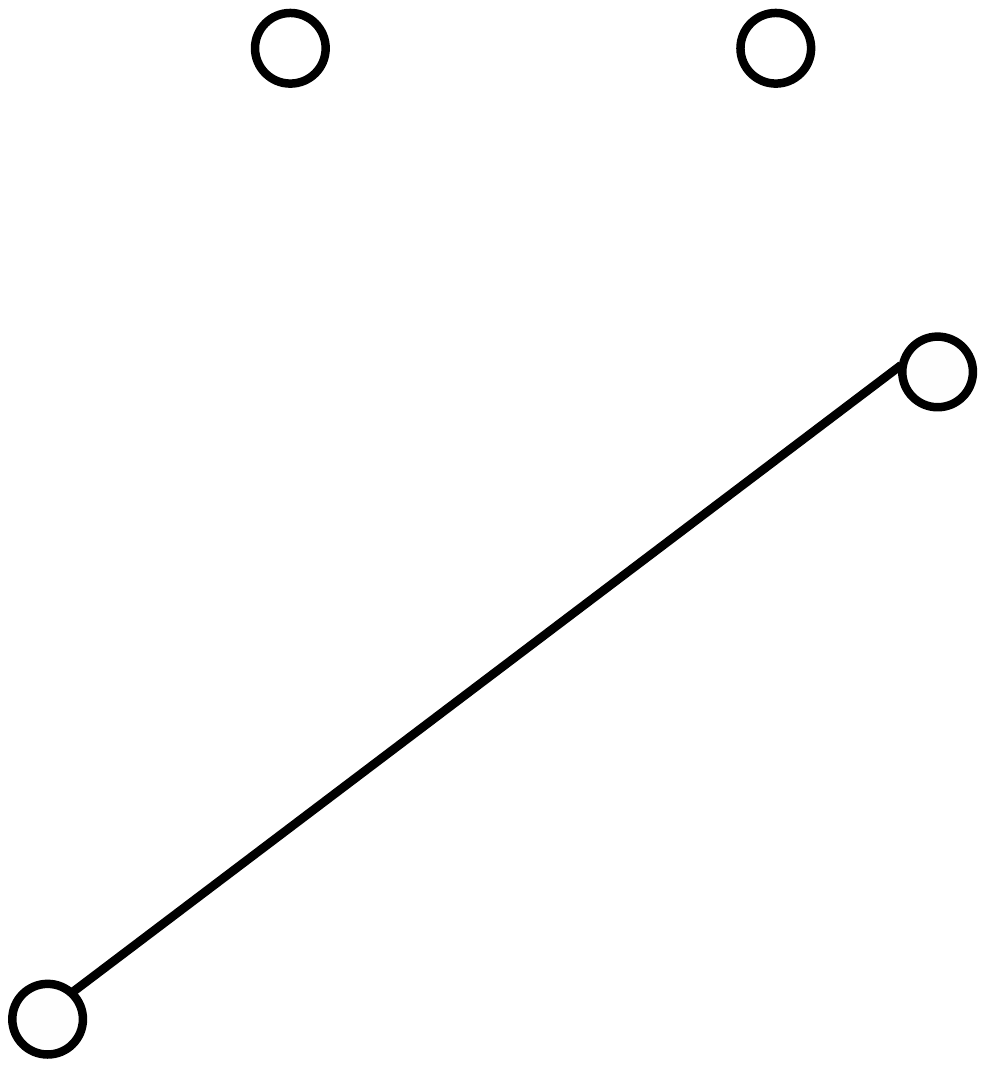

    \def\svgwidth{\columnwidth}
    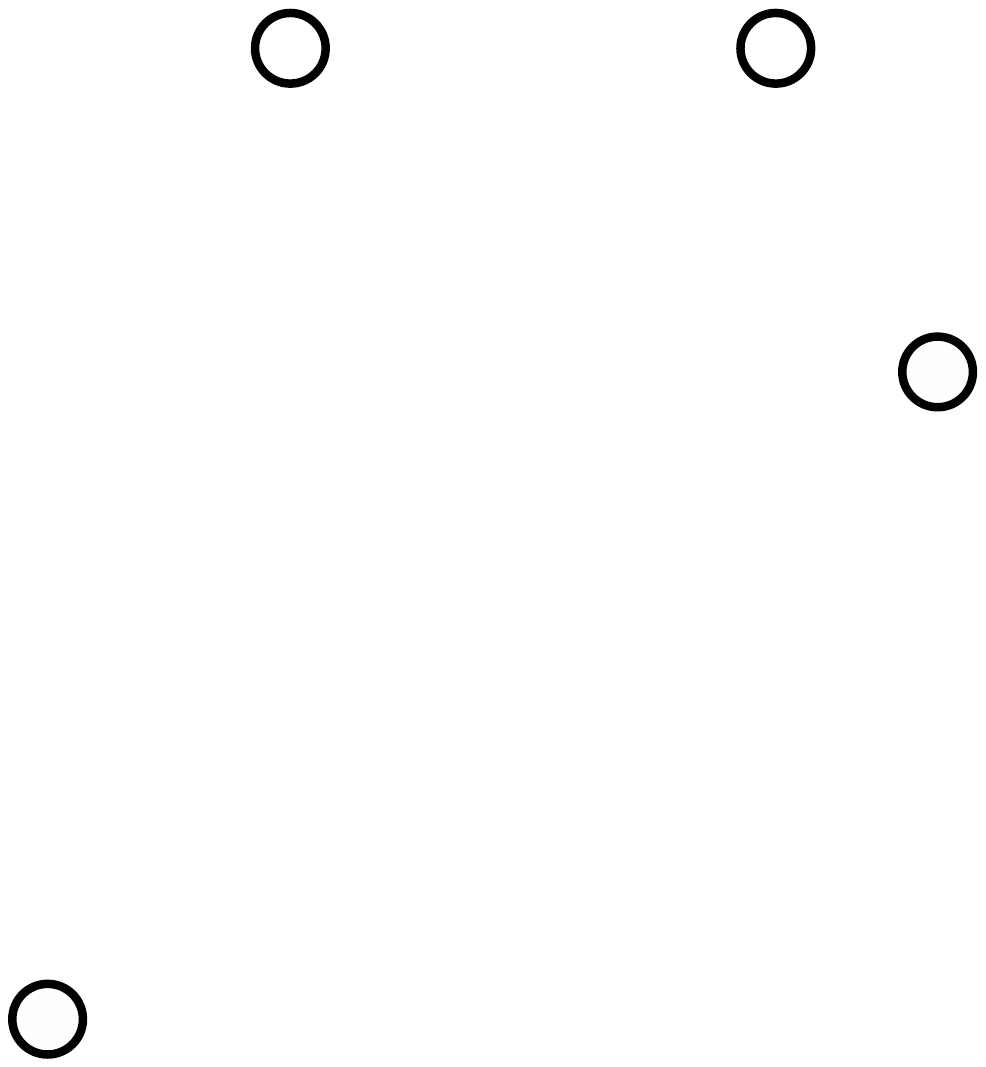

  }
  \caption{A star-shaped neighborhood of a periodic Julia branch point $v$ with two different admissible splittings with angles specified.}
  \end{subfigure}
  \begin{subfigure}{1\textwidth}
  \centering
  \resizebox{1\linewidth}{!}{
    \def\svgwidth{\columnwidth}
\begingroup%
  \makeatletter%
  \providecommand\color[2][]{%
    \errmessage{(Inkscape) Color is used for the text in Inkscape, but the package 'color.sty' is not loaded}%
    \renewcommand\color[2][]{}%
  }%
  \providecommand\transparent[1]{%
    \errmessage{(Inkscape) Transparency is used (non-zero) for the text in Inkscape, but the package 'transparent.sty' is not loaded}%
    \renewcommand\transparent[1]{}%
  }%
  \providecommand\rotatebox[2]{#2}%
  \newcommand*\fsize{\dimexpr\f@size pt\relax}%
  \newcommand*\lineheight[1]{\fontsize{\fsize}{#1\fsize}\selectfont}%
  \ifx\svgwidth\undefined%
    \setlength{\unitlength}{792bp}%
    \ifx\svgscale\undefined%
      \relax%
    \else%
      \setlength{\unitlength}{\unitlength * \real{\svgscale}}%
    \fi%
  \else%
    \setlength{\unitlength}{\svgwidth}%
  \fi%
  \global\let\svgwidth\undefined%
  \global\let\svgscale\undefined%
  \makeatother%
  \begin{picture}(1,0.77272727)%
    \lineheight{1}%
    \setlength\tabcolsep{0pt}%
    \put(0,0){\includegraphics[width=\unitlength,page=1]{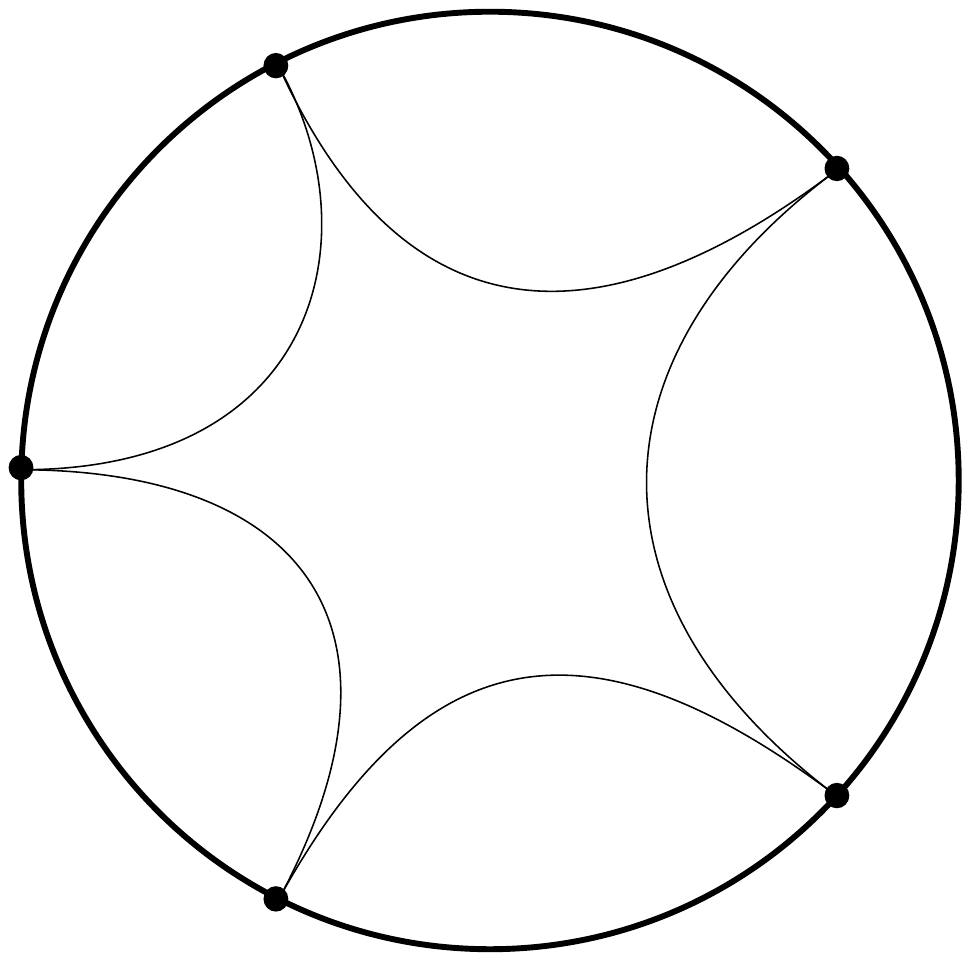}}%
    \put(0.15114024,0.36330933){\color[rgb]{0,0,0}\makebox(0,0)[lt]{\lineheight{1.25}\smash{\begin{tabular}[t]{l}{\LARGE $A_{40}$}\end{tabular}}}}%
    \put(0.33211838,0.09025614){\color[rgb]{0,0,0}\makebox(0,0)[lt]{\lineheight{1.25}\smash{\begin{tabular}[t]{l}{\LARGE $A_{01}$}\end{tabular}}}}%
    \put(0.69753348,0.15913013){\color[rgb]{0,0,0}\makebox(0,0)[lt]{\lineheight{1.25}\smash{\begin{tabular}[t]{l}{\LARGE $A_{12}$}\end{tabular}}}}%
    \put(0.70327295,0.59150604){\color[rgb]{0,0,0}\makebox(0,0)[lt]{\lineheight{1.25}\smash{\begin{tabular}[t]{l}{\LARGE $A_{23}$}\end{tabular}}}}%
    \put(0.33020522,0.66229325){\color[rgb]{0,0,0}\makebox(0,0)[lt]{\lineheight{1.25}\smash{\begin{tabular}[t]{l}{\LARGE $A_{34}$}\end{tabular}}}}%
  \end{picture}%
\endgroup%

    \def\svgwidth{\columnwidth}
\begingroup%
  \makeatletter%
  \providecommand\color[2][]{%
    \errmessage{(Inkscape) Color is used for the text in Inkscape, but the package 'color.sty' is not loaded}%
    \renewcommand\color[2][]{}%
  }%
  \providecommand\transparent[1]{%
    \errmessage{(Inkscape) Transparency is used (non-zero) for the text in Inkscape, but the package 'transparent.sty' is not loaded}%
    \renewcommand\transparent[1]{}%
  }%
  \providecommand\rotatebox[2]{#2}%
  \newcommand*\fsize{\dimexpr\f@size pt\relax}%
  \newcommand*\lineheight[1]{\fontsize{\fsize}{#1\fsize}\selectfont}%
  \ifx\svgwidth\undefined%
    \setlength{\unitlength}{792bp}%
    \ifx\svgscale\undefined%
      \relax%
    \else%
      \setlength{\unitlength}{\unitlength * \real{\svgscale}}%
    \fi%
  \else%
    \setlength{\unitlength}{\svgwidth}%
  \fi%
  \global\let\svgwidth\undefined%
  \global\let\svgscale\undefined%
  \makeatother%
  \begin{picture}(1,0.77272727)%
    \lineheight{1}%
    \setlength\tabcolsep{0pt}%
    \put(0,0){\includegraphics[width=\unitlength,page=1]{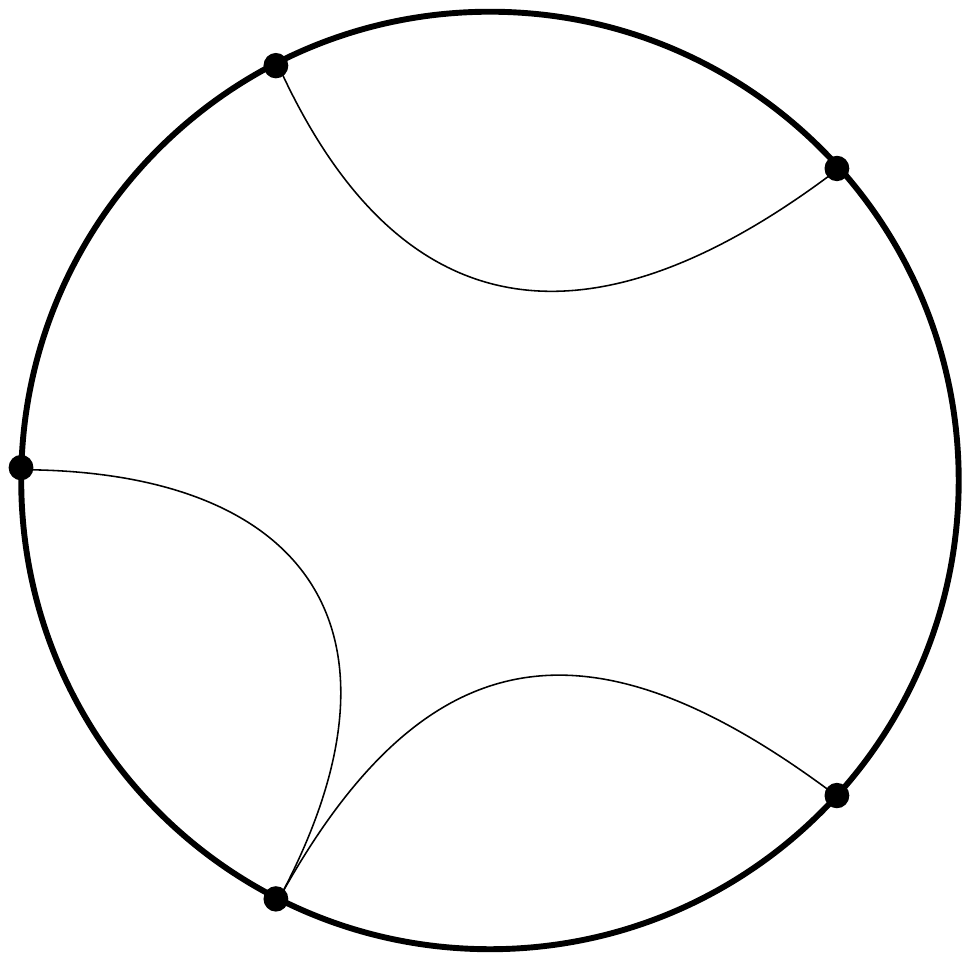}}%
    \put(0.15114024,0.36330933){\color[rgb]{0,0,0}\makebox(0,0)[lt]{\lineheight{1.25}\smash{\begin{tabular}[t]{l}{\LARGE $A_{40}$}\end{tabular}}}}%
    \put(0.33211838,0.09025614){\color[rgb]{0,0,0}\makebox(0,0)[lt]{\lineheight{1.25}\smash{\begin{tabular}[t]{l}{\LARGE $A_{01}$}\end{tabular}}}}%
    \put(0.69753348,0.15913013){\color[rgb]{0,0,0}\makebox(0,0)[lt]{\lineheight{1.25}\smash{\begin{tabular}[t]{l}{\LARGE $A_{12}$}\end{tabular}}}}%
    \put(0.70327295,0.59150604){\color[rgb]{0,0,0}\makebox(0,0)[lt]{\lineheight{1.25}\smash{\begin{tabular}[t]{l}{\LARGE $A_{23}$}\end{tabular}}}}%
    \put(0.33020522,0.66229325){\color[rgb]{0,0,0}\makebox(0,0)[lt]{\lineheight{1.25}\smash{\begin{tabular}[t]{l}{\LARGE $A_{34}$}\end{tabular}}}}%
    \put(0,0){\includegraphics[width=\unitlength,page=2]{L2.pdf}}%
  \end{picture}%
\endgroup%

    \def\svgwidth{\columnwidth}
\begingroup%
  \makeatletter%
  \providecommand\color[2][]{%
    \errmessage{(Inkscape) Color is used for the text in Inkscape, but the package 'color.sty' is not loaded}%
    \renewcommand\color[2][]{}%
  }%
  \providecommand\transparent[1]{%
    \errmessage{(Inkscape) Transparency is used (non-zero) for the text in Inkscape, but the package 'transparent.sty' is not loaded}%
    \renewcommand\transparent[1]{}%
  }%
  \providecommand\rotatebox[2]{#2}%
  \newcommand*\fsize{\dimexpr\f@size pt\relax}%
  \newcommand*\lineheight[1]{\fontsize{\fsize}{#1\fsize}\selectfont}%
  \ifx\svgwidth\undefined%
    \setlength{\unitlength}{792bp}%
    \ifx\svgscale\undefined%
      \relax%
    \else%
      \setlength{\unitlength}{\unitlength * \real{\svgscale}}%
    \fi%
  \else%
    \setlength{\unitlength}{\svgwidth}%
  \fi%
  \global\let\svgwidth\undefined%
  \global\let\svgscale\undefined%
  \makeatother%
  \begin{picture}(1,0.77272727)%
    \lineheight{1}%
    \setlength\tabcolsep{0pt}%
    \put(0,0){\includegraphics[width=\unitlength,page=1]{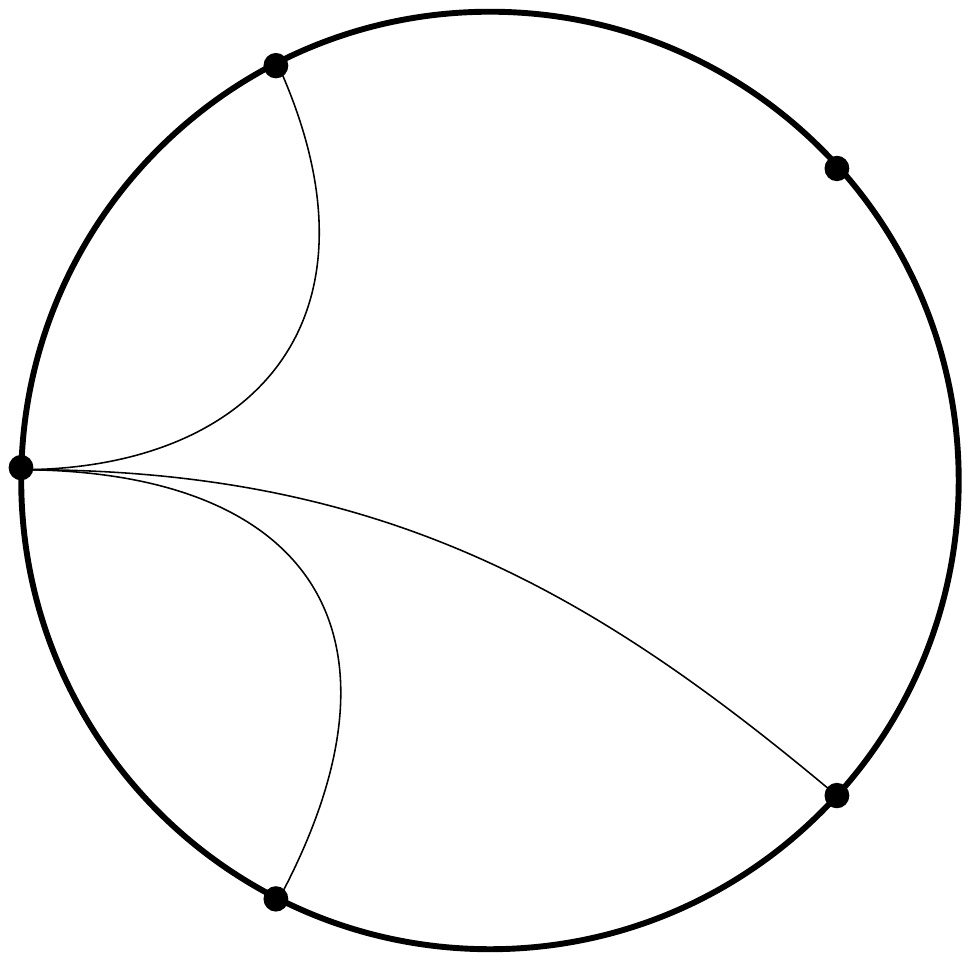}}%
    \put(0.15114024,0.36330933){\color[rgb]{0,0,0}\makebox(0,0)[lt]{\lineheight{1.25}\smash{\begin{tabular}[t]{l}{\LARGE $A_{40}$}\end{tabular}}}}%
    \put(0.33211838,0.09025614){\color[rgb]{0,0,0}\makebox(0,0)[lt]{\lineheight{1.25}\smash{\begin{tabular}[t]{l}{\LARGE $A_{01}$}\end{tabular}}}}%
    \put(0.69753348,0.15913013){\color[rgb]{0,0,0}\makebox(0,0)[lt]{\lineheight{1.25}\smash{\begin{tabular}[t]{l}{\LARGE $A_{12}$}\end{tabular}}}}%
    \put(0.70327295,0.59150604){\color[rgb]{0,0,0}\makebox(0,0)[lt]{\lineheight{1.25}\smash{\begin{tabular}[t]{l}{\LARGE $A_{23}$}\end{tabular}}}}%
    \put(0.33020522,0.66229325){\color[rgb]{0,0,0}\makebox(0,0)[lt]{\lineheight{1.25}\smash{\begin{tabular}[t]{l}{\LARGE $A_{34}$}\end{tabular}}}}%
    \put(0,0){\includegraphics[width=\unitlength,page=2]{L3.pdf}}%
  \end{picture}%
\endgroup%

  }
  \caption{The corresponding dual laminations generating the same equivalence relations on $\mathbb{S}^1$.}
  \end{subfigure}
  \caption{The split modification and dual laminations.}
  \label{fig:M}
\end{figure}

The dynamics are modified in $S$ so that each edge of $\tilde{S}$ is fixed.
The local degree function $\delta$ is defined to be the same as before the modification.
The angle function at $a_i$ is modified with the following rule (see Figure \ref{fig:M}):
\begin{itemize}
\item If $a_ia_j$ is an edge where $a_j$ is closer to $a_0$ than $a_i$, we set the angle of the tangent direction corresponding to $a_j$ to be $0$;
\item If $a_ia_j$ is an edge where $a_j$ is further to $a_0$ than $a_i$, we set the angle of the tangent direction corresponding to $a_j$ to be $0^+$ if $j<i$ and $0^-$ if $j>i$;
\item The other angles remain the same.
\end{itemize}

We also modified $H$ on the backward orbits of vertices in $S$ by pullback.
We will assume that $w \neq v \in \mathcal{V}$ is a preimage of $v$.
The general construction on backward orbits can be done by induction.

If $\delta(w) = 1$, then we remove the $1$-neighborhood of $w$ and glue back a copy of $\tilde S$.
Note that there is a unique way of gluing back $\tilde S$ by the dynamics.

If $\delta(w) \geq 2$, then the pullback is not unique. 
First, we modify the simplicial structure of $\tilde{S}$ by adding a vertex on the `midpoint' of each edge.
We call a degree $\delta(w)$ branched covering $f: \tilde{S}' \longrightarrow \tilde{S}$ {\em admissible} if the branch locus is contained in the midpoints of the edges of $\tilde{S}$. 
We call $\tilde{S}'$ an admissible $\delta(w)$ branched cover of $\tilde{S}$.
The pullback is constructed by removing the $1$-neighborhood of $w$ and then glueing back a copy of admissible $\delta(w)$ branched cover $\tilde S'$.
Similar to the degree $1$ case, the gluing is determined by the dynamics.

Note that in both cases, we may introduce some new vertices as the preimage of the full $1$-neighborhood of a periodic simple Julia branch point may not be in the Hubbard tree.
Those new vertices are defined to have local degree $1$, and the angle functions are defined by pullback.

Let $\mathcal{T}$ be the tree after the split modification over all periodic Julia branch vertices.
Removing vertices of $\mathcal{T}$ if necessary, we may assume that $f$ is minimal.
It then follows that

\begin{prop}\label{prop:eqaht}
Let $(H, \p)$ be a simplicial pointed Hubbard tree.
Let $f: (\mathcal{T}, \p) \longrightarrow (\mathcal{T}, \p)$ be the minimal angled tree map after performing admissible split modification over all periodic Julia branch vertices.
Then $f: (\mathcal{T}, \p) \longrightarrow (\mathcal{T}, \p)$ is admissible.
\end{prop}

We remark that, conjecturally,  these different splittings and pullbacks will all result in different accesses of points on $\partial \mathcal{H}$ and create complicated topology of $\overline{\mathcal{H}_d}$.
To prove Theorem \ref{thm: eq}, we only need the existence of one admissible splitting.
To prove Theorem \ref{thm:sb}, we only need two different admissible splittings (see \S \ref{sec:sb}).
It would be interesting to know whether such combinatorial operation can tell us about the `complexity' of $\partial \mathcal{H}_d$.

\begin{prop}\label{prop:lam}
For any admissible split modification $(\mathcal{T},\p)$ of $(H,\p)$, the dual laminations $\mathcal{L}_\mathcal{T}$ and $\mathcal{L}_H$ generate the same equivalence relations on $\mathbb{S}^1$.
\end{prop}
\begin{proof}
By induction and pullback, it suffices to consider the splitting at a periodic Julia branch point $v \in H$. 
Without loss of generality, we may assume $v$ is fixed.
Let $a_0,..., a_k$ be the vertices adjacent to $v$.
Then there are $k+1$ angles $A_0,..., A_k$ landing at $v$, where $A_i$ corresponds to the access between the Fatou components of $a_i$ and $a_{i+1}$.

Let $S$ be the subtree containing $a_0, ..., a_k$ after the modification.
By the dynamics, $A_i$ is the only angle landing at the right side of the union of directed edges from $a_i$ to $a_{i+1}$ (see Figure \ref{fig:M}).
Since the angle at the vertex $a_i$ between an incident edge in $S$ and an incident edge outside of $S$ is strictly bigger than $0$, $A_i$ does not land at any other edges.
Therefore, $A_0,..., A_k$ form an equivalence class of the equivalence relation generated by the lamination for $\mathcal{T}$.
The claim now follows.
\end{proof}

\subsection*{Carath\'eodory convergence}
A {\em disk} is a simply connected and connected open subset in $\C$.
It is said to be hyperbolic if it is not $\C$.
For a sequence of pointed disks $(U_n, u_n)$, we say $(U_n, u_n)$ converges to $(U, u)$ in Carath\'eodory topology if 
\begin{itemize}
\item $u_n \to u$;
\item for any compact set $K \subseteq U$, $K \subseteq U_n$ for all sufficiently large $n$;
\item for any open connected set $N$ containing $u$, if $N \subseteq U_n$ for all sufficiently large $n$, then $N \subseteq U$.
\end{itemize}
Equivalently, the convergence means that $u_n \to u$ and for any subsequence such that $(\hat\C - U_n) \to K$ in the Hausdorff topology on compact sets of the sphere, $U$ is equal to the component of $\hat\C - K$ containing $u$ (see \cite[\S 5.1]{McM94}).

Similarly, we say a sequence of proper holomorphic maps between pointed disks $f_n: (U_n, u_n) \longrightarrow (V_n, v_n)$ converges to $f:(U, u) \longrightarrow (V, v)$ if 
\begin{itemize}
\item $(U_n, u_n), (V_n,v_n)$ converge to $(U, u), (V, v)$ in Carath\'eodory topology;
\item for all sufficiently large $n$, $f_n$ converges to $f$ uniformly on compact subsets of $U$.
\end{itemize}

We have the following compactness result:
\begin{theorem}[\cite{McM94}, Theorem 5.2]\label{thm:cmpc}
The set of disks $(U_n, 0)$ containing $B(0,r)$ for some $r>0$ is compact in Carath\'eodory topology.
\end{theorem}

The definition of Carath\'eodory convergence naturally generalizes to simply connected and connected open subset of $\hat \C$.
Recall that by Lemma \ref{lem:ac}, a sequence of degree $d$ rational maps $R_n$ converges uniformly on $\hat\C$ to a rational map $R$ if and only if $R_n$ converges algebraically to $R$ and $R$ has degree $d$.
The following lemma is very useful in studying degenerations of quasi post-critically finite Blaschke products.
\begin{lem}\label{lem:key}
Let $R_n: \hat \C \longrightarrow \hat \C$ be a sequence of degree $d$ rational maps converging uniformly on $\hat\C$ to $R: \hat\C \longrightarrow \hat \C$.
Let $U_n$ be a sequence of invariant hyperbolic disks for $R_n$, and $x_n \in U_n$.
If there exists $K$ so that $d_{U_n}(x_n, R_n(x_n)) \leq K$ for all $n$, then after passing to a subsequence, either
\begin{itemize}
\item $\lim_{n\to\infty} x_n = x$ and $x$ is fixed by $R$; or
\item $(U_n, x_n)$ converges in Carath\'eodory topology to $(U, x)$. Thus $R_n: (U_n, x_n) \longrightarrow (U_n, R_n(x_n))$ converges to $R:(U, x)\longrightarrow (U, R(x))$.
\end{itemize}
\end{lem}
\begin{proof}
After passing to a subsequence, we may assume $x_n \to x$.
We may assume $U_n \subseteq \C$ for all sufficiently large $n$ and $x\in \C$.
Suppose $x$ is not fixed by $R$, we claim there exists a Euclidean ball of radius $r$ so that $B(x_n, r) \subseteq U_n$ for all sufficiently large $n$.
Suppose not, then the hyperbolic metric $\rho_{U_n}(x_n)|dz|$ at $x_n$ is going to infinity by Lemma \ref{lem:hme}.
Since $R_n$ converges to $R$ uniformly and $x$ is not fixed, for sufficiently large $n$, the Euclidean distance between $x_n$ and $R_n(x_n)$ is bounded below from $0$.
Thus, $d_{U_n}(x_n, R_n(x_n))$ is unbounded which is a contradiction.
Therefore, after passing to a subsequence, the pointed disk $(U_n, x_n)$ converges in Carath\'eodory topology to $(U, x)$ by Theorem \ref{thm:cmpc}.
\end{proof}

We remark that in the first case, the pointed disks $(U_n, x_n)$ usually diverges in the space of pointed disks under the Carath\`eodory topology.

A typical example of a divergent pointed disks can be constructed as follows.
Let $x_n = 0$ and let $U_n = B(-1+1/n, 1) \cup B(1-1/n, 1)$ be the union of two unit balls centered at $-1+1/n$ and $1-1/n$.
As $n \to \infty$, $\hat\C - U_n$ converges in Hausdorff topology to $K = \hat \C - (B(-1,1) \cup B(1,1))$ which contains limit point $0 = \lim x_n$, so $(U_n, x_n)$ diverges.

As an application of the Carath\'eodory limits, we have
\begin{prop}\label{prop:gfs}
If $f_n\in \BP_d$ is a quasi post-critically finite sequence, then after passing to a subsequence, the corresponding polynomials $P_n$ converge to a geometrically finite polynomial $\hat P$.
\end{prop}
\begin{proof}
Let $\tilde c_n$ be a critical point of $f_n$ and let $\tilde x_n$ be an iterate of $\tilde c_n$ with 
$$
d_{\D}(\tilde x_n, f_n^{q}(\tilde x_n)) \leq K.
$$
Let $x_n$ and $c_n$ be the corresponding points for $P_n$.
After passing to a subsequence, we assume $P_n \to P$, $x_n \to x_\infty \in \C$ and $c_n \to c_\infty \in \C$.

Let $U_n$ be the bounded Fatou component of $P_n$.
If $x_\infty$ is fixed by $P^q$, then $c$ is pre-periodic.
If $x_\infty$ is not fixed by $P^q$, by Lemma, \ref{lem:key}, after passing to a subsequence, the pointed disk $P_n: (U_n, x_n) \longrightarrow (U_n, P^q_n(x_n))$ converges to $P:(U_\infty, x_\infty)\longrightarrow (U_\infty, P^q(x_\infty))$.
So $U_\infty$ is contained in the Fatou set.
Hence $c_\infty$ is in the Fatou set.

Since any critical point of $P$ is approximated by critical points of $P_n$, we conclude that $P$ is geometrically finite.
\end{proof}

\subsection*{Construction of geometrically finite polynomials}
Let $(H, \p)$ be a marked simplicial pointed Hubbard tree and $f: (\mathcal{T}, \p) \longrightarrow (\mathcal{T}, \p)$ be the angled tree map after the admissible split modification.
Let $\mathcal{F}$ be a normalized mapping scheme on $\mathcal{S}_\p$ if $\delta(\p) \geq 2$.
Let $f_n \in \BP_d$ be as in Proposition \ref{prop:rt} and $P_n = f_n \sqcup z^d$ be the corresponding polynomials.
By Proposition \ref{prop:gfs}, after passing to a subsequence, $P_n$ converges to a geometrically finite polynomial $\hat P$.
Denote by $U_n$ the bounded Fatou component of $P_n$.
Using the conjugacy between $\D$ and $U_n$, the quasi-invariant tree for $f_n$ corresponds to a quasi-invariant tree for $P_n$ in $U_n$. 
Abusing the notations, we denote this quasi-invariant tree for $P_n$ as $
\T_n \subseteq U_n$.

We shall now prove that $\hat P$ has the desired property in Proposition \ref{prop:rsh}.
Recall a vertex $v\in \mathcal{V}$ is said to be a {\em Fatou point} if $v$ is eventually mapped to a critical periodic orbit.
It is called a {\em Julia point} otherwise.

\begin{lem}\label{lem:cfm}
Let $v \in \mathcal{V}$ be a periodic Fatou point of period $q$ and $v_n$ be the corresponding point for $P_n$. After passing to a subsequence, $(U_n, v_n)$ converges in Carath\'eodory topology to $(U_{v,\infty}, v_\infty)$ and 
$$
\hat P^q: (U_{v,\infty}, v_\infty) \longrightarrow (U_{v,\infty}, \hat P^q(v_\infty))
$$
is conformally conjugate to the first return rescaling limit
$$
F^q_v: (\D_v, 0) \longrightarrow (\D_v, F^q_v(0)).
$$
\end{lem}
\begin{proof}
The proof is similar to Proposition \ref{prop:gfs}.
Let $\tilde c_n$ be a critical point of $f^q_n$ in the cluster associated to $v$.
Let $c_n$ be the corresponding critical point for $P_n$.
After passing to a subsequence, we assume $c_n \to c_\infty$ and $P_n \to \hat P$.

If $\hat P^q(c_\infty) \neq c_\infty$, since $d_{U_n}(c_n, v_n)$ is bounded, the claim on Carath\'eodory limit follows by Lemma \ref{lem:key}.
Since each $P^q_n$ is conformally conjugate on $U_n$ to $f^q_n$, the limit $\hat P^q$ on $U_{v,\infty}$ is conformally conjugate to $F^q_v$.

If $\hat P^q(c_\infty) = c_\infty$, then $c_\infty$ is a super-attracting fixed point for $\hat P^q$.
Thus there exists an open set $c_\infty \in U$ so that $\hat P^q(U)$ is compactly contained in $U$.
Since $P_n$ converges to $\hat P$, for sufficiently large $n$, the iterates of $P_n$ are normal on $U$.
Thus $U \subset U_n$ for sufficiently large $n$.
The rest of the argument is the same as the previous case.
\end{proof}

By pulling back, we immediately get
\begin{cor}\label{cor:cfm}
Let $v \in \mathcal{V}$ be a Fatou point. 
Then after passing to a subsequence, $(U_n, v_n)$ converges in Carath\'eodory topology to $(U_{v,\infty}, v_\infty)$ and
$$
\hat P: (U_{v,\infty}, v_\infty) \longrightarrow (U_{f(v),\infty}, \hat P(v_\infty))
$$
is conformally conjugate to the rescaling limit
$$
F_v: (\D_v, 0) \longrightarrow (\D_{f(v)}, F_v(0)).
$$
\end{cor}

By Lemma \ref{lem:cfm} and Corollary \ref{cor:cfm}, if $v \in \mathcal{V}$ is a Fatou point, $U_{v,\infty}$ is contained in a Fatou component $\Omega_v$ of $\hat P$.

\begin{cor}\label{cor:afc}
Suppose $\delta(\p) \geq 2$. Then $U_{\p,\infty} = \Omega_\p$ and the Fatou component $\Omega_\p$ is attracting.
If $v\neq \p\in \mathcal{V}$ is a periodic Fatou point, then the Fatou component $\Omega_v$ is parabolic.
\end{cor}
\begin{proof}
Since $\delta(\p) \geq 2$, by Lemma \ref{lem:cfm}, $\Omega_\p$ is attracting.
Since $U_{\p, \infty}$ is invariant under $\hat P$ and contains the attracting fixed point, $U_{\p, \infty} = \Omega_\p$.
Since $\hat P \in \overline{\PH_d}$, there exists at most one attracting Fatou component, so $\Omega_v$ is not attracting as $v\neq \p$.
So $\Omega_v$ is parabolic by Proposition \ref{prop:gfs}.
\end{proof}

Recall $T$ is a {\em $\nu$ star-shaped tree} if $T$ is a union of the $\nu$ arcs $[c, x_i]$, $i=1,..., \nu$, glued at $c$.
We say $T$ is an {\em open $\nu$ star-shaped tree} if $T$ is a union of the $\nu$ arcs $[c, x_i)$, $i=1,..., \nu$, glued at $c$.

After passing to a subsequence, we let $\T_\infty$ be the Hausdorff limit of $\T_n$.
We show the following (see Figure \ref{fig:HD}):

\begin{lem}\label{lem:fr}
Let $v \in \mathcal{V}$ be a periodic Fatou point of period $q$. Then $\T_\infty \cap U_{v,\infty}$ is an open star-shaped tree where each limit point on $\partial U_{v, \infty}$ is pre-periodic.
Each tangent direction $a \in T_v\mathcal{T}$ gives one limit point $s_a \in \partial U_{v,\infty}$.
The limiting graph $\T_\infty$ gives an attracting direction for $s_a$ if $a$ corresponds to the direction of $\p$ and a repelling direction otherwise.
\end{lem}
\begin{proof}
We first note that the quasi-invariant trees converge to an open star-shaped tree in $\D_v$.
For any compact subset $K \subseteq U_{v,\infty}$, $\T_\infty \cap K$ is a star-shaped tree by Lemma \ref{lem:cfm}.
Let $x$ be a limit point of $\T_\infty \cap U_{v, \infty}$ in $\partial U_{v,\infty}$.
Then there exists a sequence $x_n \to x$ and $x_n \in \T_n \cap U_{v,\infty}$.

After passing to a subsequence, we may assume $x_n$ is either quasi-fixed or quasi pre-fixed by $\hat P^k$ for some large $k$ dividing $q$.
If $x_n$ is quasi-fixed, then $d_{U_n}(x_n, P^k_n(x_n)) \leq K$. 
Since $x_n \to \partial U_{v, \infty}$, the hyperbolic metric $\rho_{U_n}(x_n)|dz|$ goes to infinity by Lemma \ref{lem:hme}.
Thus the Euclidean distance $d_{\R^2}(x_n, P^k_n(x_n))$ goes to $0$ and we conclude that $x$ is fixed by $\hat P^k$.
If $x_n$ is quasi pre-fixed, then the same argument shows $x$ is prefixed.
Since the fixed points of $\hat P^k$ are discrete, each tangent gives one limit point.

The last statement follows from comparing with the dynamics of $F_v^q$ on $\D_v$.
\end{proof}

We remark that different tangent directions of $\mathcal{T}$ at $v$ may give the same limit point, thus the closure $\overline{\T_\infty \cap U_{v,\infty}}$ is a graph that is not necessarily a tree (see Figure \ref{fig:HD}).
The same proof also gives the following more general statement, where $x_n$ are allowed to be on the edges of $\T_n$:
\begin{lem}\label{lem:gfr}
Let $x_n \in \T_n \cap U_n$ be quasi periodic with period $q$ that converges to $x$. 
If $x$ is not fixed by $\hat P^q$, then $(U_n, x_n)$ converges to $(U, x)$ in Carath\'eodory topology and $\T_\infty \cap U$ is an open star-shaped tree whose limit points on $\partial U$ are pre-periodic.
Each tangent direction $a \in T_x\mathcal{T}$ gives one limit point $s_a$.
The limiting graph $\T_\infty$ gives an attracting direction for $s_a$ if $a$ corresponds to the direction of $\p$ and a repelling direction otherwise.
\end{lem}

\begin{prop}\label{prop:dfc}
Let $v\neq w \in \mathcal{V}$ be Fatou points. Then the Fatou components $\Omega_v \neq \Omega_w$.
\end{prop}
\begin{proof}
Interchanging $v$ and $w$, we may assume $v$ is closer to $\p$ than $w$.
Assume $v, w$ are periodic, the case when $v, w$ are strictly pre-periodic is proved similarly.
After passing to an iterate, we assume $v$ and $w$ are fixed.
If $v = \p$, then the statement follows from Corollary \ref{cor:afc}.
Thus we assume $v, w\neq \p$.

Let $s_{v, \infty} \in \partial U_{v, \infty}$ and $s_{w, \infty} \in \partial U_{v, \infty}$ be the corresponding parabolic fixed points.
Let $E_v, E_w$ be the incident edges at $v$ and $w$ in the direction of $\p$.
Let $s_{v,n}\in E_{v,n} \subseteq \T_n$ (and $s_{w,n}$) be a sequence that converges to $s_{v, \infty}$ (and $s_{w, \infty}$ respectively).
Suppose for contradiction that $\Omega_v = \Omega_w$. Then $s_{v, \infty} = s_{w, \infty}$ and $U_v, U_w$ are in the same attracting petal.
We consider two cases.

\begin{figure}[ht]
  \centering
  \resizebox{1\linewidth}{!}{
    \def\svgwidth{\columnwidth}
\begingroup%
  \makeatletter%
  \providecommand\color[2][]{%
    \errmessage{(Inkscape) Color is used for the text in Inkscape, but the package 'color.sty' is not loaded}%
    \renewcommand\color[2][]{}%
  }%
  \providecommand\transparent[1]{%
    \errmessage{(Inkscape) Transparency is used (non-zero) for the text in Inkscape, but the package 'transparent.sty' is not loaded}%
    \renewcommand\transparent[1]{}%
  }%
  \providecommand\rotatebox[2]{#2}%
  \newcommand*\fsize{\dimexpr\f@size pt\relax}%
  \newcommand*\lineheight[1]{\fontsize{\fsize}{#1\fsize}\selectfont}%
  \ifx\svgwidth\undefined%
    \setlength{\unitlength}{841.88976378bp}%
    \ifx\svgscale\undefined%
      \relax%
    \else%
      \setlength{\unitlength}{\unitlength * \real{\svgscale}}%
    \fi%
  \else%
    \setlength{\unitlength}{\svgwidth}%
  \fi%
  \global\let\svgwidth\undefined%
  \global\let\svgscale\undefined%
  \makeatother%
  \begin{picture}(1,0.70707071)%
    \lineheight{1}%
    \setlength\tabcolsep{0pt}%
    \put(0.3309618,0.30327303){\color[rgb]{0,0,0}\makebox(0,0)[lt]{\lineheight{1.25}\smash{\begin{tabular}[t]{l}{\LARGE $s_{w,n}$}\end{tabular}}}}%
    \put(0.3720644,0.40441984){\color[rgb]{0,0,0}\makebox(0,0)[lt]{\lineheight{1.25}\smash{\begin{tabular}[t]{l}{\LARGE $s_{v,n}$}\end{tabular}}}}%
    \put(0,0){\includegraphics[width=\unitlength,page=1]{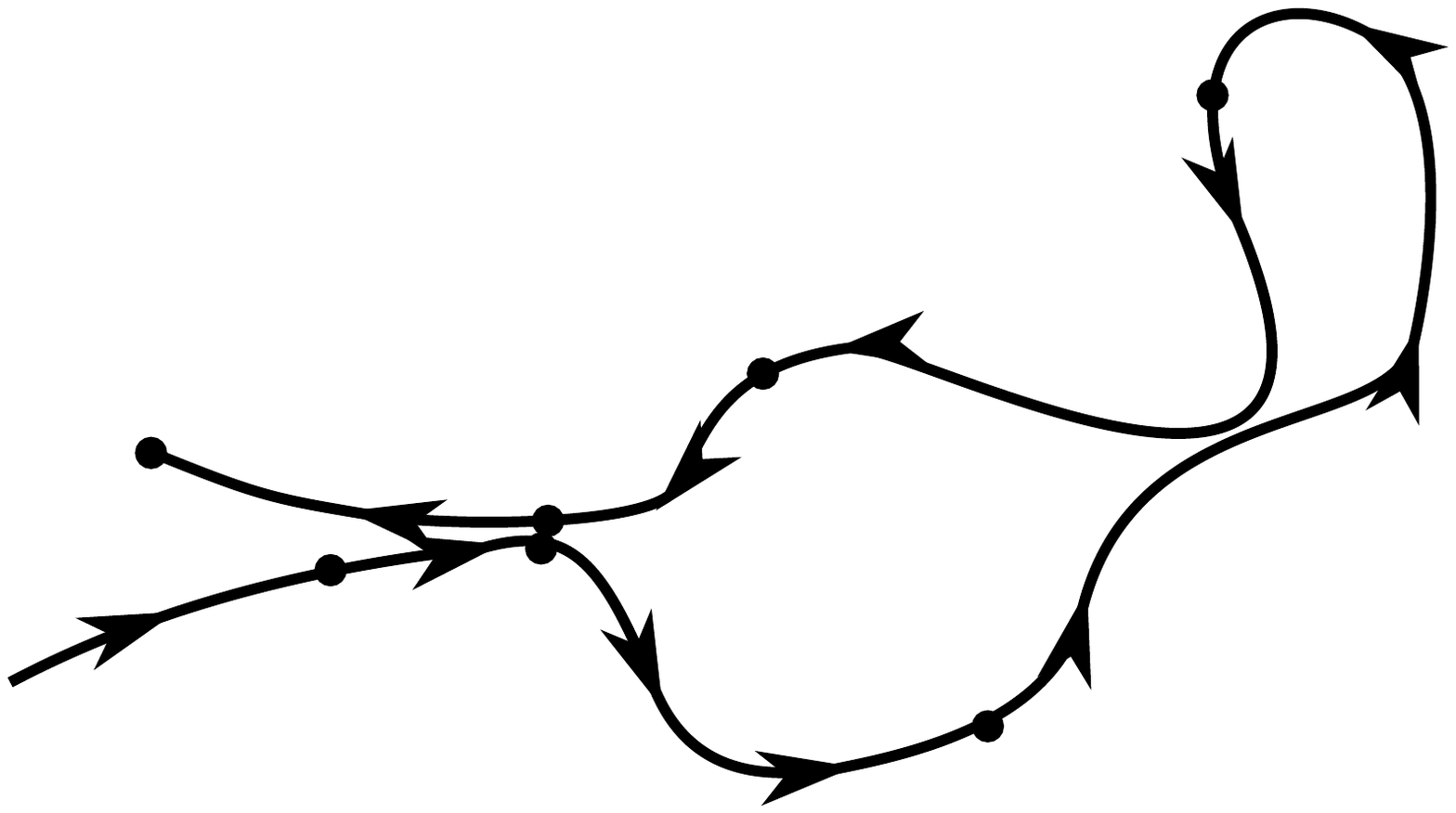}}%
    \put(0.1360856,0.43257519){\color[rgb]{0,0,0}\makebox(0,0)[lt]{\lineheight{1.25}\smash{\begin{tabular}[t]{l}{\LARGE $p_n$}\end{tabular}}}}%
    \put(0.50852216,0.4878811){\color[rgb]{0,0,0}\makebox(0,0)[lt]{\lineheight{1.25}\smash{\begin{tabular}[t]{l}{\LARGE $v_n$}\end{tabular}}}}%
    \put(0.24721969,0.29731568){\color[rgb]{0,0,0}\makebox(0,0)[lt]{\lineheight{1.25}\smash{\begin{tabular}[t]{l}{\LARGE $w_n$}\end{tabular}}}}%
    \put(0,0){\includegraphics[width=\unitlength,page=2]{TI.pdf}}%
  \end{picture}%
\endgroup%

    \def\svgwidth{\columnwidth}
\begingroup%
  \makeatletter%
  \providecommand\color[2][]{%
    \errmessage{(Inkscape) Color is used for the text in Inkscape, but the package 'color.sty' is not loaded}%
    \renewcommand\color[2][]{}%
  }%
  \providecommand\transparent[1]{%
    \errmessage{(Inkscape) Transparency is used (non-zero) for the text in Inkscape, but the package 'transparent.sty' is not loaded}%
    \renewcommand\transparent[1]{}%
  }%
  \providecommand\rotatebox[2]{#2}%
  \newcommand*\fsize{\dimexpr\f@size pt\relax}%
  \newcommand*\lineheight[1]{\fontsize{\fsize}{#1\fsize}\selectfont}%
  \ifx\svgwidth\undefined%
    \setlength{\unitlength}{841.88976378bp}%
    \ifx\svgscale\undefined%
      \relax%
    \else%
      \setlength{\unitlength}{\unitlength * \real{\svgscale}}%
    \fi%
  \else%
    \setlength{\unitlength}{\svgwidth}%
  \fi%
  \global\let\svgwidth\undefined%
  \global\let\svgscale\undefined%
  \makeatother%
  \begin{picture}(1,0.70707071)%
    \lineheight{1}%
    \setlength\tabcolsep{0pt}%
    \put(0,0){\includegraphics[width=\unitlength,page=1]{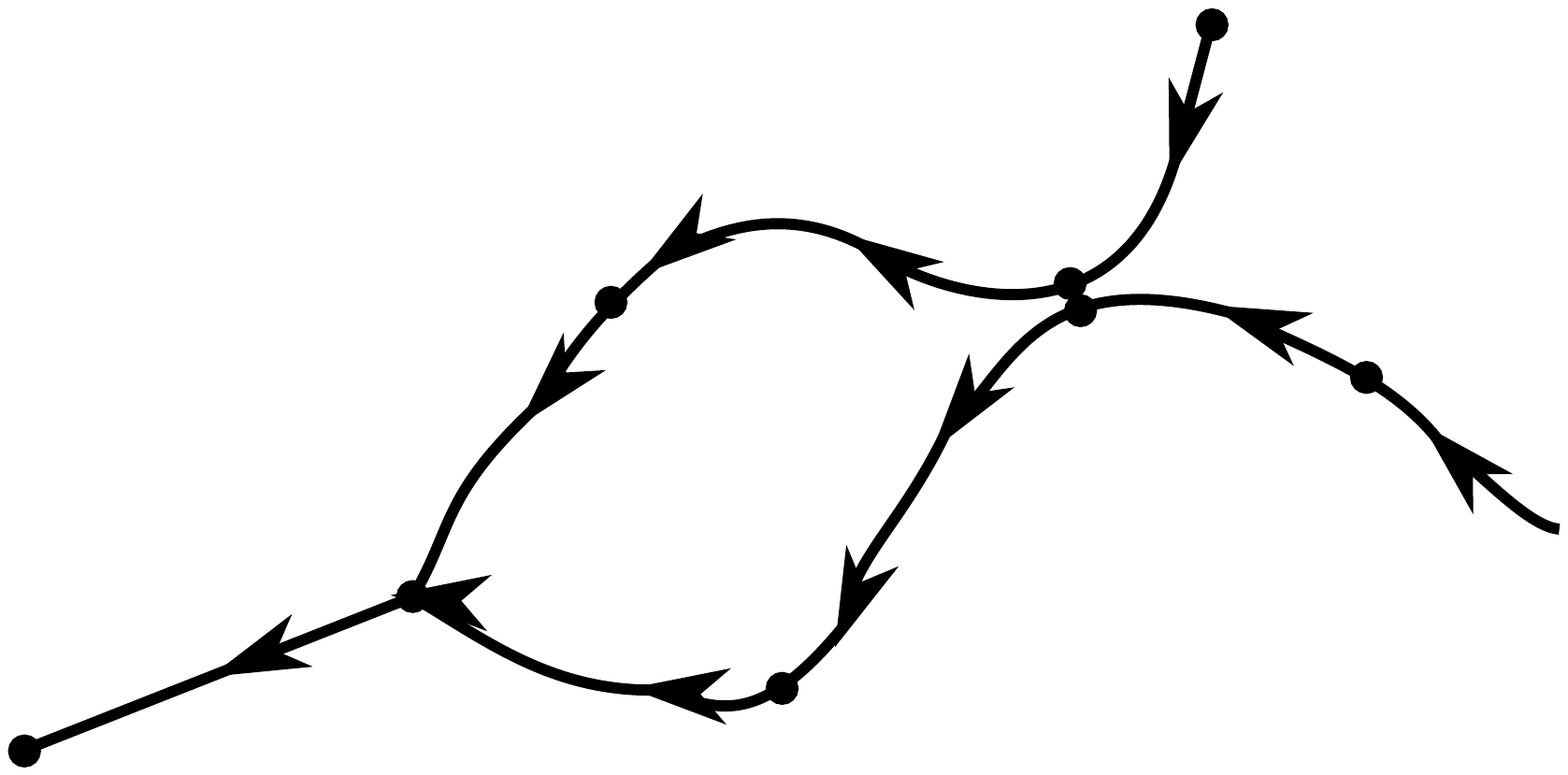}}%
    \put(0.04708794,0.18019479){\color[rgb]{0,0,0}\makebox(0,0)[lt]{\lineheight{1.25}\smash{\begin{tabular}[t]{l}{\LARGE$p_n$}\end{tabular}}}}%
    \put(0.2647106,0.25655368){\color[rgb]{0,0,0}\makebox(0,0)[lt]{\lineheight{1.25}\smash{\begin{tabular}[t]{l}{\LARGE$b_n$}\end{tabular}}}}%
    \put(0.6872294,0.44108746){\color[rgb]{0,0,0}\makebox(0,0)[lt]{\lineheight{1.25}\smash{\begin{tabular}[t]{l}{\LARGE$s_{v,n}$}\end{tabular}}}}%
    \put(0.61850654,0.52762745){\color[rgb]{0,0,0}\makebox(0,0)[lt]{\lineheight{1.25}\smash{\begin{tabular}[t]{l}{\LARGE$s_{w,n}$}\end{tabular}}}}%
    \put(0.77631473,0.67525453){\color[rgb]{0,0,0}\makebox(0,0)[lt]{\lineheight{1.25}\smash{\begin{tabular}[t]{l}{\LARGE$w_n$}\end{tabular}}}}%
    \put(0.8679453,0.44872335){\color[rgb]{0,0,0}\makebox(0,0)[lt]{\lineheight{1.25}\smash{\begin{tabular}[t]{l}{\LARGE$v_n$}\end{tabular}}}}%
  \end{picture}%
\endgroup%

    \def\svgwidth{\columnwidth}
\begingroup%
  \makeatletter%
  \providecommand\color[2][]{%
    \errmessage{(Inkscape) Color is used for the text in Inkscape, but the package 'color.sty' is not loaded}%
    \renewcommand\color[2][]{}%
  }%
  \providecommand\transparent[1]{%
    \errmessage{(Inkscape) Transparency is used (non-zero) for the text in Inkscape, but the package 'transparent.sty' is not loaded}%
    \renewcommand\transparent[1]{}%
  }%
  \providecommand\rotatebox[2]{#2}%
  \newcommand*\fsize{\dimexpr\f@size pt\relax}%
  \newcommand*\lineheight[1]{\fontsize{\fsize}{#1\fsize}\selectfont}%
  \ifx\svgwidth\undefined%
    \setlength{\unitlength}{841.88976378bp}%
    \ifx\svgscale\undefined%
      \relax%
    \else%
      \setlength{\unitlength}{\unitlength * \real{\svgscale}}%
    \fi%
  \else%
    \setlength{\unitlength}{\svgwidth}%
  \fi%
  \global\let\svgwidth\undefined%
  \global\let\svgscale\undefined%
  \makeatother%
  \begin{picture}(1,0.70707071)%
    \lineheight{1}%
    \setlength\tabcolsep{0pt}%
    \put(0,0){\includegraphics[width=\unitlength,page=1]{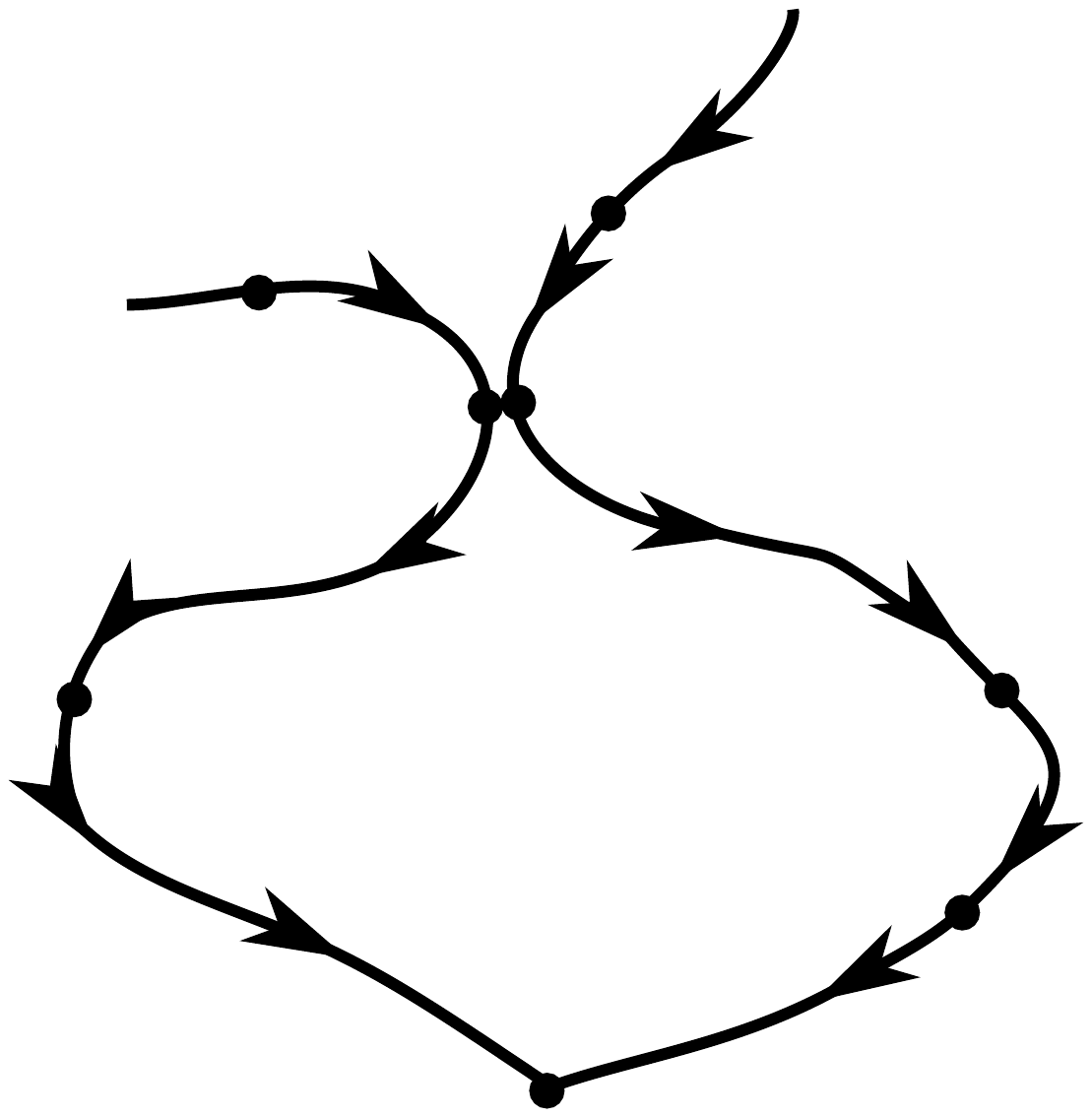}}%
    \put(0.5077862,0.10256336){\color[rgb]{0,0,0}\makebox(0,0)[lt]{\lineheight{1.25}\smash{\begin{tabular}[t]{l}{\LARGE$p_n$}\end{tabular}}}}%
    \put(0.51414942,0.48053952){\color[rgb]{0,0,0}\makebox(0,0)[lt]{\lineheight{1.25}\smash{\begin{tabular}[t]{l}{\LARGE$s_{v,n}$}\end{tabular}}}}%
    \put(0.57905435,0.56453422){\color[rgb]{0,0,0}\makebox(0,0)[lt]{\lineheight{1.25}\smash{\begin{tabular}[t]{l}{\LARGE$v_n$}\end{tabular}}}}%
    \put(0.39579323,0.47544892){\color[rgb]{0,0,0}\makebox(0,0)[lt]{\lineheight{1.25}\smash{\begin{tabular}[t]{l}{\LARGE$s_{w,n}$}\end{tabular}}}}%
    \put(0.35379588,0.55562571){\color[rgb]{0,0,0}\makebox(0,0)[lt]{\lineheight{1.25}\smash{\begin{tabular}[t]{l}{\LARGE$w_n$}\end{tabular}}}}%
    \put(0,0){\includegraphics[width=\unitlength,page=2]{TI3.pdf}}%
  \end{picture}%
\endgroup%

  }
  \caption{A schematic diagram for the limit of quasi-invariant trees $\T_n$. The arrows represent the dynamics by $P_n$: the points on $\T_n$ are moving towards the fixed point $\p_n$.}
  \label{fig:TI}
\end{figure}

Case (1): $v, w$ are in the same component of $\mathcal{T}-\{\p\}$. 
Let $b$ be the furthest point to $\p$ on $[\p, v] \cap [\p, w]$.

If $b = v$, i.e., $v\in [\p, w]$, we consider the oriented arcs $[\p_n, w_n] \subseteq \T_n$ (see Figure \ref{fig:TI} left).
Since $s_{v, \infty} = s_{w, \infty}$, the limit of oriented arcs $[s_{v,n}, s_{w,n}] \subseteq [\p_n, w_n]$ contains a simple oriented loop $\gamma$ containing $s_{v, \infty}$.
By Lemma \ref{lem:gfr}, $\T_\infty$ is contained in the bounded Fatou sets of $\hat P$ except at finitely many pre-periodic points.
Since $\gamma$ is simple, $\gamma = \cup_{j=1}^k\gamma_j$ is a finite union of closed smooth arcs whose interior of $\gamma_j$ is contained in $\Omega_v$.
Since $[\p, w]$ is fixed, $\partial \gamma_j$ are fixed points on $\partial \Omega_v$.
Since $P_n$ sends points on $[\p_n, w_n]$ towards $\p_n$, $\gamma_j$ gives a repelling direction for one of the endpoints and an attracting direction for the other.
Therefore $k=1$ and $\gamma$ is a loop in $\Omega_v \cup \{s_{v, \infty}\}$.

The loop $\gamma$ does not bound $w_\infty$ as $\gamma\cap U_{w, \infty} = \emptyset$ and $\partial U_{w, \infty}$ intersects the Julia set.
Similarly, $\gamma$ does not bound $\p_\infty = \lim p_n$.
Since $[\p_n, w_n]$ is an oriented arc, its limit $S$ containing $\gamma$ does not have a transverse intersection.
If $v \notin \mathcal{T}^C$, $\p_\infty \neq s_{v, \infty}$. 
Thus the limits of $[\p_n, s_{v,n}]$ and $[s_{v,n}, s_{w,n}]$ are non-trivial and give two repelling directions separating $U_{v, \infty}$ and $U_{w, \infty}$, so they are not in the same attracting basin, which is a contradiction.
Otherwise, $\alpha_v(T_v\mathcal{T}) \subseteq \mathbb{S}^1$.
Thus, the limit of the arc for $[\p, w]$ in $\D_v$ separates periodic points on $\mathbb{S}^1_v$, so $\gamma$ bounds repelling periodic points, which is a contradiction.

If $b \neq v$, since $\mathcal{T}$ is admissible, $b$ is a critical fixed point (see Figure \ref{fig:TI} middle).
The limit of $[b_n, s_{v,n}] \cup [b_n, s_{w,n}]$ contains a simple loop $\gamma$ containing $b_\infty$.
A similar argument as above shows that $\gamma$ is contained in $\Omega_b$ except at one point.
Since the limit of the arc for $[v,w]$ in $\D_b$ separates periodic points on $\mathbb{S}^1_b$, $\gamma$ bounds repelling periodic points, which is a contradiction.

Case (2): $v, w$ are in different components of $\mathcal{T}-\{\p\}$. 
If $v, w \notin \mathcal{T}^C$, then the proof is the same as above by considering the limit $[\p_n, s_{v,n}] \cup [\p_n, s_{w,n}]$ (see Figure \ref{fig:TI} right).
Otherwise, label the adjacent vertices to $\p$ by $v^1,..., v^k$ counterclockwise.
Let $t^i \in \mathbb{S}^1$ be the unique fixed point of $m_d$ landing at the right side of the oriented arc $[v^i,\p] \cup [\p, v^{i+1}]$.
If $t^i$ does not land at $v^i$ (or $v^{i+1}$), then the first return rescaling limit $F_{v^i}$ (or $F_{v^{i+1}}$) has a repelling direction in the clockwise (or counterclockwise) direction from the fixed point $0 \in \mathbb{S}^1_{v^i}$ (or $\mathbb{S}^1_{v^{i+1}}$), giving a repelling direction that separates $U_{v_i, \infty}$ and $U_{v_{i+1}, \infty}$.
Therefore, $\Omega_{v^i}$ are all different and the statement follows.
\end{proof}

We prove the lamination of $\hat P$ gives the desired equivalence relation.
\begin{prop}\label{prop:la}
Let $a\in H$ be a Julia point. Then there exists a corresponding pre-periodic point $\hat a \in J(\hat P)$ such that an external angle lands at $a$ if and only if it lands at $\hat a$.
\end{prop}
\begin{proof}
We assume $a$ is periodic, as the strictly pre-periodic case can be proved by pullback.
After passing to an iterate, we may assume $a$ is fixed.

If $a\neq \vp$, denote the adjacent vertices by $a_0,..., a_{m-1}$. 
We assume that the vertices are labeled counterclockwise and $a_0$ is the unique vertex that is closer to $\p$ than $a$.
Each $a_i$ is a periodic Fatou point as $H$ is simplicial.
There are exactly $m$ external rays landing at $a$.
Let $A_0,..., A_{m-1}$ be the angles landing at $a$, where $A_i$ corresponds to the access between the Fatou components associated to $a_i$ and $a_{i+1}$.
Let $S \subseteq \mathcal{T}$ be the convex hull of $a_0,..., a_m$ after the admissible splitting (see Figure \ref{fig:Par}).
Let $\hat s_i \in \partial U_{a_i, \infty}$ be the fixed point of $\hat P$ associated to the direction towards $a_0$.
Note $\hat s_i$ is a parabolic fixed point for $\hat P$.

\begin{figure}[ht]
  \centering
  \resizebox{0.8\linewidth}{!}{
    \def\svgwidth{\columnwidth}
    \import{./}{M2.pdf_tex}

    \def\svgwidth{\columnwidth}
    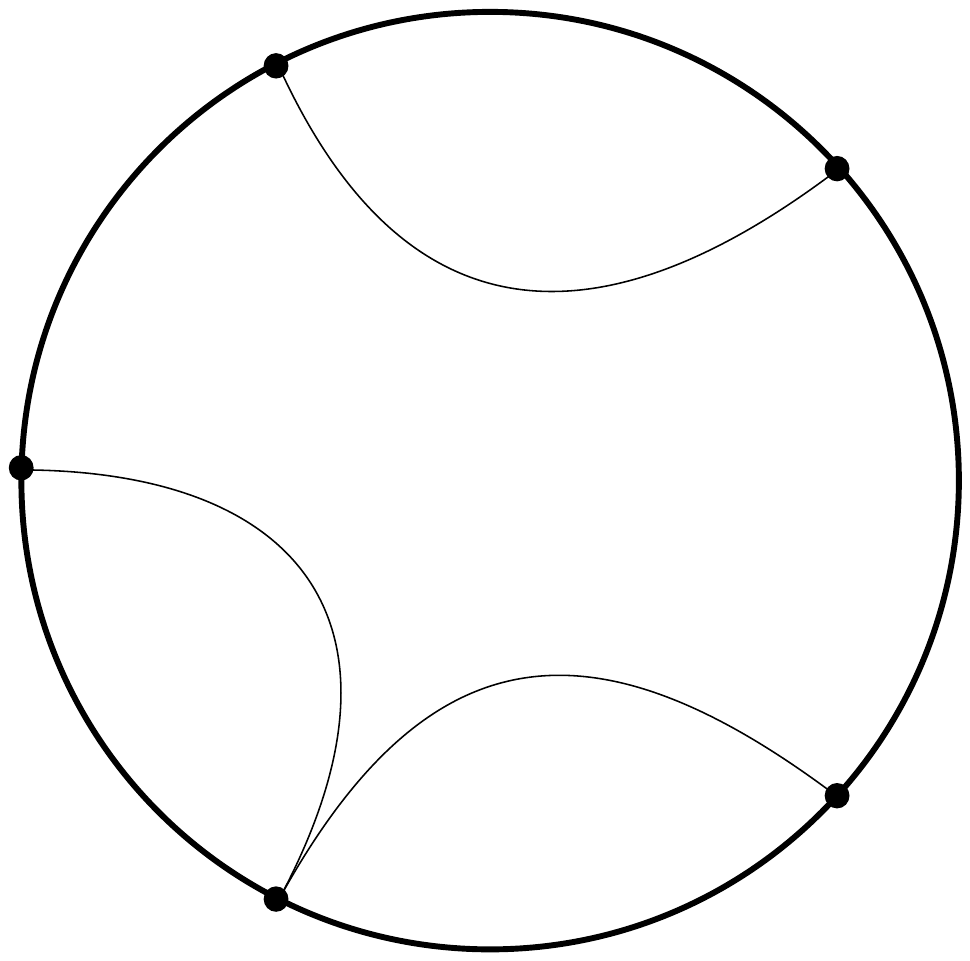

  }
  \caption{The subtree $S$ after the split modification on the left with its dual lamination on the right.}
  \label{fig:Par}
\end{figure}
By Proposition \ref{prop:lam}, we can find periodic points $s_{k}^{i,+} = \lim_{a_i}\eta_n (A^{i,+}_k), s_{k}^{i,-} =\lim_{a_i}\eta_n(A^{i, -}_k)\in \mathbb{S}^1_{a_i}$ that are not holes such that $A^{i,+}_k$ and $A^{i, -}_k$ converge to $A_i$ from below and $A_{i-1}$ from above in counterclockwise orientation (see Figure \ref{fig:Par}).
Let $\hat s_{k}^{i,\pm}$ be the corresponding periodic points of $\hat P$.
We let $R_k$ be the domain bounded by geodesics of $[\hat s_{k}^{i,-}, \hat s_{k}^{i,+}]$ in the Fatou component and external rays of angles $A^{i,\pm}_k$ for $i=0,..., {m-1}$.
Note that $R_{k+1} \subseteq R_k$ and the fixed points $\hat s_i$ $i=1,..., {m-1}$ are all contained in $R_n$.
For sufficiently large $n$, the external angles landing at $\{\hat s_i: i=1,...., {m-1}\}$ are exactly $A_j$, $j=0,..., {m-1}$.
Since $\Omega_{a_i} \neq \Omega_{a_j}$ by Proposition \ref{prop:dfc}, there are at least $m-1$ attracting petals
as the attracting fixed point of $P_n$ does not converge to $\hat s_i$. 
By counting the multiplicity of the parabolic fixed points, $\hat s_i$ must all be the same.
Let $\hat a$ be this parabolic fixed point.
Then the angles landing at it are exactly $A_i$, $i=0,..., {m-1}$.

If $a=\vp$ with valence $k$, then by Proposition \ref{prop:dfc} there are $k$ attracting petals at the parabolic fixed point $\hat a$.
The same nested domain argument shows the angles landing at $\hat a$ are exactly the $k$ angles landing at $a$.
\end{proof}

\begin{prop}\label{prop: sl}
The polynomials $\hat P$ and $P$ have topologically conjugate dynamics on their Julia sets.
\end{prop}
\begin{proof}
By Proposition \ref{prop:dfc}, the critical and post-critical Fatou components of $\hat P$ are in correspondence with those of $P$.
By Proposition \ref{prop:la}, the external angles landing at roots of critical, post-critical Fatou components and critical, post-critical Julia points are the same for $\hat P$ and $P$.
Since these landing angles uniquely determine the dynamics of $\hat P$ and $P$ on their Julia sets, the proposition follows.
\end{proof}

\subsection*{$M$-uni-critical doubly parabolic Blaschke product}
Up to conjugation, there exists a unique uni-critical doubly parabolic Blaschke product for each degree: 
$$
f(z) = \frac{z^d+a}{1+az^d}, \, \text{ for } a=\frac{d-1}{d+1}.
$$
Recall that a degree $d$ proper holomorphic map $f: \D\longrightarrow \D$ is {\em $M$-uni-critical} if the critical points are contained in $B_{\Hyp^2}(0, M)$.
The following compactness result is useful and interesting.
\begin{prop}\label{lem:Muc}
The space of degree $d$ $M$-uni-critical doubly parabolic Blaschke products is bounded in the space of degree $d$ proper maps on $\D$.
\end{prop}
\begin{proof}
Suppose for contradiction that this space is not bounded.
Then there exists a sequence of $M$-uni-critical doubly parabolic Blaschke products $f_n$ that is degenerating viewed as rational maps.
We normalize by rotation so that $1$ is the parabolic fixed point.
Since $f_n$ is degenerating and the critical points are all within $M$ hyperbolic distance away from $0$, after passing to a subsequence, $f_n$ converges algebraically to a constant function $t$ by Lemma \ref{lem:critc}.

Suppose $t=1$. By Lemma \ref{lem:sep}, the preimages $f_n^{-1} (1)$ are uniformly separated.
Thus we can choose an arc $1\in \gamma\subseteq \mathbb{S}^1$ so that $f_n$ is injective on $\gamma$ and its end points $\partial \gamma$ are not holes.
Since $f_n (\partial \gamma) \to 1$ and $f_n(\gamma)$ contains $1$, for sufficiently large $n$, $1$ has at least one attracting direction on $\mathbb{S}^1$, which is a contradiction.

Suppose $t \neq 1$. Choose a disk $U$ with $1\in U$, $t\notin \overline{U}$ and such that $\partial U$ contains no holes.
Then for sufficiently large $n$, $U$ contains no critical values of $f_n$ and $f_n^{-1}(\partial U)\cap \partial U = \emptyset$ as $f_n(\partial U) \to t$.
Thus the component $V$ of $f_n^{-1}(U)$ containing $1$ is contained in $U$, and $f_n$ is univalent on $V$, which is a contradiction to $1$ being a parabolic fixed point.
\end{proof}

\begin{proof}[Proof of Proposition \ref{prop:rsh}]
By our construction, $P_n \to \hat P$ and $\hat P$ corresponds to the Hubbard tree $(H, \p)$ by Proposition \ref{prop: sl}.
The dynamics on $\Omega_v$ for $v\in \mathcal{S}_{\p}$ is conjugate to $\mathcal{F}$ by Lemma \ref{lem:cfm} and Corollary \ref{cor:afc}.
Let $M_1 = M_1(f)$ be the constant in Proposition \ref{prop:rt}.
If $\mathcal{C} \subseteq \mathcal{V}$ is a periodic Fatou cycle of period $q$, there exists $v\in \mathcal{C}$ for which $F_v^q$ is $M_1$-uni-critical.
Since $\hat P^q: U_{v_1, \infty} \longrightarrow U_{v_1, \infty}$ is conjugate to $F^q_{v_1}$ on $\D_{v_1}$ by Lemma \ref{lem:cfm}, $\hat P^q$ on $U_{v_1, \infty}$ is $M_1$-uni-critical.
Since $U_{v_1, \infty} \subseteq \Omega_{v_1}$, and the inclusion is distance non-increasing (with respect to the hyperbolic metric) by Schwarz lemma, $\hat P^q$ on the Fatou component $\Omega_{v_1}$ is also $M_1$-uni-critical.
Since $\hat P^q$ conjugates to a doubly parabolic Blaschke product on $\Omega_{v_1}$, by Proposition \ref{lem:Muc}, $d_{\Omega_{v_1}} (c, \hat P^q(c)) \leq M_2$ for all critical point $c$ and some constant $M_2$ depending only on $M_1$.
Therefore, $\hat P^q$ on $\Omega_{w}$ is $M$-uni-critical for all $w\in \mathcal{C}$, with $M$ depending only on the tree map.

The same argument works for strictly pre-periodic points, and the proposition follows.
\end{proof}

\subsection*{Degenerations on $\partial\PH_d$}
Let $S \subseteq \overline{\PH_d}$ be the space of all geometrically finite polynomials associated to a pointed simplicial Hubbard tree $(H, \p)$.
Proposition \ref{prop:rsh} gives that 
\begin{enumerate}
\item $S$ projects onto the space $\mathcal{B}^{\mathcal{S}_\p}$ of all normalized mapping schemes on $\mathcal{S}_\p$ (see Definition \ref{defn:nms}).
\item Each fiber contains at least $1$ polynomial which is $M$-uni-critical on all other Fatou components.
\end{enumerate}

\begin{figure}[ht]
  \centering
  \resizebox{1\linewidth}{!}{
    \def\svgwidth{\columnwidth}
\begingroup%
  \makeatletter%
  \providecommand\color[2][]{%
    \errmessage{(Inkscape) Color is used for the text in Inkscape, but the package 'color.sty' is not loaded}%
    \renewcommand\color[2][]{}%
  }%
  \providecommand\transparent[1]{%
    \errmessage{(Inkscape) Transparency is used (non-zero) for the text in Inkscape, but the package 'transparent.sty' is not loaded}%
    \renewcommand\transparent[1]{}%
  }%
  \providecommand\rotatebox[2]{#2}%
  \newcommand*\fsize{\dimexpr\f@size pt\relax}%
  \newcommand*\lineheight[1]{\fontsize{\fsize}{#1\fsize}\selectfont}%
  \ifx\svgwidth\undefined%
    \setlength{\unitlength}{371.25bp}%
    \ifx\svgscale\undefined%
      \relax%
    \else%
      \setlength{\unitlength}{\unitlength * \real{\svgscale}}%
    \fi%
  \else%
    \setlength{\unitlength}{\svgwidth}%
  \fi%
  \global\let\svgwidth\undefined%
  \global\let\svgscale\undefined%
  \makeatother%
  \begin{picture}(1,1.04646465)%
    \lineheight{1}%
    \setlength\tabcolsep{0pt}%
    \put(0,0){\includegraphics[width=\unitlength,page=1]{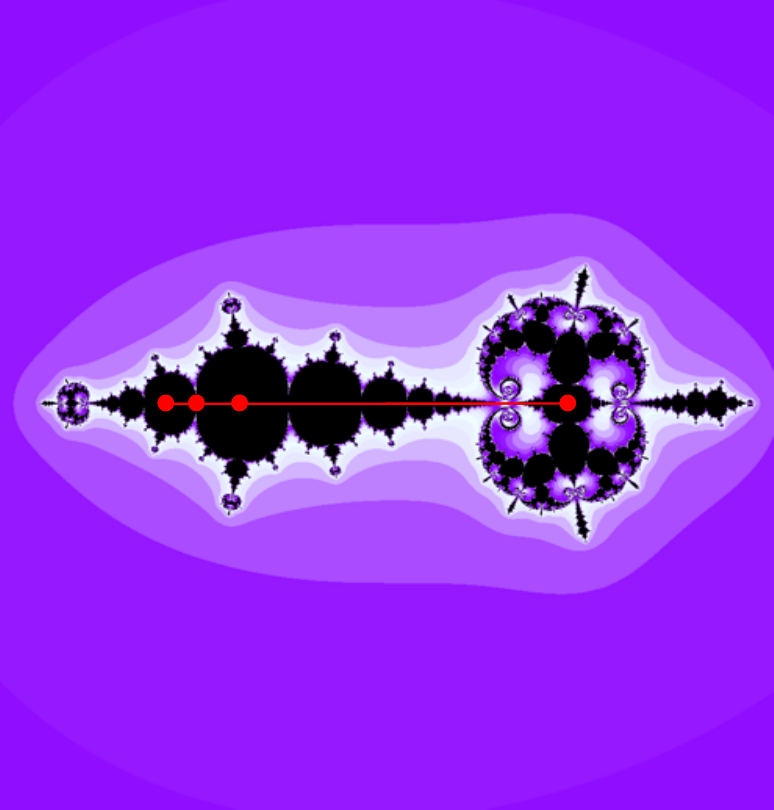}}%
    \put(0.25263705,0.49569761){\color[rgb]{1,0,0}\makebox(0,0)[lt]{\lineheight{1.25}\smash{\begin{tabular}[t]{l}{\LARGE$p_n$}\end{tabular}}}}%
    \put(0.31502261,0.49543151){\color[rgb]{1,0,0}\makebox(0,0)[lt]{\lineheight{1.25}\smash{\begin{tabular}[t]{l}{\LARGE$c_{1,n}$}\end{tabular}}}}%
    \put(0,0){\includegraphics[width=\unitlength,page=2]{D32.pdf}}%
    \put(0.72317635,0.49396332){\color[rgb]{1,0,0}\makebox(0,0)[lt]{\lineheight{1.25}\smash{\begin{tabular}[t]{l}{\LARGE$c_{2,n}$}\end{tabular}}}}%
  \end{picture}%
\endgroup%

    \def\svgwidth{\columnwidth}
\begingroup%
  \makeatletter%
  \providecommand\color[2][]{%
    \errmessage{(Inkscape) Color is used for the text in Inkscape, but the package 'color.sty' is not loaded}%
    \renewcommand\color[2][]{}%
  }%
  \providecommand\transparent[1]{%
    \errmessage{(Inkscape) Transparency is used (non-zero) for the text in Inkscape, but the package 'transparent.sty' is not loaded}%
    \renewcommand\transparent[1]{}%
  }%
  \providecommand\rotatebox[2]{#2}%
  \newcommand*\fsize{\dimexpr\f@size pt\relax}%
  \newcommand*\lineheight[1]{\fontsize{\fsize}{#1\fsize}\selectfont}%
  \ifx\svgwidth\undefined%
    \setlength{\unitlength}{371.25bp}%
    \ifx\svgscale\undefined%
      \relax%
    \else%
      \setlength{\unitlength}{\unitlength * \real{\svgscale}}%
    \fi%
  \else%
    \setlength{\unitlength}{\svgwidth}%
  \fi%
  \global\let\svgwidth\undefined%
  \global\let\svgscale\undefined%
  \makeatother%
  \begin{picture}(1,1.04646465)%
    \lineheight{1}%
    \setlength\tabcolsep{0pt}%
    \put(0,0){\includegraphics[width=\unitlength,page=1]{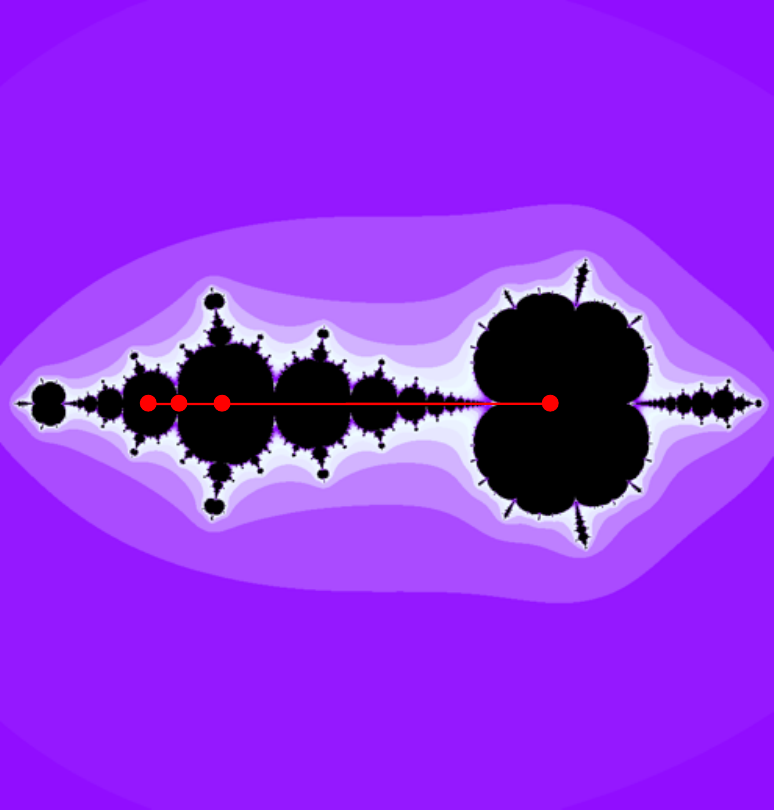}}%
    \put(0.22997812,0.49539942){\color[rgb]{1,0,0}\makebox(0,0)[lt]{\lineheight{1.25}\smash{\begin{tabular}[t]{l}{\LARGE$p$}\end{tabular}}}}%
    \put(0,0){\includegraphics[width=\unitlength,page=2]{D33.pdf}}%
    \put(0.28939439,0.49469738){\color[rgb]{1,0,0}\makebox(0,0)[lt]{\lineheight{1.25}\smash{\begin{tabular}[t]{l}{\LARGE$c_1$}\end{tabular}}}}%
    \put(0.69754817,0.4932292){\color[rgb]{1,0,0}\makebox(0,0)[lt]{\lineheight{1.25}\smash{\begin{tabular}[t]{l}{\LARGE$c_2$}\end{tabular}}}}%
  \end{picture}%
\endgroup%

  }
  \caption{The convergence to a degree $3$ geometrically finite polynomial $\hat P$ (on the right) with iterated-simplicial pointed Hubbard tree that is not simplicial. The critical point $c_{2,n}$ is `hidden' behind $c_{1,n}$. The escaping rates of the two critical points are incompatible, and $c_{2,n}$ is not quasi pre-periodic. $\hat P$ cannot be obtained directly by a quasi post-critically finite degeneration.}
  \label{fig:D3}
\end{figure}

The space $\mathcal{B}^{\mathcal{S}_\p}$ is a product of (normalized) space of Blaschke products (see \cite[\S 4]{Milnor12}).
The degenerations on $\BP^{\mathcal{S}_\p}$ can be defined similarly:
\begin{defn}
Let $\mathcal{F}_n \in \BP^{\mathcal{S}_\p}$. 
We say it is {\em $K$-quasi post-critically finite} if we can label the critical points as $c_{1,n},..., c_{k,n}$ such that for any sequence $c_{i,n}$, there exists quasi pre-period $l_i$ and quasi period $q_i$ with
$$
d_{\D}(\mathcal{F}_n^{l_i}(c_{i,n}), \mathcal{F}_n^{l_i+q_i}(c_{i,n})) \leq K.
$$
\end{defn}
Let $q_{n} \in \D_\p$ be the unique attracting fixed point for $\mathcal{F}_n$.
For $a\in \mathcal{S}_\p$, define $Q_{a,n}$ inductively with $Q_{\p, n} = \{q_n\}$, and $Q_{a,n} = \mathcal{F}_{a,n}^{-1}(Q_{f(a),n})$.

For each vertex $a\in \mathcal{S}_\p$, the quasi-invariant trees $\mathcal{T}_{a,n}$ are constructed similarly as in the case of $\BP_d$.
The dynamics are modeled by a collection of simplicial maps 
$$
\mathcal{F}_a: (\mathcal{T}_a, Q_a) \rightarrow (\mathcal{T}_{f(a)}, Q_{f(a)}),
$$
with rescaling limits $F_v: \D_v \rightarrow \D_{\mathcal{F}(v)}$ where $v$ is a vertex of $\mathcal{T}_a, a\in \mathcal{S}_\p$.
Here $Q_a$ is a finite set corresponding to $Q_{a,n}$.

These simplicial maps with the rescaling limits are combinatorially modeled by {\em angled tree mapping scheme}: a collection of simplicial maps 
$\mathcal{F}_a: (\mathcal{T}_a, Q_a) \rightarrow (\mathcal{T}_{f(a)}, Q_{f(a)})$ with local degree functions and anchored, compatible angle functions.
The angled tree mapping schemes are said to be {\em admissible} if $\mathcal{F}_\p: (\mathcal{T}_\p, q) \rightarrow (\mathcal{T}_\p, q)$ is admissible.

Minimality of the angled tree mapping schemes is defined similarly as angled tree maps.
Note that for a minimal angled tree mapping scheme, $\mathcal{F}_\p: (\mathcal{T}_\p, q) \rightarrow (\mathcal{T}_\p, q)$ may not be minimal as it may contain vertices corresponding to orbits of critical points in $\D_v$ with $v\neq \p$.

Let $\mathcal{S}_q \subseteq \bigcup_{a\in \mathcal{S}_\p} \mathcal{T}_a$ be the backward orbits of $q$ under the angled tree mapping schemes $\mathcal{F}$.

Since the induced dynamics $f: \mathcal{S}_\p \rightarrow \mathcal{S}_\p$ is particularly easy (it has a unique fixed point at $\p$ and all the other points are in the backward orbits of $\p$), by pulling back the degeneration on $\D_\p$, Proposition \ref{prop:rt} can be easily generalized to degenerations in $\BP^{\mathcal{S}_\p}$:
\begin{prop}\label{prop:rt2}
Let $\mathcal{F}: \bigcup_{a\in\mathcal{S}_\p} (\mathcal{T}_a, Q_a) \rightarrow \bigcup_{a\in\mathcal{S}_\p} (\mathcal{T}_a, Q_a)$ be a minimal admissible angled tree mapping scheme for $\mathcal{S}_\p$.
In the case when $\delta(q) \geq 2$,
we let $\mathcal{G}$ be a normalized mapping scheme on $\mathcal{S}_q$.
There exists a $K$-quasi post-critically finite sequence $\mathcal{F}_n \in \BP^{\mathcal{S}_\p}$ realizing $\mathcal{F}$ such that
\begin{enumerate}
\item the rescaling limits on $\mathcal{S}_q$ are conjugate to $\mathcal{G}$;
\item there exists a constant $M$ depending only on the angled tree mapping scheme $\mathcal{F}$ (and thus independent of $\mathcal{G}$) so that 
\begin{enumerate}
\item[(a)] for any periodic cycle $\mathcal{C}$ other than $q$, there exists a periodic point $v \in \mathcal{C}$ so that the first return rescaling limit 
$$
F^k_v: \D_v\longrightarrow \D_v
$$ 
is $M$-uni-critical, where $k$ is the period of $v$;
\item[(b)] for any strictly pre-periodic vertex $w \notin \mathcal{S}_q$, the rescaling limit $F_w: \D_w \longrightarrow \D_{f(w)}$ is $M$-uni-critical and the critical values are within $M$ hyperbolic distance from $0 \in \D_{f(w)}$.
\end{enumerate}
\end{enumerate}
\end{prop}

Let $(H', \p')$ be a pointed simplicial tuning of $(H, \p)$. 
Then there exists an associated angled tree mapping scheme
$$
\mathcal{F}_a: (\mathcal{T}_a, Q_a) \rightarrow (\mathcal{T}_{f(a)}, Q_{f(a)}).
$$
By performing admissible splitting on $\mathcal{T}_\p$ and pullback the modifications accordingly to $\mathcal{T}_a, a\neq \p$, we can associate a quasi post-critically finite degeneration $\mathcal{F}_n\in\mathcal{B}^{\mathcal{S}_\p}$ by Proposition \ref{prop:rt2}.
By Proposition \ref{prop:rsh}, we can find a sequence $Q_n \in S \subseteq \overline{\PH_d}$ whose dynamics on Fatou components of $\mathcal{S}_\p$ are conjugated to $\mathcal{F}_n\in\mathcal{B}^{\mathcal{S}_\p}$ while remaining $M$-uni-critical on all other Fatou components.

Since $\overline{\PH_d}$ is compact, after passing to a subsequence, $Q_n$ converges to $Q\in \overline{\PH_d}$.
The same argument as in Proposition \ref{prop:gfs}, which uses Lemma \ref{lem:key}, shows that $Q$ is geometrically finite.
By Proposition \ref{lem:Muc}, the Fatou component $\Omega_{v,n}$ of $Q_n$ for $v\in \mathcal{V}-\mathcal{S}_\p$ converges to a corresponding Fatou component $\Omega_v$.
The same argument as Proposition \ref{prop:dfc} shows the corresponding Fatou component of $\Omega_v$ are different for different Fatou vertices of $\mathcal{T}_{a}, a\in \mathcal{S}_\p$.
Thus, a similar proof for Proposition \ref{prop:la} and Proposition \ref{prop: sl} gives that the associated pointed Hubbard tree for $Q$ is $(H',\p')$. 
Moreover, by Proposition \ref{prop:rt2}, the space $S' \subseteq \overline{\PH_d}$ of all geometrically finite polynomials associated to $(H', \p')$ again satisfies the two properties listed at the beginning of this subsection.
Therefore, by induction, we have
\begin{prop}\label{prop:conv}
Let $(H, \p)$ be an iterated-simplicial Hubbard tree. Then there exists a geometrically finite polynomial $\hat P \in \overline{\PH_d}$ associated to it.
\end{prop}

\begin{proof}[Proof of Theorem \ref{thm: eq} ]
The theorem follows from Proposition \ref{prop:NC} and Proposition \ref{prop:conv}.
\end{proof}

\section{Self-bumps on $\partial \PH_d$}\label{sec:sb}
In this section, we shall prove Theorem \ref{thm:sb}.
We first explain the phenomenon with an example in degree $4$, which can then be easily generalized to any higher degree.

Consider the geometrically finite polynomial
$$
\hat P(z) = z^4-\frac{3}{8}z^2+\frac{9}{8}z-\frac{3}{256}.
$$
It has a super attracting fixed point at $-\frac{3}{4}$, and a double parabolic fixed point at $1/4$.
The associated pointed Hubbard tree is a tripod
$$
H = [v, a_0] \cup [v, a_1] \cup [v, a_2]
$$
where $\vp = a_0$ and the vertices are ordered counterclockwise (see Figure \ref{fig:P}). 

The dynamics fix all three edges.
There are two distinct admissible split modifications, resulting in two different admissible angled tree maps:
$$
{^1}\mathcal{T}:= [a_0, v] \cup [v, a_1] \cup [a_1, a_2]
$$
and 
$$
{^2}\mathcal{T}:= [a_0, v] \cup [v, a_2] \cup [a_2, a_1].
$$
Let ${^1}P_n, {^2}P_n\in \PH_d$ be the two sequences of polynomials associated to ${^1}\mathcal{T}$ and ${^2}\mathcal{T}$, where every polynomial in the sequences is assumed to have a super attracting fixed point associated to $a_0$ (see Figure \ref{fig:AccessTop} and \ref{fig:AccessBottom}).
There are three bounded critical Fatou components for $\hat P$,
and each Fatou component is fixed and contains one critical point.
One is super attracting, and the other two are parabolic.
Thus, the dynamics on the Fatou components are rigid (see \cite[\S 6]{McM88}).
Hence both ${^1}P_n$ and ${^2}P_n$ converge to $\hat P$.

Denote the repelling fixed points by ${^1}x_{1,n}, {^1}x_{2,n}, {^1}x_{3,n}$ and ${^2}x_{1,n}, {^2}x_{2,n}, {^2}x_{3,n}$ respectively, then ${^i}x_{j,n}$ converges to the parabolic fixed point $1/4$ for $\hat P$.
We label them so that the three fixed points are ordered counterclockwise and ${^i}x_{2,n}$ corresponds to the fixed point that is accessible from the positive real axis in the limit.
To prove that $\hat P$ gives a self-bump on $\partial \PH_d$, we study the multipliers of the three fixed points.

\subsection*{Residue computation}
Let $f$ be a holomorphic function defined in a neighborhood of $z_0 \in \C$, and $z_0$ is an isolated fixed point of $f$.
The {\em residue} of $f$ at the fixed point $z_0$ is
$$
res(f, z_0) = Res_{z=z_0} \frac{dz}{f(z) -z}
$$
where the right-hand side is the residue of the $1$-form at $z_0$.
The residue is invariant under conformal change of coordinate, thus it can also be defined for a fixed point at $\infty$.

If the multiplier at $z_0$ is $\lambda \neq 1$, then the residue is
$$
res(f, z_0) = \frac{1}{\lambda -1}.
$$
Thus, $z_0$ is a repelling fixed point if and only if
$$
\mathfrak{Re} (res(f, z_0)) > -\frac{1}{2}.
$$
Let $C$ be an oriented closed curve that bounds a domain $D$ with no fixed point on $C$.
Then the Residue Theorem gives that
$$
\sum_{z_i \text{ fixed point in } D} res(f, z_i) = \frac{1}{2\pi i}\int_{C} \frac{dz}{f(z) - z}.
$$
For a global rational map $f: \hat \C \longrightarrow \hat \C$, we have that
$$
\sum_{z_i \text{ fixed point}} res(f, z_i) = -1.
$$

Back to our setting, an explicit computation shows $res(\hat P, 1) = 1$.
\begin{lem}\label{lem:rf}
Let $U\subseteq \MP_4$ be any sufficiently small neighborhood of $\hat P$.
Let $f \in U \cap \PH_4$, and $z$ be any repelling fixed point of $f$.
Then $\mathfrak{Im} (f'(z)) \neq 0$.
\end{lem}
\begin{proof}
Let $z_1, z_2, z_3$ be the three repelling fixed points of $f$.
If we choose $U$ sufficiently small, we have that the multipliers $\lambda_i$ are close to $1$ and
$$
\sum_{i=1}^3 res(f, z_i) = \sum_{i=1}^3 \frac{1}{\lambda_i-1}
$$
is close to $res(\hat P, 1) = 1$.
If $\mathfrak{Im} (\lambda_1) = 0$, then $\lambda_1 >1$, so
$$ 
\mathfrak{Re} (\frac{1}{\lambda_2-1}+ \frac{1}{\lambda_3-1})
$$
is very negative. 
Thus at least one of the multiplier is attracting, which is a contradiction.
\end{proof}

To prove that the intersection of any sufficiently small neighborhood $U$ of $\hat P$ with $\PH_4$ is disconnected, it suffices to show the {\em signatures}, i.e., signs of the imaginary part of the multipliers of ${^i}x_{j,n}$ are different for the two sequences.

\subsection*{Signatures of a simple parabolic point}
Let $f$ be a holomorphic function in a neighborhood of $0 \in \C$. Assume that $0$ is an isolated simple parabolic fixed point of $f$.
After a conformal change of coordinate, we may assume that
$$
f(z) = z+z^2+O(z^3)
$$
where $a > 0$.
Note that near $0$, the positive real axis is a repelling direction, while the negative real axis is an attracting direction.
Assume that we perform a small perturbation so that the parabolic fixed point splits into two repelling fixed ones.
After conjugating with $z\mapsto z+c$ if necessary, we may assume that these two fixed points are $\pm \epsilon$, symmetric with respect to $0$. 
Since we assume that the two fixed points are all repelling, $\epsilon$ is not a real number.
By this normalization, an easy computation shows that the signatures of $\pm \epsilon$ equal to the signs of the imaginary part of $\pm \epsilon$.

\subsection*{Splitting the double parabolic point}
To compute the signature, we degenerate in two steps, which allows us to consider perturbations of only simple parabolic points.
Consider the degree $4$ polynomial $Q\in \partial \PH_4$ with super attracting fixed point of local degree $3$ with a parabolic fixed point on the boundary of the immediate super attracting basin.
We can degenerate the dynamics on the attracting Fatou component while staying on $\partial \PH_4$ (see Proposition \ref{prop:rsh}).

We choose the marking for the degree $3$ Blaschke product so that $0 \in \R/\Z \cong \mathbb{S}^1$ corresponds to the parabolic fixed point of $Q$.
We call this repelling fixed point the {\em marked fixed point}.
We construct two geometrically finite sequences of Blaschke products ${^1}f_n$ and ${^2}f_n$, with quasi-invariant trees 
$$
{^i}\mathcal{T} = [\p, a],
$$
where $\p$ and $a$ are fixed points of local degree $2$. 
The rescaling limits at $a$ are different:
For ${^1}\mathcal{T}$, the marked fixed point has angle $0^+$ at the vertex $a$, while it has angle $0^-$ at the vertex $a$ for ${^2}\mathcal{T}$.
We also assume that both sequences have a super attracting fixed point.
Then the corresponding sequence of polynomials ${^1}Q_n$ and ${^2}Q_n$ both converge to $\hat P(z)$ (see Figure \ref{fig:AccessTop} and \ref{fig:AccessBottom}).

For sufficiently large $n$, we can perturb ${^1}Q_n$ and ${^2}Q_n$ slightly to get two polynomials ${^1}P_n$ and ${^2}P_n$ in $\PH_4$, with ${^i}P_n \to \hat P$.
As before, we denote the three repelling fixed points of ${^i}P_n$ by ${^i}x_{1,n}, {^i}x_{2,n}, {^i}x_{3,n}$, where ${^i}x_{2,n}$ corresponds to the fixed point that is accessible from the positive real axis in the limit (see Figure \ref{fig:SB}).

Note that ${^i}P_n$ is constructed from ${^i}Q_n$ by splitting the simple parabolic fixed point into two repelling fixed points.
Using the orientation of the dynamics near the simple parabolic fixed point of ${^i}Q_n$, we can compute the signature at ${^1}x_{2,n}$ is $+$ while the signature at ${^2}x_{2,n}$ is $-$.
Thus, by Lemma \ref{lem:rf}, Theorem \ref{thm:sb} holds for degree $4$.

If we replace the super attracting fixed point by a super attracting fixed point of degree $d-2$, Theorem \ref{thm:sb} holds for any degree $d \geq 4$.


\end{document}